\documentclass[11pt]{amsart}
\usepackage[headings]{fullpage}
\usepackage{amssymb,epic,eepic,epsfig,amsbsy,amsmath,amscd,color}
\usepackage[all]{xy}
\usepackage{graphicx,subcaption}
\usepackage{texdraw}
\usepackage{url}
\usepackage{bbm}
\usepackage{mathrsfs}
\usepackage{standalone}
\usepackage{tikz}
\usetikzlibrary{cd,positioning,knots,patterns,hobby}
\usepackage{accents}
\usepackage[normalem]{ulem}
\usepackage{comment}
\usepackage{calc}
\usepackage{soul,xcolor}

\setstcolor{red}

\def\biblbl{\bibitem}

\theoremstyle{plain}
\newtheorem{theorem}{Theorem}[section]

\newtheorem{lemma}[theorem]{Lemma}
\newtheorem{corollary}[theorem]{Corollary}
\newtheorem{proposition}[theorem]{Proposition}

\newtheorem{definition}{Definition}

\theoremstyle{definition}

\newtheorem{remark}[theorem]{Remark}
\newtheorem{example}[theorem]{Example}

\newcommand\FIGc[3]{\begin{figure}[htpb]
    \includegraphics[height=#3]{#1.eps}
    \caption{#2}
    \label{fig:#1}
    \end{figure}}

\newcommand\incl[2]{{\includegraphics[height=#1]{#2.eps}}}

\newcommand{\stateS}{\mathscr{S}}

\newcommand{\no}[1]{}

\def\BC{\mathbb C}
\def\BN{\mathbb N}
\def\BZ{\mathbb Z}

\def\CO{\mathcal O}

\def\CR{\mathcal R}
\def\CS{\mathcal S}

\def\la{\langle}
\def\ra{\rangle}

\def\al{\alpha}
\def\ve{\varepsilon}

\def\be { \begin{equation} }
\def\ee { \end{equation} }

\def\bS{\overline \CS}

\def\cS{\mathscr S}
\def\ot{\otimes}

\def\cE{\mathcal E}

\def\bn{\mathbf n}

\def\cP{\mathcal P}

\def\tP{\widehat \Phi}

\def\fS{\mathfrak S}
\def\bfS{\overline{\fS}}

\def\cX{\mathcal X}

\def\embed{\hookrightarrow}

\def\cV{\mathcal V}

\def\cX{\mathcal X}
\def\cX{\mathcal X}

\def\SM{(\Sigma,\cP)}
\def\pS{\partial \Sigma}

\def\id{\mathrm{id}}

\def\pM{\partial M}
\def\cN{\mathcal N}
\def\MN{(M,\cN)}

\def\tPhi{\widetilde{\Phi}}

\def\fr{\operatorname{fr}}

\def\ord{\mathrm{ord}}
\def\Sx{\cS_\omega}

\def\Se{\cS_\eta}
\def\tSe{\widehat{\cS}_\eta}

\def\MNp{(M',\cN')}

\def\cSp{\cS^+}

\def\Po{\Phi_\omega}
\def\TD{\Theta_D}

\def\ord{\mathrm{ord}}

\begin{document}

\title[Chebyshev-Frobenius homomorphism]{The Chebyshev-Frobenius homomorphism for stated skein modules of 3-manifolds}

\author[Wade Bloomquist]{Wade Bloomquist}
\address{School of Mathematics, 686 Cherry Street,
 Georgia Tech, Atlanta, GA 30332, USA}
\email{wbloomquist3@gatech.edu}

\author[Thang  T. Q. L\^e]{Thang  T. Q. L\^e}
\address{School of Mathematics, 686 Cherry Street,
 Georgia Tech, Atlanta, GA 30332, USA}
\email{letu@math.gatech.edu}

%\date{\today}

\thanks{
2010 {\em Mathematics Classification:} Primary 57N10. Secondary 57M25.}

\begin{abstract}
We study the stated skein modules of marked $3-$manifolds.  We generalize the splitting homomorphism for stated skein algebras of surfaces to a splitting homomorphism  for stated skein modules of $3-$manifolds. We show that there exists a Chebyshev-Frobenius homomorphism for the stated skein modules of 3-manifolds which extends the Chebyshev homomorphism of the skein algebras of unmarked surfaces originally constructed by Bonahon and Wong.  Additionally, we show that the Chebyshev-Frobenius map commutes with the splitting homomorphism. This is then used to show that in the case of the stated skein algebra of a surface, the Chebyshev-Frobenius map is the unique extension of the dual Frobenius map (in the sense of Lusztig) of $\mathcal{O}_{q^2}(SL(2))$ through the triangular decomposition afforded by an ideal triangulation of the surface.  In particular, this gives a skein theoretic construction of the Hopf dual of Lusztig's Frobenius homomorphism.  A second conceptual framework is given, which shows that the Chebyshev-Frobenius homomorphism for the stated skein algebra of a surface is the unique restriction of the Frobenius homomorphism of quantum tori through the quantum trace map.
\end{abstract}

\maketitle

\def\cSp{\cS^+}

\section{Introduction} 

\subsection{Skein modules/algebras}
For an oriented 3-dimensional manifold, $M$, the Kauffman bracket skein module, $\cS(M)$, is the quotient of the free module spanned by isotopy classes of framed links in $M$ subject to the relations (A) and (B)  of  Figure \ref{fig:stated-relations}. Here the ground ring is a commutative domain with a distinguished invertible element $q^{1/2}$.
An explanation of how to interpret these figures can be found in Subsection \ref{SkeinModule}.

\begin{figure}[htpb]
\begin{subfigure}[b]{0.45\linewidth}
\centering
\includegraphics[scale=.35]{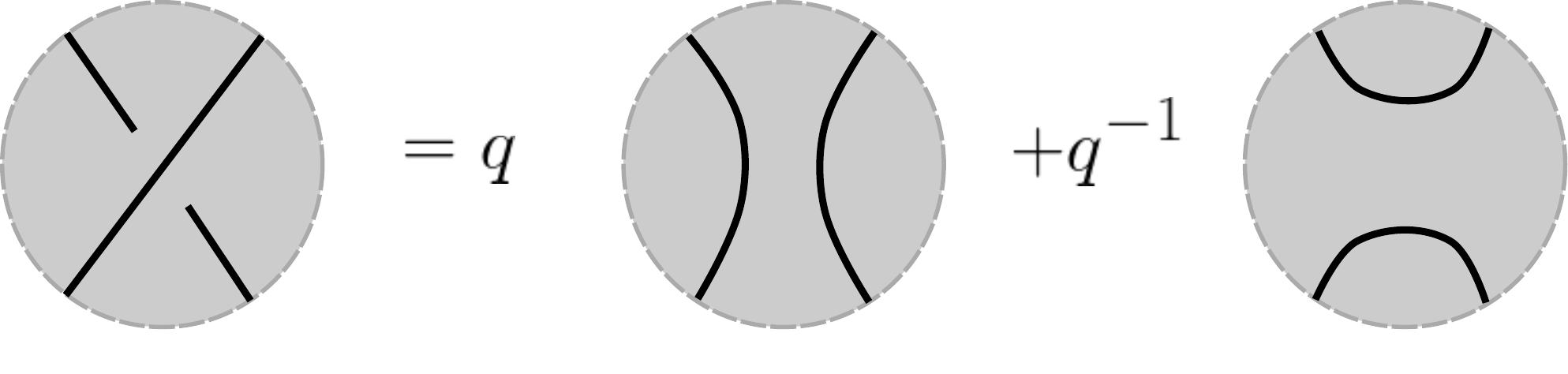}
\subcaption{Skein relation}
\end{subfigure}
\hfill
\begin{subfigure}[b]{0.45\linewidth}
\centering
\includegraphics[scale=.35]{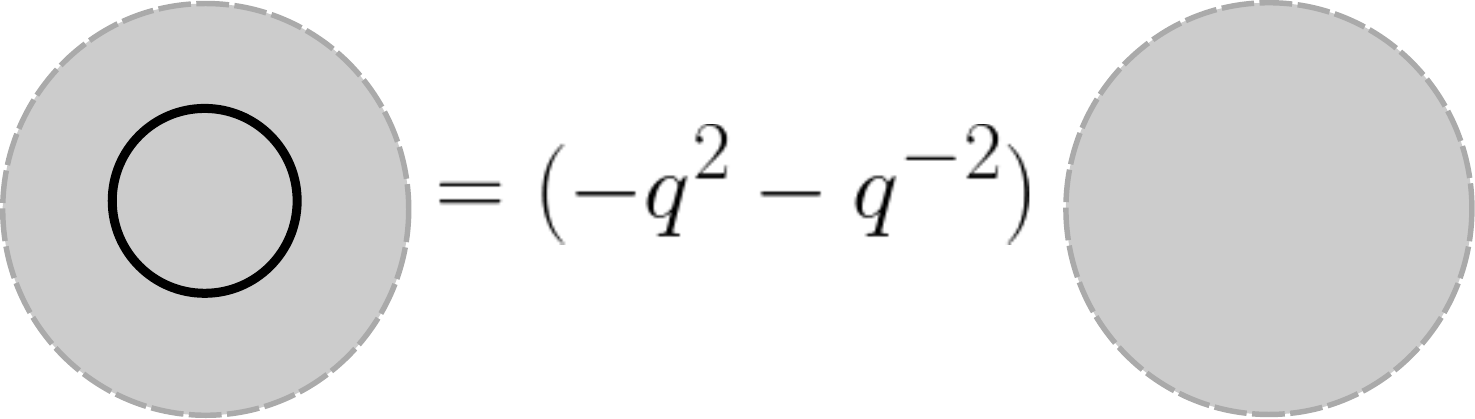}
\subcaption{Trivial knot relation}
\end{subfigure}
\begin{subfigure}[b]{0.45\linewidth}
\centering
\includegraphics[scale=.35]{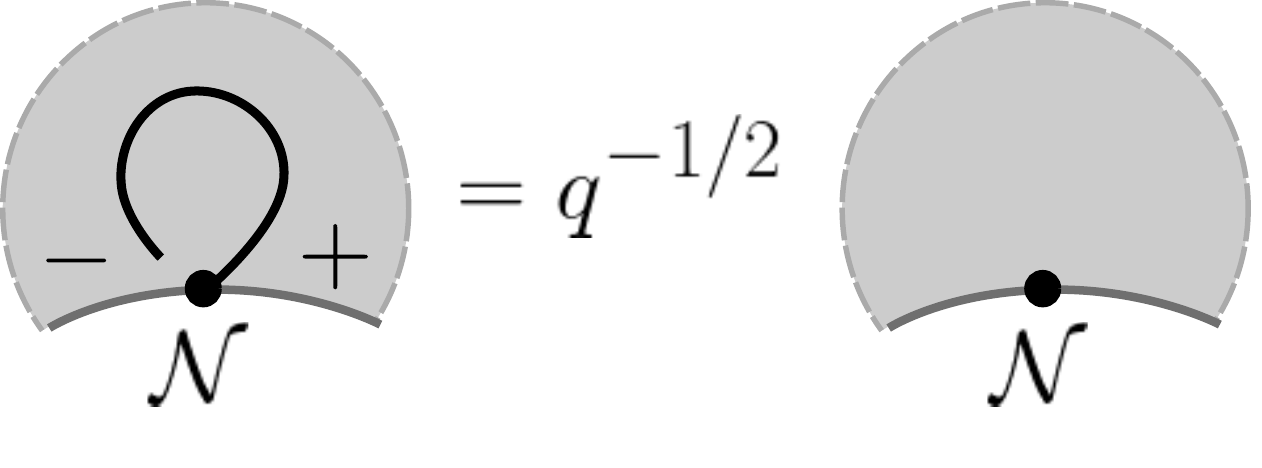}
\subcaption{Trivial arc relation 1}
\end{subfigure}
\hfill
\begin{subfigure}[b]{0.45\linewidth}
\centering
\includegraphics[scale=.35]{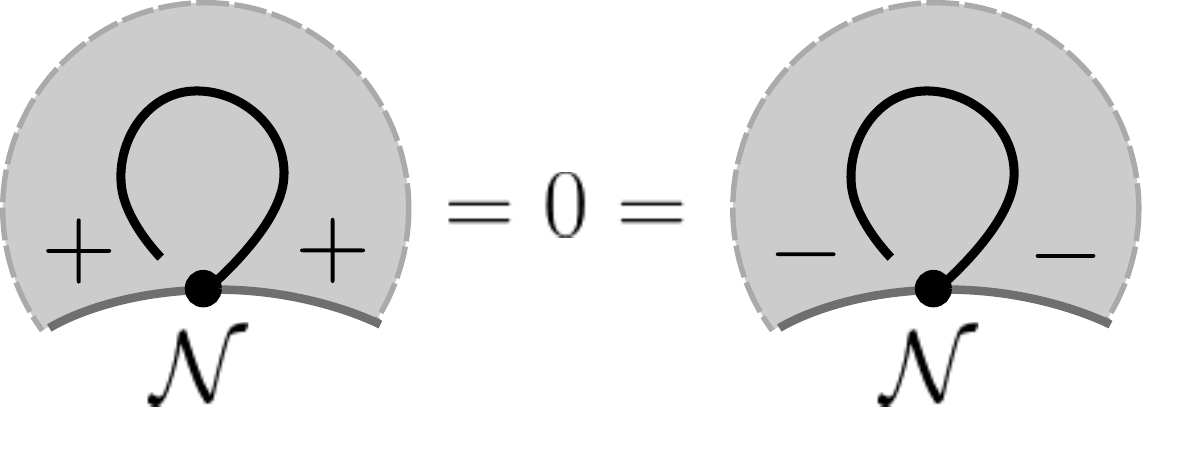}
\subcaption{Trivial arc relation 2}
\end{subfigure}
\begin{subfigure}[b]{\linewidth}
\centering
\includegraphics[scale=.35]{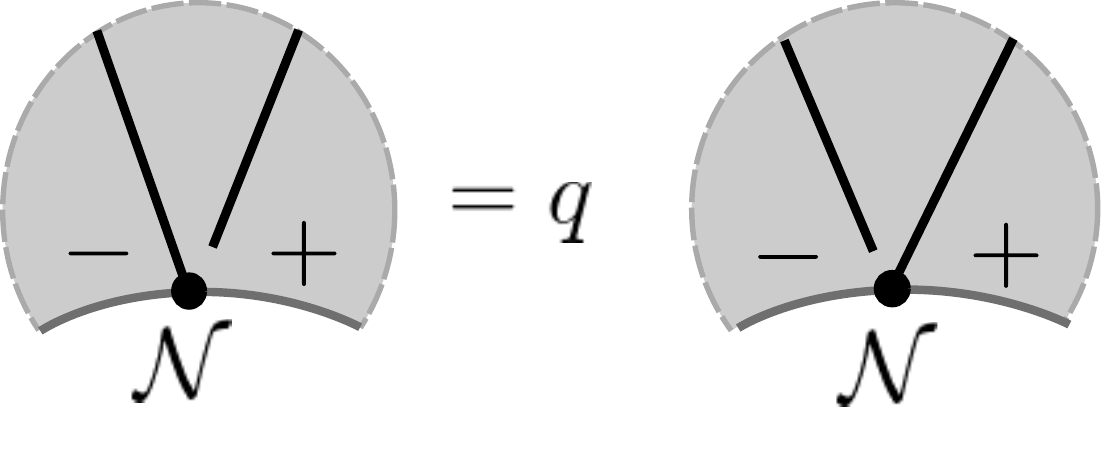}
\subcaption{State exchange relation}
\end{subfigure}
\caption{The defining relations in $\stateS(M,\mathcal{N})$}
\label{fig:stated-relations}
\end{figure}

Skein modules were introduced independently by Przytycki in \cite{Przy} and Turaev in \cite{Turaev,Turaev2}, building on work of Kauffman \cite{Ka} in his study of the Jones polynomial.  These skein modules have found many connections throughout low dimensional topology and quantum algebra.
%, but we refer to \cite{LY} for a more thorough history and survey of the field.  
In the specific case of the thickening of a surface, Turaev introduced an algebra structure for skein modules via vertical stacking while exploring quantizations of the Atiyah-Bott-Weil-Petersson-Goldman symplectic structure of the character variety \cite{Turaev}.

Bonahon and Wong \cite{BW0} showed that for a surface with at least one puncture the skein algebra  can be embedded, through a map called the quantum trace, into a quantum torus which is known as the quantum Teichm\"uller space of Chekhov and Fock \cite{CF}.  When $q=1$, the quantum trace map calculates the trace of a closed curve in the shear coordinates of Teichm\"uller space.
The existence of the quantum trace map based on an ideal triangulation of the surface indicated that skein algebras should decompose into elementary blocks.  This idea was formalized by the second author, \cite{Le4}, through the introduction of stated skein algebras of punctured bordered surfaces.
Facilitating the decomposition of stated skein algebras is the existence of a splitting homomorphism, which maps a stated skein algebra of a surface to the stated skein algebra of the result of cutting the surface along an ideal arc.  This splitting homomorphism not only provides insight into the quantum trace map, but also provides a wealth of structure for stated skein algebras in terms of quantum groups utilizing an isomorphism of the stated skein algebra of the bigon with $\mathcal{O}_{q^2}(SL(2))$, see \cite{CL}.  The stated skein algebra considered here involves tangles which may end on the boundary of the surface, and is a quotient of a different, larger, stated skein algebra originally considered by Bonahon and Wong.
%Note that the stated skein modules of marked $3-$manifolds are a generalization of both the stated skein algebras of surfaces as well as the skein modules of marked $3-$manifolds originally defined for surfaces by Muller \cite{Muller}, and extended to $3-$manifolds by the second author in \cite{Le3}.  Where the skein module of $3-$manifolds can then be recovered by restricting all states to $+$, called the positive submodule. 

In this paper we study the stated skein modules of marked $3-$manifolds.  Where a marked $3-$manifold is a $3-$manifold, $M$, along with a collection of markings, $\cN$, which are oriented arcs on $\partial M$.  A {\em stated $\cN-$tangle $\al$} is a $1-$dimensional framed properly embedded submanifold of $M$, whose boundary, $\partial \al$, is contained in the markings $\cN$, and additionally at each boundary point the framing is in the positive direction of $\cN$ and a state from $\{\pm\}$ is assigned. Then the stated skein module of $(M,\cN)$, denoted $\mathscr{S}(M,\cN)$ is the quotient of the free module spanned by isotopy classes of stated $\cN-$tangles subject to the relations seen in Figure \ref{fig:stated-relations}. See  Subsection \ref{SkeinModule} for a detailed explanation. When $\MN$ is the thickening of a marked surface, the stated skein module has a natural algebra structure, and was introduced in \cite{Le4}.

\subsection{The splitting homomorphism for $3-$manifolds} Paralleling the case of surfaces, we show that there is a splitting homomorphism of stated skein modules for $3-$manifolds which are split along a properly embedded disk.
%We introduce the splitting homomorphism of stated skein modules along a properly embedded disk.  

\def\OSL{\mathcal{O}_{q^2}(SL(2))}

Suppose $D$ is a properly embedded disk in $M$ that is disjoint from the closure of the markings $\cN$. Additionally, let $a\subset D$ be an oriented open  interval. By splitting $M$ along $D$ we have a $3-$manifold, $M'$, whose boundary contains two copies of $D$, and hence two copies of $a$, denoted by $a_1$ and $a_2$. Let $\cN'= \cN \cup a_1 \cup a_2$ and consider $(M', \cN')$ as a marked $3-$manifold.
 If an $\cN-$tangle, $\al$,  meets $D$ transversely in $a$, such that the framing at every point in $\al \cap a$ is a positive tangent vector of $a$, then given a state at each point of $\al\cap a$ we call the resulting splitted $\cN'$-arc a lift of $\al$.    
\begin{theorem}[Splitting Homomorphism, see Theorem \ref{SplittingHomo}]
If $M'$ is a result of cutting $M$ along a properly embedded disc then there is a well defined 
$\CR$-linear map 
\[\Theta:\mathscr{S}(M,\cN)\rightarrow\mathscr{S}(M',\cN')\]
that sends any $\cN$-tangle $\alpha$ to the sum of all lifts of $\alpha$. 
\end{theorem}

Analogously to the case of surfaces, the splitting homomorphism provides stated skein modules with additional structure. For example, as seen in Subsection \ref{sec.bigon}, every connected component of the marking provides $\cS\MN$ a comodule structure over the Hopf algebra $\OSL$ with the coaction given by splitting.

\subsection{The Chebyshev homomorphism for unmarked surfaces}
\label{CHomoIntro}
The Chebyshev homormoprhism, introduced by Bonahon and Wong in \cite{BW1}, has played a crucial role in understanding the algebraic structure of the Kauffman bracket skein algebras of unmarked surfaces. Suppose 
$\Sigma$ is an unmarked surface, and $\omega$ is a root of unity with $N=\ord(\omega^8)$ and $\eta=\omega^{N^2}$. The Chebyshev homomorphism is a an algebra homomorphism 
\[Ch_\omega:\mathscr{S}_\eta(\Sigma)\rightarrow \mathscr{S}_\omega(\Sigma)\]
which sends any knot $\al \subset \Sigma$ to $T_N(\al)$. Here $T_N$ are the 
 Chebyshev polynomials of type one defined recursively by
\begin{align*}
T_0(z)=2, \ \ \ T_1(z)=z, \ \ \ T_n(z) = zT_{n-1}(z)-T_{n-2}(z), \ \ \forall n \geq 2. 
\end{align*}

Two features supporting the ability of this map to be used in deriving algebraic information are
\begin{itemize}
    \item $\mathscr{S}_\eta(\Sigma)$ is interpreted classically in terms of regular functions on the character variety of $\Sigma$
    \item The image of $Ch_\omega$ is central, or almost central
\end{itemize}
both of which were instrumental in understanding the representations of the skein algebra $\cS_\omega(\Sigma)$, see \cite{BW2,BW3,FKL1,FKL2}. 

The original proof of the existence of the Chebyshev homomorphism given in \cite{BW1} is based on the quantum trace map. A more elementary, skein theoretic proof is given in \cite{Le2}.

\subsection{The Chebyshev-Frobenius homomorphism for marked 3-manifolds}
\label{CFHomoIntro}
The main result of the paper, which is stated briefly here and given in full detail in Theorem \ref{main}, shows that the Chebyshev homomorphism can be extended to the case of marked 3-manifold. Suppose 
$\MN$ is a marked 3-manifold, and $\omega$ is a root of unity with $N=\ord(\omega^8)$ and $\eta=\omega^{N^2}$.

For a one-component $\cN$-tangle $\al$ let $\al^{(k)}$ be the disjoint union of $k$ parallel copies of $\al$ (taken in the direction of the framing). If $p= \sum_k x^k$ is a polynomial, define 
$$p^{\fr} (\al) = \sum c_k \al^{(k)},$$
which is considered as an element in $\cS_\omega\MN$.  
\begin{theorem}[Main Result, see Theorem \ref{main}]
Let $(M,\cN)$ be a marked $3-$manifold and let $\omega$ be a root of unity with $N=\ord(\omega^8)$ and $\eta=\omega^{N^2}$, then we have a unique well-defined $\mathbb{C}$-linear map 
\[\Phi_\omega: \mathscr{S}_\eta(M,\cN)\rightarrow \mathscr{S}_\omega(M,\cN),\]
such that for an $\cN$-tangle $T = a_1 \cup \cdots \cup a_k \cup \al_1 \cup \cdots \cup \al_l$ where the $a_i$ are $\cN$-arcs and the $\al_i$ are $\cN$-knots,
\begin{align}
\Phi_\omega (T) & = a_1^{(N)} \cup \cdots \cup a_k^{(N)} \cup (T_N)^{\fr}(\al_1) \cup \cdots \cup (T_N)^{\fr}(\al_l) \quad \text{in }\Sx\MN
%&: = \sum_{0\le j_1, \dots, j_l\le N} c_{j_1} \dots c_{j_l}  a_1^{(N)} \cup \dots \cup  a_k^{(N)} \cup \, \al_1^{(j_1)} \cup \cdots \cup \al_l^{(j_l)} \quad \text{in }\Sx\MN.
\end{align}
%where $T_N$ is the $N^{\text{th}}$ Chebyshev polynomial, and $a^{(N)}$ the $N^{\text{th}}$ framed power of $a$. 

Additionally, $\Phi_\omega$ is compatible with the splitting homomorphism, in the sense that they may be applied in either order.
\end{theorem}
This generalizes the existence of the Chebyshev homomorphism of Bonahon and Wong to stated skein modules of $3-$manifolds by dealing with links in the same manner, and sending $\cN-$arcs to their framed power.  This generalizes the work of the second author and Paprocki in the case of the positive submodule of $\cS\MN$ in \cite{LP}. However, the introduction of mixed states gives rise to considerably different defining relations making the work independent. For comments on known partial results and related results, \cite{KQ,GJS,BR}, see Remark \ref{PrevKnown}.  In particular in the case of the stated skein algebras of surfaces, the work of Korinman and Quesney \cite{KQ}, provides an independent proof of our Corollary \ref{r.surfacePhi} when the order of $\omega$ is odd.
%A collection of known partial results on Chebyshev-Frobenius homomorphisms, \red{in particular work of Koriman and Quesney \cite{KQ} in the case stated skein algebras at certain roots of unity}, is provided in Remark \ref{PrevKnown}. 

As an application, in  Theorem \ref{r.transparent} we  find transparent elements of stated skein modules, and  central elements of stated skein algebras of surfaces.

\subsection{Quantum tori and the quantum trace map}
In Section \ref{QTorusSect} a conceptual framework based on quantum tori is described for the Chebyshev-Frobenius homomorphism in the case of the stated skein algebras of surfaces.  

Given an anti-symmetric matrix $P$ with integer entries and a non-zero complex number $\omega$, the quantum torus $\mathbb{T}(P;\omega)$ is defined as $\BC\langle \{x_i^{\pm 1}\}: x_ix_j=\omega^{P_{ij}}x_jx_i\rangle$.  Given any integer $N$,  there is an algebra map,  known as the Frobenius homomorphism:  
\be F_N:\mathbb{T}(P;\omega^{N^2})\rightarrow \mathbb{T}(P; \omega),\quad x_i \to x_i^N.
\label{eq.Fro}
\ee 

Suppose the marked surface $\SM$ has at least one marked point and has a {\em quasitriangulation $\cE$}, see Section \ref{QTorusSect}. A recent result of the second author and Yu, see \cite[Theorem 5.1]{LY}, shows that for any non-zero complex number $\omega$ there exists an algebra embedding 
\be
\mathscr{S}_\omega(\Sigma,\cP)\embed \mathbb{T}(P_\cE;\omega), 
\label{eq.emb1}
\ee
where $P_\cE$ is an anti-symmetric matrix depending on the quasitriangulation.
This map, which extends the work of Muller \cite{Muller}, is similar to the quantum trace map of Bonahon and Wong, but different, as the classical specialization expresses the trace of a closed curve in terms of Penner's lambda length coordinates \cite{Penner}.

Combining with the Frobenius homomorphism $F_N$ of \eqref{eq.Fro}, we have the following diagram
\be
\begin{tikzcd}
\mathscr{S}_{\omega^{N^2}}(\Sigma,\cP)\arrow[r,hook]
\arrow[d,dashed,"?"]  
&  \mathbb{T}(P_\cE;\omega^{N^2})\arrow[d,"F_N"] \\
\mathscr{S}_{\omega}(\Sigma,\cP)\arrow[r,hook] & \mathbb{T}(P_\cE;\omega)
\end{tikzcd}
\ee
This leads to the natural question of for which values of $\omega$ the map $F_N$ restrics to a map  \[\mathscr{S}_{\omega^{N^2}}(\Sigma,\cP) \to 
\mathscr{S}_{\omega}(\Sigma,\cP),\] such that the restriction does not depend on the underlying quasitriangulation.
\begin{theorem}[See Theorem \ref{FrobeniusRestrict}] Suppose $\SM$ has at least two quasitriangulations. Then $F_N$ restricts to a map $\mathscr{S}_{\omega^{N^2}}(\fS) \to 
\mathscr{S}_{\omega}(\fS)$ and the restriction does not depend on the quasitriangulation if and only if $\omega$ is a root of unity and $N= \ord(\omega^8)$. Moreover, in this case, the restriction map is $\Phi_\omega$.
\end{theorem}  This is proven by leveraging the corresponding statement proven in the posititive submodule case in \cite{LP}.  The theorem gives a perspective in which the Chebyshev-Frobenius map is unique, and at the same time highlights the role of roots of unity and the order $\ord(\omega^8)$.

\def\fB{\mathfrak B}
\subsection{Quantum group  and Lusztig's Frobenius homomorphism}
In Section \ref{QGroupSect} a conceptual framework based on quantum groups is described for the Chebyshev-Frobenius homomorphism in the case of the stated skein algebras of surfaces. 

We first turn to the stated skein algebra of the bigon and its isomorphism to $\mathcal{O}_{q^2}(SL(2))$, which is the Hopf dual of Lusztig's quantum group of divided powers $U^L_{q^2}(\mathfrak{sl}_2)$.  
In this case we see that the Chebyshev-Frobenius homormorphism $\Phi_\omega$ is the dual of Lusztig's Frobenius homomorphism from the study of quantum groups specialized at roots of unity.
%, $f:U^L_{\omega^4}(\mathfrak{sl}_2)\rightarrow U^L_{\eta^4}(\mathfrak{sl}_2) $, in the sense of Lusztig.  Then we have a skein theoretic descirption of the dual Frobenius map and additionally, seen in full detail in Theorem \ref{uniqueness},
\begin{theorem}[See Theorem \ref{uniqueness}] Suppose $\fB$ is the bigon.  
The Chebyshev-Frobenius homomorphism $\Phi_\omega: \cS_\eta(\fB) \to \cS_\omega(\fB)$ is the Hopf dual of Lusztig's
Frobenius homomorphism.
\end{theorem}
In a sense, this provides a skein theoretic description of Luzstig's Frobenius homomorphism. A discussion of the dual Frobenius map for $\mathcal{O}_{q^2}(SL(n))$ can be found in Section $7$ of \cite{PW}.

In addition, the Hopf algebra $\mathcal{O}_{q^2}(SL(2))$ is co-braided, i.e. has a co-$R$-matrix \cite{Kassel} which allows for its representation theory to be described with a ribbon category, and we show  in Proposition \ref{triangGlue} that our Chebyshev-Frobenius map respects the co-braided structure. In \cite{CL} it was proved that the stated skein algebra of a triangle is the braided tensor product of two copies of the stated skein algebra of the bigon. This shows for an ideal triangle, the Chebyshev-Frobenius map can be built from Lusztig's Frobenius homomorphism via this braided tensor product. For an arbitrary triangulable marked surface, via an ideal triangulation and the splitting homomorphism, we see that our Chebyshev-Frobenius map can be interpreted as a natutral extension of the dual of Lusztig's Frobenius homomorphism from the bigon to the whole surface.

\subsection{Structure of the paper}
We provide a brief summary of the paper:
\begin{enumerate}\setcounter{enumi}{1}
    \item This section defines the stated skein module of marked $3-$manifolds, describes the relationship to stated skein algebras of punctured bordered surfaces, and introduces the key concept of functoriality for stated skein modules.
    \item This section proves the existence of a splitting homomorphism for stated skein modules along properly embedded disks, generalizing the splitting homomorphism of stated skein algebras along ideal arcs.   This endows stated skein modules with the structure of an $\mathcal{O}_{q^2}(SL(2))$ comodule.
    \item This section proves the existence of the Chebyshev-Frobenius homomorphism for stated skein modules, and shows the compatibility of this map with the splitting homomorphism defined in the previous section.   Additionally, an observation is made regarding the center of stated skein algebras utilizing transparency and the image of the Chebyshev-Frobenius homomorphism. 
    \item This section provides a conceptual framework for the Chebyshev-Frobenius homomorphism in terms of the Frobenius homomorphism of quantum tori.  
    \item This section provides a conceptual framework for the Chebyshev-Frobenius homomorphism in terms of the Hopf dual of the Frobenius homomorphism (in the sense of Lusztig) for quantum groups.
\end{enumerate}

\section{Skein modules/algebras} 
\subsection{The ground ring and quantum integers}
\label{qnotation}

Throughout this paper the ground ring $\mathcal{R}$ will be a commutative domain with a distinguished invertible element $q^{\frac{1}{2}}$. When $q$ is a root of 1, the {\em order of $q$}, denoted by $\ord(q)$, is the smallest positive integer $n$ such that $q^n=1$.

We follow the standard definition of quantum integers, quantum factorials, and quantum binomials, but fix notation as follows:
\begin{align*}
    & [n]_q:=\frac{q^{n}-q^{-n}}{q-q^{-1}} \\
    & [n]_q!:=[n]_q\cdot [n-1]_q!\\
    & {n\brack k}_q=\prod_{i=0}^{k-1}\frac{1-q^{n-i}}{1-q^{i+1}}
\end{align*}

%We will often leave off the subscript when the choice of $q$ is clear from context.

The following facts concerning quantum integers are well-known, see e.g. \cite{Kassel}.
\begin{lemma}
\label{qfacts}
(a) One has

\be  {n \brack k}_q=q^{k}{n-1 \brack k}_q+{n-1 \brack k-1}_q
\ee

(b) If  $q$ is a root of unity of order $N$ then
\be 
{N\brack k}_{q}=\begin{cases} 0  &\text{if } \ 0<k<N
 \\
 1 &\text{if }\  k=0, N.
\end{cases}
\ee

(c) If $yx=qxy$ and $q$ is a root of unity of order $N$, then 
\[(x+y)^N=x^N+y^N.\]
\end{lemma}

\subsection{Skein modules of marked 3-manifolds}
\label{SkeinModule}

By a {\em marked 3-manifold} we mean a pair $(M,\cN)$, where $M$ is an oriented  3-manifold with (possibly empty) boundary $\partial M$ and $\cN$ is a $1-$dimensional oriented submanifold of $\partial M$ such that each connected component of $\cN$ is diffeomorphic to the  interval $(-1,1)$, the closure of each component is diffeomorphic to $[-1,1]$, and the closures of the connected components of $\cN$ are disjoint.
 %An alternative, but equivalent, approach would be to consider each connected component of $\cN$ as an oriented point on $\pM$, but we will not make use of this convention.

By an {\em $\cN$-tangle in $M$} we mean a compact 1-dimensional non-oriented submanifold $\al$ of $M$, equipped with a framing,  such that $\partial \al=\al\cap \cN$.
Where a {\em framing} is a continuous assignment of a vector to each point of $\al$, which is not tangent at that point, and such that at each boundary point  the framing  is a positive tangent vector of $\cN$. Two $\cN$-tangles are {\em $\cN$-isotopic} if they are isotopic in the class of $\cN$-tangles. The empty set is considered as a $\cN$-tangle isotopic only to itself. A 1-component $\cN$-tangle $\al$ is  diffeomorphic to either the circle $S^1$ or  the closed interval $[0,1]$; we call $\al$ an {\em $\cN$-knot} in the first case and an {\em $\cN$-arc} in the second case.

A {\em stated $\cN$-tangle} $\al$ is a $\cN$-tangle equipped with a  map $s : \partial \al \to \{\pm \}$, called the state of $\al$.   
% As all $\cN-$tangles we will discuss will be stated, we will often omit the adjective stated unless an additional emphasis is desired.

The {\em stated skein module} $\cS\MN$ 
 is the $\mathcal{R}$-module freely spanned by all stated $\cN$-tangles modulo $\cN$-isotopy and the {\em defining} relations described in Figure \ref{fig:stated-relations}. These relations are understood as follows.
In each of the figures the dashed outline indicates a ball, $B \subset M$, containing part of a stated tangle diagram. A relation involving, say 3 terms, indicates that whenever we have 3 stated tangles which are identical outside the ball (corresponding to the dashed part of each term), then we have a relation on stated tangles as described. In each term the $\cN$-tangle is described by strands, with standard tangle diagram convention. The framing vector is always perpendicular to the page and pointed at the reader. The ball $B$ will be called the {\em support of the relation}.
 
Relations $(A)$ and $(B)$ do not involve the boundary of $M$ and are standard in skein module theory. The first and the second diagrams on the right hand side of $(A)$ are called respectively the positive and the negative resolutions of the diagram on the left hand side.

 Each of  Relations $(C), (D),$ and $(E)$, involve $\partial M$ and the marking set $\cN$. Here $B\cap \cN$ is a small interval perperdicular to the page at the point marked by the bullet on the boundary, and its positive direction is pointing at the reader. When two strands come to the marked point, the lower (in height order) one is depicted by the broken line. The states of the endpoints of the tangle are indicated there.

\begin{remark}
\label{alternative}
Introducing a crossing through an isotopy in Relation $(E)$, it is easy to see that
in the presence of Relations $(A)-(D)$,  the state exchange relation $(E)$ is equivalent to the relation  in Figure \ref{fig:StateExchangeAlternate}.
\begin{figure}[htpb]
  \centering
    \includegraphics[scale=.35]{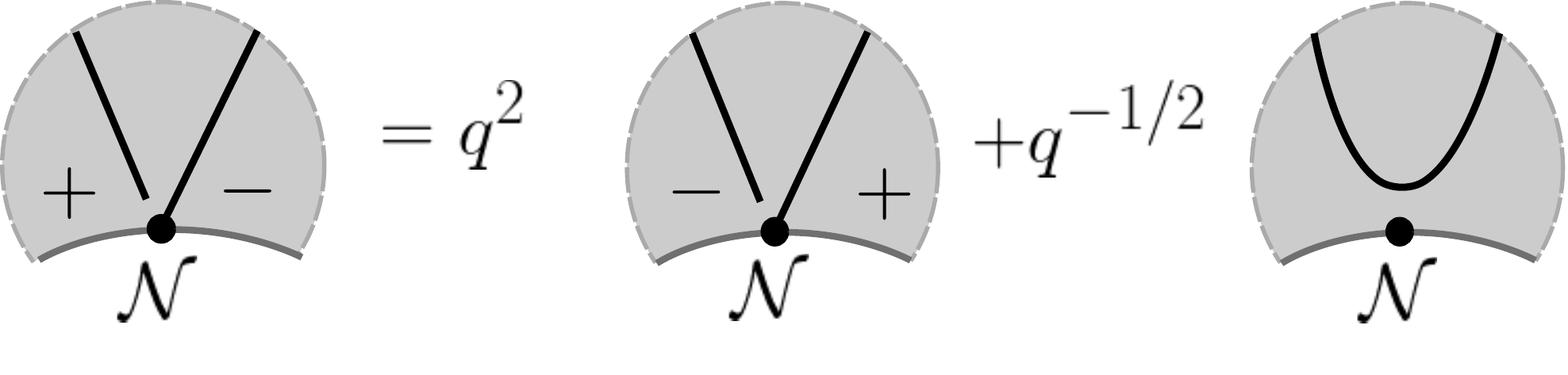}
    \caption{A relation equivalent to Relation $(E)$ using Relations $(A)$ and $(C)$} 
    \label{fig:StateExchangeAlternate}
\end{figure}
\end{remark}

\def\cC{\mathcal C}
\def\cR{\mathcal R}

We use the notation $[\al]$ to denote the element of $\cS\MN$ determined by the stated $\cN$-tangle $\al$. However, we will often abuse notation and use $\al$ and $[\al]$ interchangably when there is no chance for confusion.

\subsection{Functoriality/the category of marked $3-$manifolds}
\label{functoriality} Let $\cC$ be the category whose objects are marked 3-manifolds, and a morphism from 
a  marked 3-manifold $\MN$ to a marked 3-manifold $(M', \cN')$ is an isotopy class of embeddings $f:\MN \embed (M', \cN')$. Here {\em an embedding} $f:\MN \embed (M', \cN')$  is an orientation preserving   proper embedding $f: M \embed M'$ such that  $f$ restricts to an orientation preserving embedding on $\cN$. 
Such an embedding  induces  an $\mathcal{R}$-module homomorphism $f_*: \cS\MN \to \cS\MNp$ by $f_*([\al])=[f(\al)]$ for any stated $\cN$-tangle $\al$.

We have a natural isomorphism of $\cR$-modules
\be 
\mathscr{S}(M_1\sqcup M_2, \cN_1 \sqcup \cN_2)\cong \mathscr{S}(M_1,\cN_1)\otimes_\mathcal{R}\mathscr{S}(M_2,\cN_2).
\ee
%which identifies $[T_1]\otimes [T_2]$ and $[T_1\bigsqcup T_2]$ for any $\cN_1-$tangle $T_1$ and $\cN_2-$tangle $T_2$.

In other words, the assignment $\MN\to \cS\MN$ and a morphism $f$ to $f_*$ is a functor from $\cC$ to the category of $\cR$-module, and if we define a monoidal structure on $\cC$ by 
 $\MN\ot (M', \cN')=(M\sqcup M, \cN \sqcup \cN')$, then this is a monoidal functor.

\begin{example} \label{exa.1}
One particular case of functoriality will be used frequently. Let $\MN$ be a marked 3-manifold and $X$ be a closed subset of $\pM$ disjoint from $\cN$. Define $M'= M \setminus X$. We will say that the marked $3-$manifold $(M',\cN)$ is {\em pseudo-isomorphic} to $(M,\cN)$. The natural embedding $\iota: (M', \cN) \embed \MN$ induces an isomorphism of $\cR$-modules $\iota_*: \cS(M',\cN) \cong \cS\MN$.
\end{example}

%  $U\subset\partial M$ is a open neighborhood of $\cN$.  Then define the marked $3-$manifold $(int(M)\cup N,\cN)$.  This leads to an isomorphism \[\mathscr{S}(M,\cN)\cong \mathscr{S}(int(M)\cup N,\cN).\]

%We will, in particular, make free use of this isomorphism to talk about closed disks embedded in $3-$manifolds.  We will say a closed disk is embedded in $M$ if the closed disk is embedded in $cl(M)$.

\def\tal{\tilde \al}

\subsection{Marked surfaces} 
A {\em finite type surface} is a surface homeomorphic to a surface obtained by removing a finite number of points from a  compact oriented 2-manifold with (possibly empty) boundary.
By a {\em marked surface} we mean a pair $\SM$, where 
$\Sigma$ is a finite type surface and 
 $\cP$, called the set of marked points, is a finite subset of the boundary $\pS$.  

\def\tSM{(\widetilde \Sigma, \widetilde {\cP}) }
For a marked surface $\SM$, the {\em thickening of $\SM$} is the marked 3-manifold $\tSM$ where  
$\widetilde{\Sigma} = \Sigma \times (-1,1)$ and $\widetilde {\cP} =\mathcal{P}\times(-1,1)$. 
 
 Define $\cS\SM= \cS\tSM$.  Given two stated $(\mathcal{P}\times(-1,1))$-tangles $\alpha, \al'$ define the product $\al \al'$ by stacking $\al$ above $\al'$. This gives $\cS\SM$ an $\mathcal{R}$-algebra structure.
 
% A vector at a point $(x,t)\in \Sigma\times (-1,1)$ is {\em upward vertical} if it is the positive direction of the component $(-1,1)$.  

 %The algebra $\cS\SM$ is closely related to quantum Teichm\"uller space and the quantum cluster algebra of the surface \cite{???}. 

\def\pfS{\partial \fS}
\def\tpfS{\widetilde{ \pfS}}
\subsection{Equivalence with the definition for punctured bordered  surfaces.}
\label{sec.equiv}

We will show that the marked surface definition of a stated skein algebra is equivalent to the original definition given in \cite{Le4} utilizing punctured bordered surfaces.

A {\em  punctured bordered surface}, $\fS$,
 is  a finite type surface whose boundary is the disjoint union of open intervals.

To a marked surface $\SM$, with $\mathcal{P}=\{p_i\}_{i=1}^k$, one can associate a punctured bordered surface $\fS$ as follows. 
 Let $v_i \subset \pS$ be a small open neighborhood of $p_i$ in $\partial \Sigma$. 
Now let $\mathfrak{S}:=(\Sigma\setminus \partial \Sigma)\cup  (\sqcup _{i=1}^k v_i)$. Then  $\partial \mathfrak{S}=\sqcup _{i=1}^k v_i$ and each $v_i$ is called a {\em boundary edge} of $\fS$. 

%Conversely, to every punctured bordered surface $\fS$ one can construct a marked surface $(\Sigma', \cP')$ where $\Sigma'=\fS$ and $\cP'\subset \pfS$  is the collection of points containing exactly one point in each boundary component of $\pfS$. If $\fS$ is the associated punctured bordered surface of the marked surface $\SM$, then by Example \ref{exa.1}, we have a natural isomorphism $\cS\SM = \cS(\Sigma', \cP')$.

\def\tfS{\widetilde {\fS}}
\def\tSigma{\widetilde {\Sigma}}
\def\ptfS{\partial \tfS}

The manifold $\tfS:= \fS \times (-1,1)$ is a submanifold of $\tSigma= \Sigma \times (-1,1)$.
% and  has boundary $\tpfS= \pfS \times (-1,1)$. 
We call $c\times (-1,1)$, where $c$ is a boundary edge of $\fS$, is a {\em boundary wall} of $\tfS$. The boundary $\tpfS$ of $\tfS$ is the disjoint union of all the boundary walls.

By a {\em $\tpfS$-tangle} $\al$ in $\tfS$ we mean a framed 1-dimensional compact non-oriented submanifold properly embedded in $\tfS$ with {\em vertical framing} at each endpoint and distinct heights for endpoints in each boundary wall. Here a framing vector at a point in $\tfS= \fS\times (-1,1)$ is {\em vertical} if it is tangent to and has the positive direction of the component $(-1,1)$.

Two $\tpfS$-tangles are {\em $\tpfS$-isotopic} if they are isotopic in the class of $\tpfS$-tangles. Note that the endpoints of $\al$ in one boundary wall are linearly ordered by heights since they have distinct heights, and $\tpfS$-isotopy does not change the height order.

One nice feature of $\tpfS$-tangles is that their tangle diagrams on $\fS$ can have distinct boundary points, unlike the case of $\cN$-tangle diagrams.

A {\em $\pfS$-tangle diagram} is a tangle diagram $\al$ on $\fS$ whose endpoints are distinct points in $\pfS$, and on each boundary edge $c$ the set $\partial \al \cap c$ is equipped with a linear ordered. Any $\pfS$-tangle diagram $\al$ defines a $\tpfS$-tangle, unique up to {\em $\tpfS$-isotopies}, if one  equips $\al$ with a vertical framing, and the height order on each boundary wall is the given order. 
Every $\tpfS$-tangle can be represented by a $\pfS$-tangle diagram.

Return to the marked surface $\SM$. Every $\tilde \cP$-tangle in $\Sigma\times (-1,1)$ is automatically a {\em $\tpfS$-tangle}, and 
this gives a bijection of $\tilde \cP$-isotopy classes of $\tilde \cP$-tangles and $\tpfS$-isotopy classes of $\tpfS$-tangles.

\begin{figure}[htpb]
\no{
\begin{subfigure}[b]{0.45\linewidth}
\centering
\includegraphics[scale=.35]{SkeinRelation.pdf}
\subcaption{Skein relation}
\end{subfigure}
\hfill
\begin{subfigure}[b]{0.45\linewidth}
\centering
\includegraphics[scale=.35]{TrivialKnotRelation.pdf}
\subcaption{Trivial knot relation}
\end{subfigure}
}
\begin{subfigure}[b]{0.45\linewidth}
\setcounter{subfigure}{2}
\centering
\includegraphics[scale=.35]{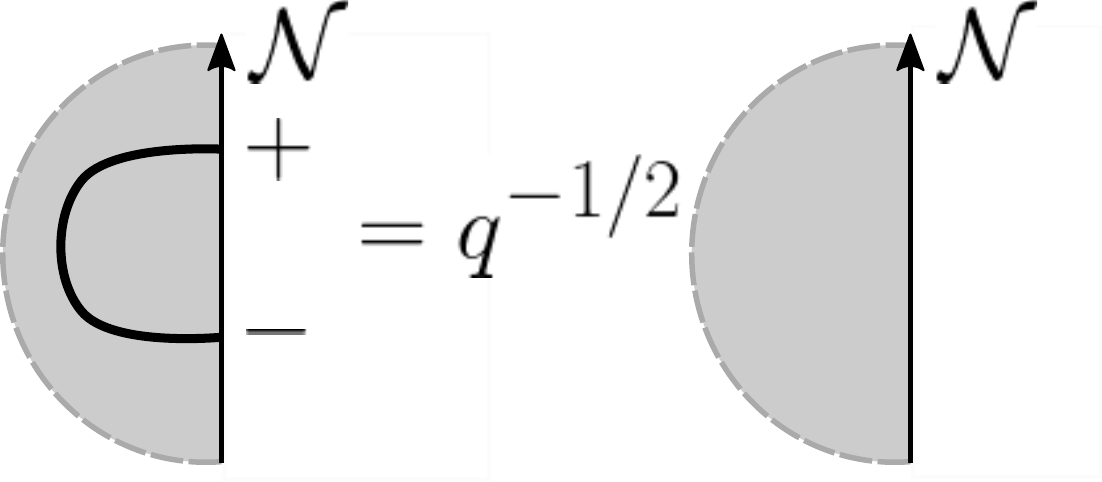}
\subcaption{Trivial arc relation 1}
\end{subfigure}
\hfill
\begin{subfigure}[b]{0.45\linewidth}
\centering
\includegraphics[scale=.35]{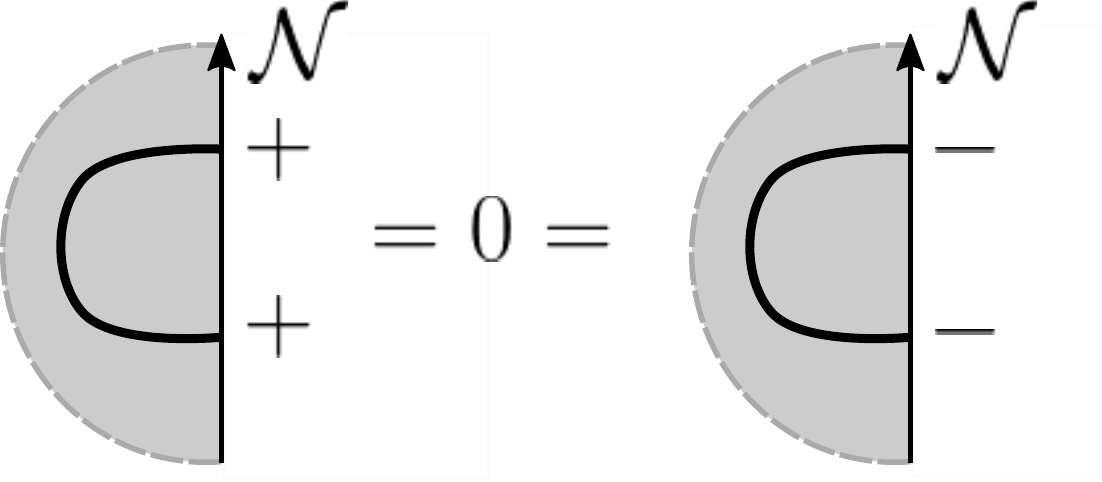}
\subcaption{Trivial arc relation 2}
\end{subfigure}
\begin{subfigure}[b]{\linewidth}
\centering
\includegraphics[scale=.35]{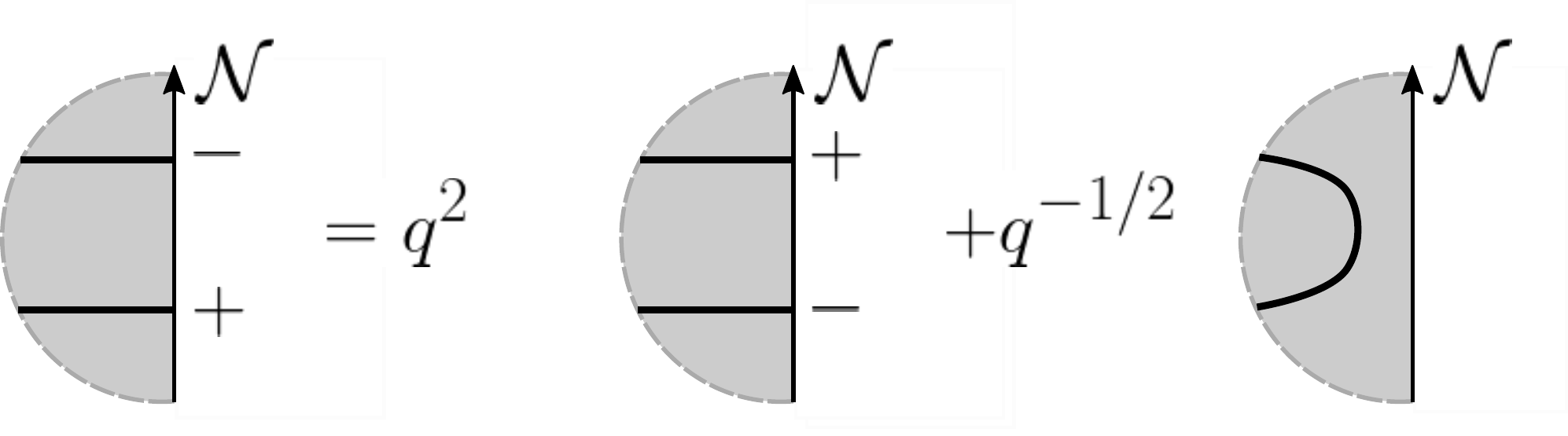}
\subcaption{State exchange relation}
\end{subfigure}
\caption{A translation of the defining relations in $\mathscr{S}(M,\cN)$ to the language of punctured bordered surfaces, using the alternative version of the state exchange relation}
\label{fig:punctstated-relations}
\end{figure}
 
 Hence the skein algebra $\cS\SM$ is canonically isomorphic to the skein algebra $\cS(\fS)$ defined as the $\cR$-module freely generated by $\tpfS$-isotopy classes of stated $\tpfS$-tangles modulo the original relations (A) and (B), and the new relations (C-E) of Figure \ref{fig:punctstated-relations}. We have that relations $(C)$ and $(D)$ of Figure \ref{fig:punctstated-relations} are direct translations of relations $(C)$ and $(D)$ in Figure \ref{fig:stated-relations} and relation $(E)$ of Figure \ref{fig:punctstated-relations} is a translation of the alternative version of relation $(E)$ from \ref{fig:stated-relations} seen in Remark \ref{alternative}.  These figures are understood as follows.  In each of the figures the dashed outline represents a disk, $D\subset \fS$, containing part of a stated $\tpfS-$tangle diagram. The arrow on the boundary edge is used to indicate the height order of the two endpoints presented there, meaning that going along the direction of the arrow increases the height order. These two endpoints are consecutive in the height order, and the order of other endpoints not presented in the figure is not given by the direction of the arrow.
  These diagrams are taken with the blackboard framing meaning they correspond to a piece of $\fS\times(-1,1)$ where the $(-1,1)$ component is perpendicular to the page.  Each relation corresponds to a relation  on stated $\tpfS-$tangles in $\fS\times (-1,1)$ which are identical outside of $D\times (-1,1)$, and which satisfy the expression given inside of $D\times (-1,1)$.

 As mentioned above, $\ptfS$-tangles can be conveniently depicted by $\pfS$-tangle diagrams. This leads us to often formulate statements for punctured bordered surfaces, but these statements can be easily converted to statements for stated skein algebras of marked surfaces.

 \subsection{A basis for the skein algebra of surfaces}
 Let $\fS$ be a punctured bordered surface.
 
  A $\pfS$-tangle diagram is {\em simple} if it has no double points corresponding to crossings and no components which are trivial as defined below.  A $\pfS-$arc is a $1-$component simple stated $\tpfS-$tangle diagram having non-empty boundary.  A $\pfS-$knot is a $1-$component simple stated $\tpfS-$tangle diagram having empty boundary. Note that a $\pfS-$knot is determined by a simple closed curve on $\fS$.   A $\pfS-$knot is said to be trivial if it bounds a disk in $\fS$.  A $\pfS-$arc is said to be trivial if it can be isotoped to a subset of a boundary edge.
  
  Let $B(\fS)$ be the set of all isotopy classes of simple $\pfS-$tangles such that the height of the intersection with the boundary wall increases along the orientation of the boundary induced by the orientation of $\fS$, and all $-$ boundary states occur before any $+$ boundary states in the order determined by this orientation.  Then we have the following result of the second author.
  
  \begin{theorem}[Theorem $2.8$ in \cite{Le4}]
  \label{surfbasis}
  Suppose $\fS$ is a punctured bordered surface.  Then $B(\fS)$ is an $\mathcal{R}$-basis of $\cS(\fS).$
  \end{theorem}

\subsection{Height exchange relations}

We have the following height exchange relations.

\def\reordonez{  \raisebox{-10pt}{\incl{1 cm}{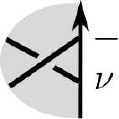}} }
\def\reordonezz{  \raisebox{-10pt}{\incl{1 cm}{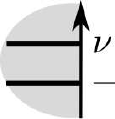}} }
\def\reordonezy{  \raisebox{-10pt}{\incl{1 cm}{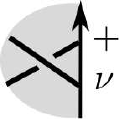}} }
\def\reordonezyy{  \raisebox{-10pt}{\incl{1 cm}{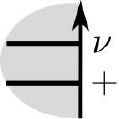}} }

\begin{lemma}[Height exchange move, Lemma 2.4 of \cite{Le4}]\label{r.refl}
 For $\nu\in \{\pm \}$  one has
\begin{align}
\label{eq.reor}
\reordonez \  &= \ q^{\nu } \,  \reordonezz \\
\label{eq.reorsecond}
 \reordonezy &= q^{\nu} \, \reordonezyy
\end{align}
\end{lemma}

Here we have identify $\pm$ with $\pm 1$ when we write $q^\nu$.

\def\cSs{\cS}
\def\tsigma{\sigma}
\subsection{Reflection anti-involution}

\begin{proposition}[Reflection anti-involution, Proposition 2.7 in \cite{Le4}]   \label{r.reflection}
Let $\fS$ be a punctured bordered surface. Suppose $\cR=\BZ[q^{\pm 1/2}]$. There exists a unique $\BZ$-linear map $\sigma: \cSs(\fS) \to \cSs(\fS)$ such that
\begin{itemize}
\item  $\sigma(q^{1/2})= q^{-1/2}$, 
\item $\sigma$ is an anti-automorphism, i.e. for any $x,y \in \cSs(\fS)$,
$$\tsigma(x+y)= \tsigma(x) + \tsigma(y), \quad \tsigma(xy) = \tsigma(y) \tsigma(x),$$
\item if $\al$ is a stated $\pfS$-tangle diagram then $\sigma(\al)$ is the result of switching all the crossings of $\al$ and reversing the linear order on each boundary edge.
\end{itemize} 
 \end{proposition}
Clearly $\tsigma^2=\id$. We call $\tsigma$ the reflection anti-involution.

\subsection{Functoriality for surfaces} \label{sec.cut}
 {\em An embedding} of a marked surface $\SM$ into a marked surface  $(\Sigma', \mathcal{P}')$ is an orientation preserving   proper embedding $f: \Sigma \embed \Sigma'$ such that  $f(\cP)\subset \cP'$. Then $f$ induces  an $\mathcal{R}$-algebra  homomorphism $f_*: \cS\SM \to \cS(\Sigma', \cP')$.

The case of punctured bordered surface is different, and more subtle. 

Suppose $f: \fS \embed \check{\fS}$ is an orientation preserving proper embedding of a punctured bordered surface $\fS$ into another punctured bordered surface $\check{\fS}$. In general, the induced map $f_*: \cS(\fS) \to \cS(\check{\fS})$ might not be well-defined, since if two boundary edges of $\fS$ are mapped into one boundary edge $b$ of $\check{\fS}$, then the height order of the image of (stated) tangle on $b$ might not be well-defined. 

For each boundary edge $b$ of $\check{\fS}$ choose a linear order of all the boundary edges of $\fS$ that mapped into $b$. Then $f$, equipped with such an ordering for every boundary edge of $\check{\fS}$, gives a well-defined $\cR$-linear map   $f_*: \cS(\fS) \to \cS(\check{\fS})$, but $f_*$ is not an algebra homomorphism in general.

For example, if the boundary of $\check{\fS}$ is equipped with an orientation (which on a component might or might not be the orientation induced from $\check{\fS}$) then we can order all the boundary edges of $\fS$ that are mapped into a boundary edge $b$ of $\check{\fS}$ using the orientation of $b$: as one moves on $b$ along the direction of the orientation, the order is increasing. Thus, for any proper embedding of punctured bordered surfaces
$f: \fS \embed \check{\fS}$ where $\partial \check{\fS}$ is equipped with an orientation, we can define an $\cR$-linear map $f_*: \cS(\fS) \to \cS(\check{\fS})$.

\subsection{Ideal triangulations of surfaces}
\label{IdealTriang}

An ideal triangle, denoted $\mathfrak{T}$, is the punctured bordered surface obtained from the closed disk by removing three boundary points, which we view as a triangle without its vertices. An ideal triangulation of a punctured bordered surface, $\fS$, is a realization of $\fS$ as the result of gluing a finite collection of ideal triangles along their boundary edges. We say that a punctured bordered surface is triangulable if it  admits an ideal triangulation. 

Another way to define ideal triangulations is the following. An {\em ideal arc} in $\fS$ is the image of a proper embedding
 $c:(0,1)\embed \fS$. This means, if we present $\fS=\bfS \setminus \cV$ where $\bfS$ is a compact surface with boundary and $\cV$ is a finite subset, then $c$ can be extended to an immersion $\bar c : [0,1] \to \bfS$ such that $\bar c(0), \bar c(1) \in 
 \cV$. An ideal arc is trivial if it bounds a disk.
 
  Then when $\fS$ is triangulable,  an {\em ideal triangulation} of $\fS$ is a maximal collection of disjoint non-trival ideal arcs which are pairwise non-isotopic.  
 
 We note that non-triangulable surfaces are (i) surfaces with $|\cV|=0$, (ii) $\bfS$ is a sphere with $|\cV|\leq 2$, and (iii) $\bfS$ is a disk with $\cV\subseteq \partial \bfS$ where $|\cV|\leq 2$.

\subsection{Positive state submodule}
\label{positivesubmod}
The submodule $\cSp\MN$ of $\cS\MN$ spanned by $\cN$-tangles such that each state is positive was introduced in \cite{Le3}. The corresponding notion $\cSp\SM$ of a marked surface was first defined by Muller \cite{Muller} in connection with quantum cluster algebras. In a sense, the Muller algebra $\cSp\SM$ is a quantization of the decorated Teichm\"uller space of Penner \cite{Penner}.

\section{Splitting homomorphism} 
The ability to split stated skein algebras along ideal arcs, as introduced in \cite{Le4}, has been instrumental in understanding the structure of stated skein algebras of surfaces \cite{CL}.  We now show that this can be generalized to stated skein modules in general.  
%\red{ An important property of the stated skein modules is the behavior under the splitting homomorphism. This is a generalization of a similar fact proved for surfaces in \cite{Le4}.}

\subsection{3-manifold case}
Suppose $\MN$ is a marked 3-manifold and $D$ is a properly embedded closed disk in $M$ which is disjoint from the closure of $\cN$. By splitting $M$ along $D$ we mean taking a 3-manifold $M'$ whose boundary contains two copies $D_1$ and $D_2$ of $D$ such that gluing $D_1$ to $D_2$ gives a manifold homeomorphic to $M$. Additionally, let $a\subset D$ be any choice of an oriented open interval along with $a_1\subset D_1$ and $a_2\subset D_2$ being the image of $a$ under the splitting $M$ along $D$.  Then the marked 3-manifold $(M', \cN')$,  where $\cN'= \cN \cup a_1 \cup a_2$, is called a splitting of $\MN$ along $(D,a)$.

An $\cN$-tangle $\al$ in $M$ is said to be {\em $(D,a)$-transverse} if its splitting along $(D,a)$ is an $\cN'$-tangle, meaning $\al$ is transverse to $D$,  $\al \cap D = \al \cap a$, and the framing at every point of $\al \cap a$  is a positive tangent vector of $a$. Suppose additionally that $\al$ is stated. The splitting $\al'$ of $\al$ then has new boundary points on $a_1 \cup a_2$ which do not have an associated value in $\{\pm\}$.
Given any map $s:\al \cap a \to \{\pm\}$ let $(\al',s)$ be the stated $\cN'$-tangle with $s(x)$ determining the state at any boundary point $x'$ in $a_1 \cup a_2$ by $s(x)$, where $x'$ is the image under splitting along $(D,a)$ of $x\in a$. We call $(\tal,s)$  a lift of $\al$, and note that if $|\al\cap a|=k$ then $\al$ has $2^k$ distinct lifts.

\begin{theorem}
\label{SplittingHomo} Suppose $\MN$ is a marked $3-$manifold and $D$ is a closed, properly embedded disk in $M$. Assume additionally that $D$ is disjoint from the closure of $\cN$, and $a$ is an oriented open interval in the  interior of $D$. 
Let $(M',\cN')$ be a result of splitting $(M, \cN)$ along $(D,a)$, as described above.

There is a unique $\mathcal{R}$-module homomorphism
$$ \Theta_{(D,a)}: \cS(M,\cN) \to \cS(M',\cN')$$
such that if $\al$ is an $\cN$-tangle $\al$  in $M$ which is $(D,a)$-transveral, then $\Theta_{(D,a)}(\al)$ is the sum of all lifts of $\al$. Utilizing the notation from abvoe,
$$ \Theta_{(D,a)}([\al]) = \sum_{s: \al \cap a \to \{\pm\}} [(\al',s)].$$

\end{theorem} 
\begin{remark}
When it is clear from context we will denote $\Theta_{(D,a)}$ by $\Theta$, or $\Theta_D$ if we look to only emphasize the disk $D$.  
\end{remark}

\def\pr{\mathrm{pr}}
\begin{proof} We will reduce the proof to the case covered by \cite[Theorem 3.1]{Le4}.  

Let $T(D,a)$ denote the $\mathcal{R}-$module freely generated by the set of stated $\cN$-tangles which are $(D,a)-$transverse, noting that this is the set of all tangles and not the collection of isotopy classes.  We also note that any isotopy class of $\cN$-tangles contains a $(D,a)-$transverse representative. Thus the skein module $\cS\MN$ is the quotient  of $T(D,a)$ by the equivalence relation generated by isotopies of $(D,a)-$transverse  $\cN$-tangles and the defining relations in Figure \ref{fig:stated-relations}.

 Define an $\cR$-linear map $\widehat{\Theta}:T(D,a)\rightarrow \mathscr{S}(M',\cN')$
by
\[\widehat{\Theta}(\alpha)=\sum_{s: \al \cap a \to \{\pm\}} [(\tal,s)].\]

We will show that $\widehat{\Theta}$ descends to a well defined map, $\Theta:\mathscr{S}(M,\cN)\rightarrow \mathscr{S}(M',\cN')$, by  showing that $\widehat{\Theta}$ is constant on isotopy classes, and that the image of the defining relations in Figure \ref{fig:stated-relations} are equal.

Now take $U$ to be  an embedded collared neighborhood of $D$.
Then any isotopy can be decomposed into stages which are supported in the interior of $U$ and stages which are supported in the complement of $D$. The support of any defining relation 
can also be assumed either to be in $U$ or disjoint from $D$.
 If the support of an isotopy or a defining relation is disjoint from $D$, then the image under $\widehat{\Theta}$ is unchanged by definition.  Now we see that isotopies and defining relations which are supported in the interior of $U$ can be represented diagrammatically.  Take a diffeomorphism of the interior of $U$ to the open cube $(-1,1)^3$ such that $a\subset \{(0,0)\}\times(-1,1)$ and $D$ is the closure of $\{0\}\times(-1,1)^2$.  In this standardized form,  an isotopy supported in $U$ can be further decomposed into a combination of Reidemeister moves and additional moves created by tangency to $D$, crossings coinciding with $D$ and height exchanges.  The invariance of these moves and the defining relations is verified by considering diagrammatic projections onto $D$ in Theorem $3.1$ of \cite{Le4}. Thus $\Theta$ is well defined.
\end{proof}

\begin{remark}
\label{notcloseddisk} In Theorem \ref{SplittingHomo} we can relax the requirement that $D$ is a closed disk. We may instead assume that $D$ is a disk obtained from a closed disk by removing  a finite number of closed intervals on the boundary. When such a disk $D$ is properly embedded into $M$ and has an open interval $a$ in its interior, we can split $\MN$ as usual, and Theorem \ref{SplittingHomo} still holds.  This case actually reduces to the case of splitting along a closed disk by utilizing the pseudo-isomorphisms of Example \ref{exa.1}.
\end{remark}

\begin{remark}
Immediately from the definition of the splitting homomoprhism we see that for any two disjoint embedded disks $D_1$ and $D_2$ splitting along the two disks commutes in the sense that
\[\Theta_{D_1}\circ \Theta_{D_2}=\Theta_{D_2}\circ \Theta_{D_1}\]
\end{remark}

\def\fB{{\mathfrak B}}

\subsection{Specializing to surfaces} Specializing the splitting homomorphism of Theorem \ref{SplittingHomo} to the case of thickenings of marked surfaces, and then translating to the language of punctured bordered surfaces as explained in Subsection \ref{sec.equiv}, we can recover the splitting homomorphism of stated skein algebras originally developed in \cite{Le4}.

Let $c$ be an ideal arc of a punctured bordered surface $\fS$.
By splitting $\fS$ along an ideal arc $c$ we mean taking a new punctured bordered surface $\fS'$ whose boundary contains two boundary edges $c_1$ and $c_2$  which when glued together give a surface homeomorphic to $\fS$ where the image of $c_1$ and $c_2$ is $c$. The thickening of $\fS$, denoted $\tfS:=\fS \times (-1,1)$, has boundary $\tpfS=\pfS \times (-1,1)$. A $\tpfS$-tangle $\al$ is {\em vertically transverse to} $c\times (-1,1)$ if $\al$ is transverse to $c\times (-1,1)$, the framing at every point of $\al\cap (c\times(-1,1))$ is vertical, and the heights of points in $\al\cap (c\times(-1,1))$ are distinct.  Additionally, assume that $\al$ is stated. Then given any $s: \al\cap (c\times(-1,1))\to \{\pm\}$, let $(\al's)$ be the $\tpfS'$-tangle (in $\tfS'$) which is the splitting of $\al$ with states on the newly created boundary points determined by $s$, and call such $(\al',s)$ a lift of $\al$. Theorem \ref{SplittingHomo} for punctured bordered surfaces becomes the following, which was originally proved in \cite{Le4}.

\begin{theorem}[Theorem $3.1$ of \cite{Le4}]
\label{surfacesplit} Suppose $c$ is  an ideal arc of a punctured bordered surface $\fS$ and $\fS'$ is a splitting of $\fS$ along $c$. There is a 
 is a unique $\mathcal{R}$-algebra homomorphism \[\Theta_c:\mathscr{S}(\fS)\rightarrow \mathscr{S}(\fS')\]
 such that if $\al$ is a stated $\tpfS$-tangle vertically transverse to $c \times (-1,1)$ then $ \Theta_c(\al)$ is the sum of all the lifts of $\al$. Moreover $\Theta_c$ is an algebra embedding. 
\end{theorem}

The fact that $\Theta_c$ is an algebra homomorphism is clear from the definition; that it is an embedding follows from a consideration of the bases of the relevant stated skein algebras.

\def\onto{\twoheadrightarrow}
\def\tc{{\tilde c}}
\def\tfS{\check{\fS}}

\subsection{Splitting the splitting homomorphism}\label{sec.spl2}
Suppose $c$ is an oriented ideal arc of a punctured bordered surface $\fS$, and $\cV\subset c$ is a finite collection of points on $c$. Let $\tfS = \fS\setminus \cV$. Then $c\setminus \cV =\sqcup_{i=1}^k c_i$ is the disjoint union of ideal arcs $c_i$ of $\tfS$. Let $\fS'$ be the result of splitting $\fS$ along $c$, and $\tfS'$ be the result of splitting $\tfS$ along all $c_i$. We have natural embeddings $\iota: \tfS \embed \fS$ and $\iota: \tfS' \embed \fS'$ as seen in Figure \ref{fig:split4}.

\begin{figure}[htpb]
    \centering
    \includegraphics[scale=.45]{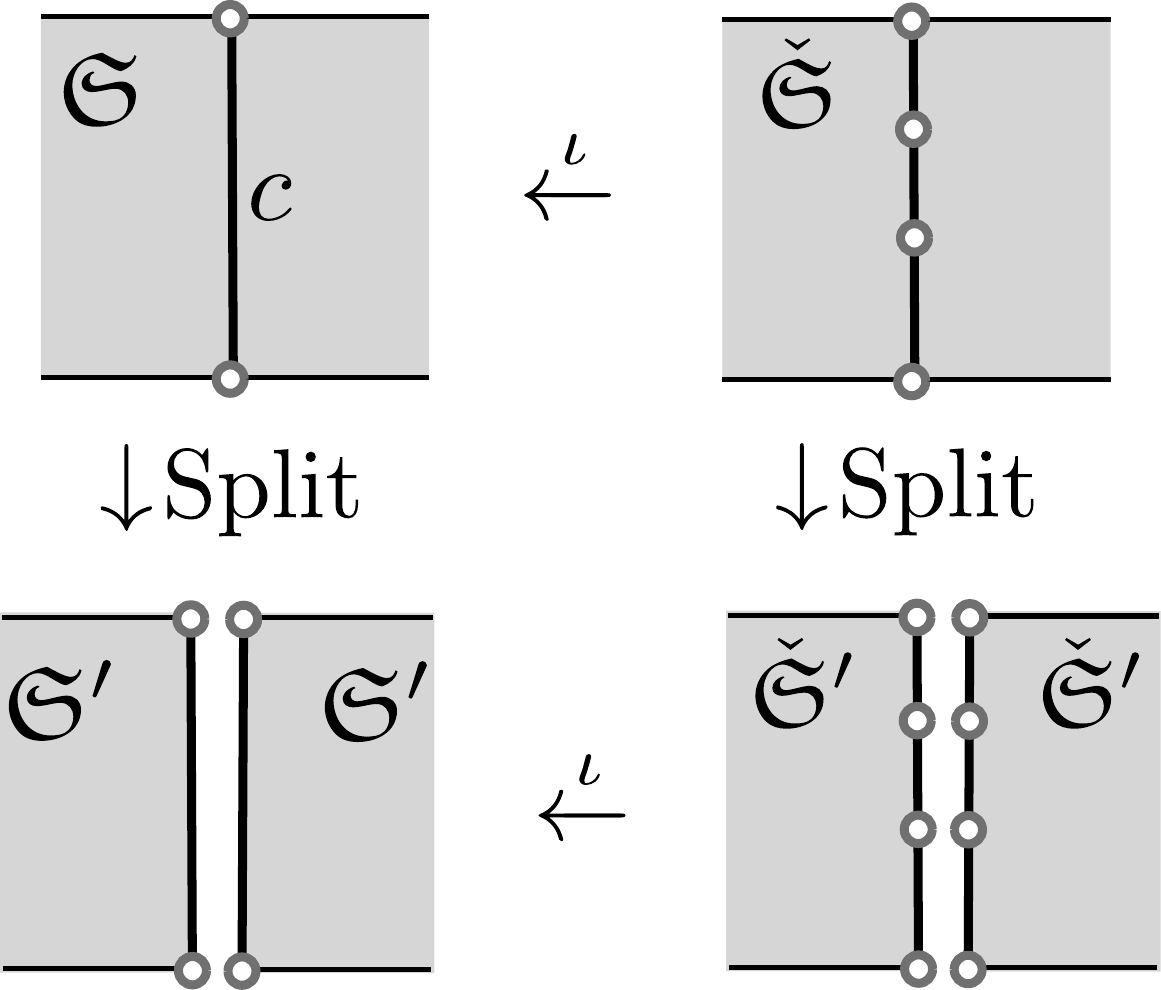}
    \caption{The top left of this diagram shows the ideal arc $c$ in $\fS$. The top right shows the result of removing $k$ points (in this case $3)$ on $c$ to give a family of ideal arcs, $c_i$, in $\tfS$. The bottom left shows a splitting of $\fS$ along $c$, called $\fS'$. The bottom right shows the result of splitting $\tfS$ along each $c_i$ to get $\tfS'$.
 The horizontal maps $\iota$ are the natural embeddings of punctured bordered surfaces which can be considered additionally with compatible orientations on each $c_i$ with an orientation on $c$. The vertical arrows indicate splitting along the appropriate ideal arcs, but these do not correspond to maps between surfaces.} 
    \label{fig:split4}
\end{figure}

The orientation of $c$ induces an orientation on the two boundary edges of $\fS'$ which are the image of $c$ under splitting along $c$. Using the orientation of these copies copies of $c$ we can use functoriality to define the induced map $\iota_*: \cS(\tfS') \to \cS(\fS')$, as described in Subsection \ref{sec.cut}.

By the commutativity of the splitting homomorphisms, the composition of all of the splitting homomorphisms $\Theta_{c_i}$ can be taken in any order. With a slight abuse of notation, we also denote this composition by $\Theta_c: \cS(\tfS) \to \cS(\tfS')$.
From the definition of the splitting homomorphism applied to stated tangle diagrams we have the following result.
\begin{lemma}
\label{r.cutk} One has that $\Theta_c$ and $\iota_*$ commute, in the sense that
the following diagram is commutative:

\[\begin{tikzcd}
\mathscr{S}(\fS)\arrow[d,"\Theta_c"'] & \arrow[l,"\iota_*"']\mathscr{S}(\tfS)\arrow[d,"\Theta_c"] \\
\mathscr{S}(\fS') & \arrow[l,"\iota_*"]\mathscr{S}(\tfS')\\
\end{tikzcd}\]

\end{lemma}

\subsection{The bigon and a coaction}
\label{sec.bigon}
 
\def\SB{\cS(\fB)} 
\def\USL{{U_{q^2}(sl_2)}}
 
The  Hopf algebra $\mathcal{O}_{q^2}(SL(2))$, known as the quantum coordinate ring of the Lie group $SL_2$, is 
 the $\mathcal{R}-$algebra generated by $a,b,c,d$ with relations
\begin{align}
&ca=q^2ac, \quad  db=q^2bd,\quad  ba=q^2ab,\quad  dc=q^2cd,\quad bc=cb,  \label{eq.OSL}\\
& ad-q^{-2}bc=1, \quad da-q^2cb=1 \label{eq.adda}
\end{align}
The coproduct $\Delta$, counit $\epsilon$, and antipode $S$  are given by
\be \Delta(a)=a\otimes a+b\otimes c, \Delta(b)=a\otimes b+b\otimes d, \Delta(c)=c\otimes a+d\otimes c, \Delta(d)=c\otimes b+d\otimes d  \label{eq.Delta}
\ee
\[\epsilon(a)=\epsilon(d)=1, \quad \epsilon(c)=\epsilon(b)=0\]
\[S(a)=d,\quad  S(d)=a, \quad S(b)=-q^{2}b,\quad  S(c)=-q^{-2}c.\]

In addition, the Hopf algebra $\mathcal{O}_{q^2}(SL(2))$ has a co-$R$-matrix with which it becomes a dual quasitriangular Hopf algebra (see \cite[Section 2.2]{Majid}), which is also known as a co-braided Hopf algebra (see \cite[Section VIII.5]{Kassel}). The co-$R$-matrix, whose explicit  formula will be recalled in Definition \ref{co-Rmat}, helps to make the category of $\OSL$-modules a ribbon category allowing for the construction of quantum invariants of links and 3-manifolds. 

The bigon $\fB$ is the punctured border surface obtained from the closed disk by removing two of its boundary points. Alternatively we can identify $\fB$ with $(-1,1) \times [0,1]$, and the arc $\al= 0 \times [0,1]$ is called the {\em core} of $\fB$, see Figure \ref{fig:bigon1}.

\begin{figure}[htpb]
    \centering
    \includegraphics[scale=.3]{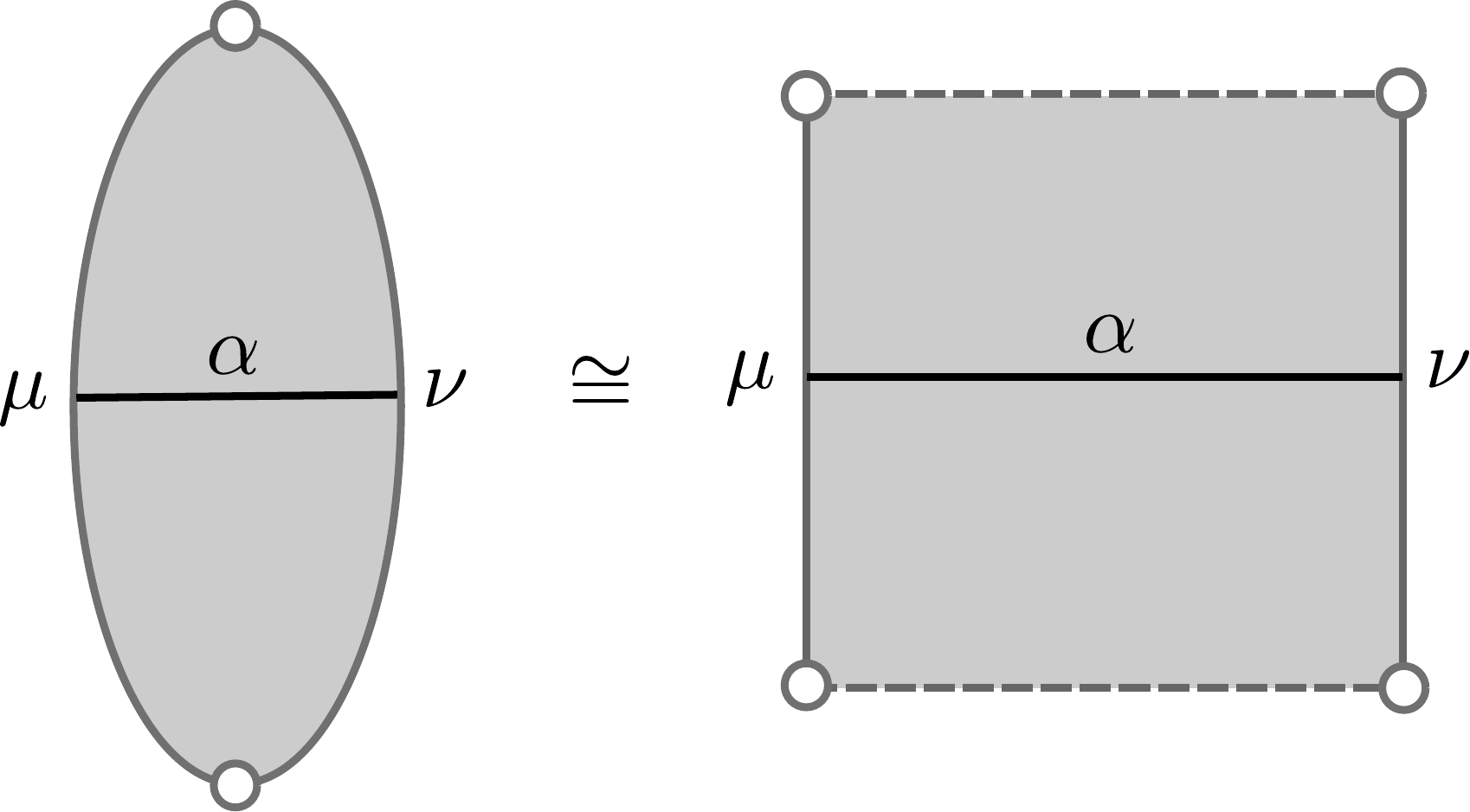}
    \caption{Two depictions of the bigon, with the arc $\alpha$ stated with $\mu$ and $\nu$.}
    \label{fig:bigon1}
\end{figure}

Splitting $\fB$ along an ideal arc connecting the two ideal points gives a disjoint union of two copies of $\fB$. Hence the splitting homomorphism gives an algebra homomorphism
$$ \Delta: \SB \to \SB \ot \SB,$$
which is compatible with the algebra structure. There are also geometric definitions of a counit, an antipode, and a co-$R$-matrix which make $\SB$ a co-braided Hopf algebra.  It is shown in \cite{CL} that there is an isomorphism of  co-braided Hopf algebras $\SB$ and $\OSL$ given by
\be 
\al_{++} \to a, \ \al_{+-} \to b, \ \al_{-+} \to c, \ \al_{--} \to d. \label{eq.isoOSL}
\ee

The bigon $\fB$, a punctured bordered surface, corresponds to the marked surface $(\fB, \cP)$, where  $\cP$ consists of two points, one on each on each boundary edge. We call $(\fB, \cP)$ the {\em marked bigon}.

Let $a$ be a connected component of the set of markings $\cN$ of a marked 3-manifold $\MN$. A small neighborhood of $a$ in $\partial M$ is a disk $D$. By pushing the interior of $D$ inside $M$ we get a new disk $D'$ which is properly embedded in $M$. Splitting $\MN$ along $D'$, we get a new marked 3-manifold $(M', \cN')$ isomorphic to $\MN$, and another marked 3-manifold bounded by $D$ and $D'$. The latter, after removing the common boundary of $D$ and $D'$, is isomorphic to the thickening of the marked bigon.
 Hence the splitting homomorphism gives an $\cR$-homomorphism
$$ \Delta_a: \cS\MN \to \cS\MN \ot \OSL.$$
Directly following the argument for the case of surfaces in \cite{CL}, one can check that $\Delta_a$ gives $\cS\MN$ a right comodule structure over $\OSL$. As $\OSL$ is the Hopf dual of the quantized envelopping algebra $\USL$, a right comodule over $\OSL$ is a left module over $\USL$, whose representation theory is well known. In the case of surfaces, the $\USL$-module structure of $\cS\MN$ is described in \cite{CL}, and the method developed there can be used to study the $\USL$-module structure of $\cS\MN$ in the 3-manifold case as well.

For later use, let us record the  following computation.
\begin{lemma} \label{r.ad1}
Consider the algebra $\OSL$ with generators $a,b,c,d$ as defined above.

(a) For any $x\in \{a,b,c,d\}$ let $x_+$ and $x_-$ be respectively the first and the second term in the formula of $\Delta(x)$ given by \eqref{eq.Delta}. Then 
\be 
x_- x_+ = q^4 x_+ x_-.  \label{eq.x1}
\ee

 (b) If $q$ is a root of 1 with $\ord(q^4)=N$ then 
\be 
T_N(a+d) = a^N + d^N. \label{eq.ad1}
\ee

\end{lemma}
\begin{proof}
$(a)$ This follows immediately from  the commutations in equation \eqref{eq.OSL}.

(b) From the identities of equation \eqref{eq.adda} we see that
\be q^2 ad - q^{-2} da = q^2 - q^{-2}  \ee  
Consider the following algebras which are closely related to the $q-$boson algebra considered by Kashiwara \cite[Section 3]{Kashiwara},
 \begin{align*}
 A&= \cR\la a, d \ra /(q^2 ad - q^{-2} da = q^2 - q^{-2}) \\
 A'&= \cR \la x^{\pm 1}, y^{\pm 1} \ra /(yx = q^4 xy).
\end{align*}  
There is an algebra embedding from $f:A\embed A'$ given by  $a\to x, d \to x^{-1} +y$. The injectivity of this map can be seen from a lead term argument using the degree of each term.

As $yx = q^4 xy$ and $q^4$ is a root of 1 order $N$, we have, following identity $(c)$ of Lemma \ref{qfacts}, that
\be 
(x+y)^N = x^N + y^N.
\label{eq.q1}
\ee
We then have
\begin{align}
f(a^N+ d^N) &=  x^N + (x^{-1} + y)^N = x^N + x^{-N}+ y^N.  \label{eq.z1}\\
f(T_N(a+d))&= T_N(f(a+d)) = T_N(x+ x^{-1} + y). \label{eq.z2}
\end{align}
It is known that, see \cite[Equation (2)]{Bonahon} or \cite[Corollary 3.2]{LP}, in the quantum torus $A'$, with $q^4$ a root of unity of order $N$, the right hand side of equation \eqref{eq.z1} and the right hand side of equation \eqref{eq.z2} are equal. Hence we have equation \eqref{eq.ad1}.
 \end{proof}
The special case of \eqref{eq.ad1} when $q^{1/2}$ is a roof of odd order is also proved in \cite{KQ}.
 
 \def\fA{\mathfrak A}
\subsection{Open annulus} \label{sec.openannulus}
The open annulus $\fA:= (-1,1) \times S^1$ is a  punctured bordered surface having empty boundary. It is diffeomorphic to  the sphere with 2 points removed, meaning it has two ideal points. The  curve $z=0 \times S^1$ is called the core of $\fA$. The skein algebra $\cS(\fA)$ is isomorphic to $\cR[z]$, the polynomial algebra in $z$. Splitting $\fA$ along an ideal arc, $c$, which connects the two ideal points gives a bigon. The splitting hommorphism is given by
 \be 
 \Theta_c(z)= a+d,
 \label{eq.AB}
 \ee
 where $a$ and $d$ are elements of $\OSL$ which is identified with $\SB$ under the isomorphism defined in equation \eqref{eq.isoOSL}.

\section{The Chebyshev-Frobenius homomorphism}

\subsection{Root of unity}  When $\mathcal{R}=\BC$, the field of complex numbers, and $q^{1/2}=\omega$ is a non-zero complex number, we denote $\cS\MN$ by $\Sx\MN$. If $\omega$ is a root of 1, let $\ord(\omega)$ be the least positive integer $n$ such that $\omega^n =1$.

\def\fr{{\mathrm{fr}}}
\subsection{The main theorem}

Recall the Chebyshev polynomials of type one, $T_n(x) \in \BZ[x]$, are defined recursively as
\begin{align*}
T_0(x)=2, \ \ \ T_1(x)=x, \ \ \ T_n(x) = xT_{n-1}(x)-T_{n-2}(x), \ \ \forall n \geq 2. 
\end{align*}

If $\al$ is an $\cN$-arc or $\cN$-knot in a marked 3-manifold $\MN$, then for every $k\in \BN$ let $\al^{(k)}$ be the disjoint union of $k$ parallel copies of $\al$ where this parallel push-off is done in the direction of the framing of $\al$. Note that the $\cN$-isotopy class of  $\al^{(k)}$ is uniquely determined by $\alpha$.  We will often abuse notation and not distinguish between $\al^{(k)}$ and $[\alpha^{(k)}]$ as an element of the skein module $\Sx\MN$. If $f(x)= \sum c_k x^k\in \BC[x]$ is a polynomial, then we define the threading of $\al$ by $f$ to be the element
$ f^\fr(\al) := \sum _k c_k \al^{(k)}$ in $\Sx\MN$.
If $\beta$ is another one-component $\cN$-tangle and $g=\sum d_j x^j\in \BC[x]$ then we define
$$ f^\fr(\al) \cup g^\fr(\beta)= \sum c_k d_j\,  \al^{(k)} \cup \beta^{(j)}$$
as an element of the skein module $\Sx\MN$. Similarly, we can define $\cup_{i=1}^m (f_i)^\fr(\al_i) \in \Sx\MN$, where each $\al_i$ is a one-component $\cN$-tangle and each $f_i$ is a polynomial.

\begin{theorem}
\label{main}

Suppose
 $\MN$ is a marked 3-manifold and $\omega$ is a complex root of unity. Let $N=\ord(\omega^8)$ and $\eta=\omega^{N^2}$.  
 
(a)  There exists a unique $\BC$-linear map $\Phi_\omega : \cS_\eta(M,\cN) \to \cS_\omega (M,\cN)$ such that for any $\cN$-tangle $\al$, considered as an element of $\cS_\eta(M,\cN)$,  its image $\Phi_\omega(\al)$ is the result of threading each knot component of $\al$ by $T_N$ each $\cN-$arc component by $x^N$.  Meaning explicitly, that for an $\cN$-tangle $\al = a_1 \cup \cdots \cup a_k \cup \al_1 \cup \cdots \cup \al_l$ where the $a_i$ are $\cN$-arcs and the $\al_i$ are $\cN$-knots,
\begin{align}
\Phi_\omega (\al) & = a_1^{(N)} \cup \cdots \cup a_k^{(N)} \cup (T_N)^{\fr}(\al_1) \cup \cdots \cup (T_N)^{\fr}(\al_l) \quad \text{in }\Sx\MN.  \label{eq.defC}
%\\ &: = \sum_{0\le j_1, \dots, j_l\le N} c_{j_1} \dots c_{j_l}  a_1^{(N)} \cup \dots \cup  a_k^{(N)} \cup \, \al_1^{(j_1)} \cup \cdots \cup \al_l^{(j_l)} \quad \text{in }\Sx\MN.
 \end{align}

(b) 
Additionally, we have that $\Phi_\omega $ commutes with the splitting homomorphism, meaning if $(M',\cN')$ is the result of splitting $\MN$ along $(D,a)$, then
 the following diagram commutes:
\be 
\begin{tikzcd}
\mathscr{S}_\eta(M,\cN)\arrow[r,"\Theta_{(D,a)}"]\arrow[d,"\Phi_\omega "] & \mathscr{S}_\eta(M',\cN')\arrow[d,"\Phi_\omega "] \\
\mathscr{S}_\omega (M,\cN)\arrow[r,"\Theta_{(D,a)}"]  & \mathscr{S}_\omega (M',\cN')\\
\end{tikzcd}
\label{eq.dia}
\ee

\end{theorem}

We call $\Phi_\omega$ the Chebyshev-Frobenius homomorphism. For a brief history see Subsections  \ref{CHomoIntro} and \ref{CFHomoIntro}.

We first prove a few technical lemmas, before giving a proof of Theorem \ref{main} in Subsection \ref{sec.proof}.

\subsection{Setting}

Let $\tSe\MN$ be the $\BC$-module freely generated by $\cN$-tangles, noting that this is the set of tangles and not isotopy classes of tangles.
The kernel 
of the projection $\tSe\MN \onto \Se\MN$ is the $\BC$-subspace of $\tSe\MN$ generated 
$\al-\beta$ where $\al$ and $\beta$ are either related by a $\cN$-isotopy or are
 two sides of one of the defining relations seen in Figure \ref{fig:stated-relations}.

Define the $\BC$-linear map
$\tP_\omega : \tSe\MN \to \cS_\omega (M,\cN)$ such that if $\al$ is a stated $\cN$-tangle then $\tP_\omega (\al)$ is  the right hand side of \eqref{eq.defC}. By this definition $\tP_\omega(\al)$ depends only on the $\cN$-isotopy class of $\al$.
We will show that $\tP_\omega $ descends to a $\BC$-linear map
$\Phi_\omega : \cS_\ve(M,\cN) \to \cS_\omega (M,\cN)$ by proving that if $\al$ and $\beta$ are two sides of one of the defining relations seen in Figure \ref{fig:stated-relations} 
then 
\be 
\tP_\omega(\al) = \tP_\omega(\beta). \label{eq.premain}
\ee

\def\pfS{\partial \fS}

\subsection{Technical lemmas}

We first prove a special case the commutativity of Diagram~\eqref{eq.dia}.
\begin{lemma} \label{r.comm}
Suppose $D$ is a disk properly embedded into a marked 3-manifold $\MN$ and $D$ contains an open interval $u$. Let  $\al$ be a stated $\cN$-tangle which is $(D,u)$-transverse, such that $\al$ intersects $u$ in exactly one point. Let $(M',\cN')$ be the result of splitting $\MN$ along $(D,u)$. For $\nu\in \{\pm \}$ let $\al_\nu$ be the lift of $\al$, with both newly created boundary points having state $\nu$. Then
\be \tP_\omega(\al_+ + \al_-) = \Theta_c (\tP_\omega(\al)).
\label{eq.t1}
\ee

\end{lemma}
\begin{remark}Since $\Theta_c([\al])= [\al_+] + [\al_-]$, it is tempting to write the left hand side of \eqref{eq.t1} as $\tP_\omega(\Theta_c(\al))$. As $\tP_\omega$ is defined on tangles themselves and not equivalence classes $\tP_\omega(\Theta(\alpha))$ is not well defined.
\end{remark}

\def\pfB{\partial \fB}
\begin{proof} 
It is enough to consider the case when $\al$ has one component.  There are two cases which we will address separately: (i) $\al$ is an $\cN$-arc, and (ii) $\al$ is an $\cN$-knot.

(i) Assume $\al$ is a $\cN$-arc. A small tubular open neighborhood of $\al \cup D$ in $M$ is pseudo-isomorphic to 
the thickening of a marked bigon (see Subsection \ref{sec.bigon}), with $\al$ being identified with the core of the bigon and $D$ being the thickening of an ideal arc connecting the two ideal points of the bigon. Then utilizing functoriality we may assume, without loss of generality, from the outset that $\MN$ is the thickening of the marked bigon $\fB$, $\al$ is the core of the bigon, and $D$ is the thickening of an ideal arc connecting the two ideal points.

 Under the identification of $\SB$ with $\OSL$ by equation \eqref{eq.isoOSL}, the $\cN$-tangle $\al$, depending on the states at the end points, becomes an element  $ x\in \{a,b,c,d\}$. Moreover, $\al_+$ and $\al_-$ are exactly the two terms on the right hand side of the formula of $\Delta(x)$ given by equation \eqref{eq.Delta}.
 The left hand side of equation \eqref{eq.t1}  is
\begin{align}
\tP_\omega(\al_+ + \al_-)&= (\al_+)^N + (\al_-)^N, \label{t.1al}
\end{align}
 while the right hand side of equation \eqref{eq.t1} is
\be \Theta_c (\tP_\omega(\al))= (\al_+\, +\, \al_-)^N. \label{t.1ar} \ee
Then by equation \eqref{eq.x1} we have $\al_- \al+ = \omega^8 \al_+ \al_-$. With this commutation and $\ord(\omega^8)=N$, we have from equation \eqref{eq.q1} that $(\al_+\, +\, \al_-)^N=(\al_+)^N + (\al_-)^N$. This proves the Lemma in case (i).

(ii) Let $\al$ be a $\cN-$knot. A small tubular neighborhood of $\al\cup D$ is pseudo-isomorphic to the thickening of an open annulus (see Subsection \ref{sec.openannulus}) with $\al$  being identified with the core $z$ of the open annulus, and $D$ being identified with the thickening of an ideal arc connecting the two ideal point. Then utilizing functoriality we may assume, without loss of generality, that $\MN$ is equal to this neighborhood with the above identifications. The splitting of $\MN$  along $D$ is then the thickening of the bigon  $\fB$.
With the identification of $\SB$ with $\OSL$ by the isomorphism of equation \eqref{eq.isoOSL}, we have $\al_+=a$ and $\al_-= d$.
The left hand side of equation \eqref{eq.t1} is then
\be 
\tP_\omega(\al_+ + \al_-) =  a^N + d^N. \label{t.1bl}
 \ee

Now as $\Theta_c$ is an algebra homomorphism,  the right hand side of  of equation
\eqref{eq.t1} is
\be 
\Theta_c (\tP_\omega(\al))= \Theta_c (T_N(\al)) = T_N( \Theta_c (\al)  )= T_N (a+d). \label{t.1br}
 \ee
By Lemma \ref{r.ad1}(b), when $\ord(\omega^8)=N$, the right hand sides of \eqref{t.1bl} and \eqref{t.1br} are equal. Thus, we have \eqref{eq.t1}.
\end{proof}

\def\SS{\mathfrak{Q}}

\subsection{The square}
In the next three lemmas we prove special cases of Identities \eqref{eq.premain}.

Recalling the correspondence described in full in Subsection \ref{sec.equiv}, we say a punctured bordered surface $\fS$ corresponds to the marked surface $\SM$, where $\cP$ consists of one point on each boundary edge of $\fS$. By identifying  $\tilde \cP$-isotopy classes of $\tilde \cP$-tangles with $\tpfS$-isotopy classes of $\tpfS$-tangles (as in Subsection \ref{sec.equiv}), for every $\tpfS$-tangle, and hence for every $\pfS$-tangle diagram $\al$, we can define $\tPhi_\omega(\al) \in \cS_\omega(\fS)= \cS_\omega(\fS, \cP)$.

The ideal square $\SS$  is the punctured bordered surface obtained from the closed disk by removing four boundary points, which we view as a square without its vertices.
% Its corresponding marked surface is $(\SS,\cP)$ where $\cP$ consists of 4 points, one on each edge. 

% As $\cS_\omega(\SS,\cP)$ and $\cS_\omega(\SS)$ are canonically isomorphic, we will 
Let $X$ be a stated $\partial \SS$-tangle diagram in $\SS$ consisting of 2 arcs as in Figure \ref{fig:square}, with arbitrary states on the boundary. Let $X_+$ (respectively $X_-$) be result of the positive (respectively negative) resolution of the only crossing of $X$.

\begin{figure}[htpb]
    \centering
    \includegraphics[scale=.35]{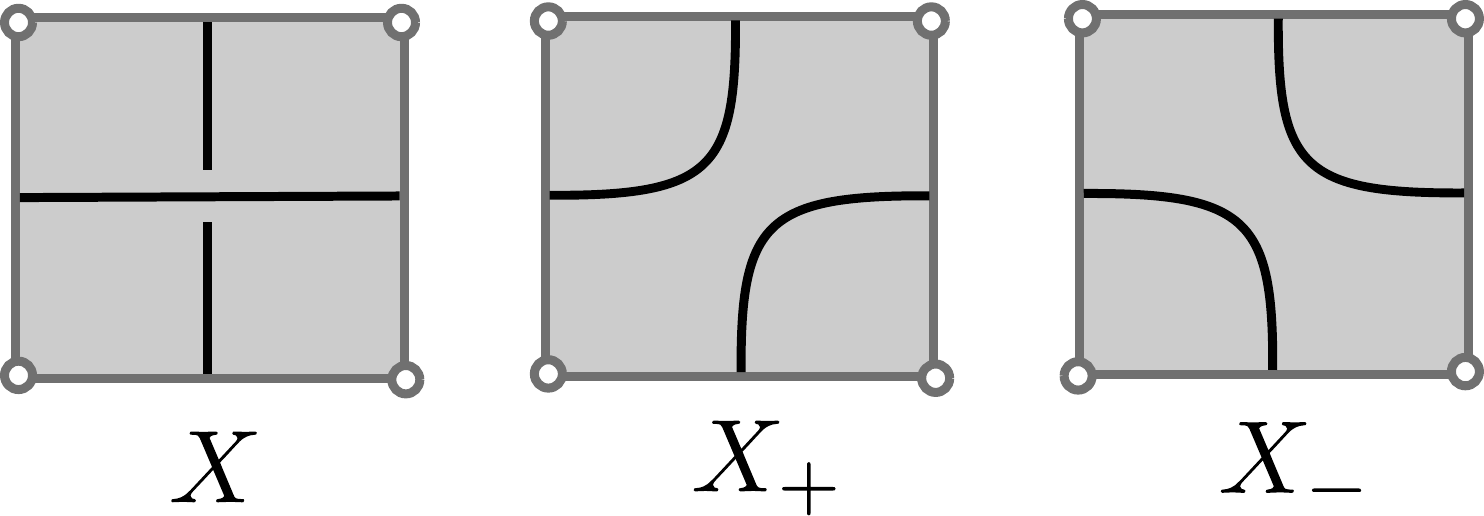}
    \caption{The ideal square and stated tangles $X, X_+, X_-$ on it. The states of $X,X_+, X_-$ are the same on each boundary component.}
    \label{fig:square}
\end{figure}

\begin{lemma}
\label{Square} In $\Sx(\SS)$ one has
\be 
 \tP_\omega  (X)   =\eta ^{2}\, \tP_\omega  (X_+) +\eta ^{-2}\, \tP_\omega  (X_-).
 \label{eq.s00}
\ee
\end{lemma}

\begin{proof} 
We proceed by induction to prove that in $\mathscr{S}(\SS)$, for any ground ring $\cR$ containing an invertible $q^{1/2}$ and any $m\in \BN$, we have 
\be 
X^{(m)}= \sum_{j=0}^m q^{m^2-4mj+2j^2}{m\brack j}_{q^4} X_+^{(m-j)}X_-^{(j)}.
\label{eq.qq}
\ee

  The base case of $m=1$ is 
   $X = q X_+ + q^{-1} X_-$, which follows immediately from the skein relation $(A)$ of Figure \ref{fig:stated-relations}.  Now assume the formula holds true for a fixed $k$. The diagram of $X^{(k+1)}$ is presented by a grid as in Figure \ref{fig:squareinducthyp}, and the crossings will be parameterized similarly to entries of a $(k+1) \times (k+1)$ matrix. For example, the upper left crossing is the $(1,1)$-crossing. 
  Observe that all endpoints on a given boundary edge have identical states. Under these conditions, it follows from a combination of height exchange moves and the trivial arc relation that any tangle diagram having an arc with both endpoints on the same boundary is 0 in $\cS(\SS)$.  Note that this fact requires that all of the states on a given boundary edge are identical, and not just that the endpoints of the returning arc have equal states.
  
  \begin{figure}[htpb]
    \centering
    \includegraphics[scale=.35]{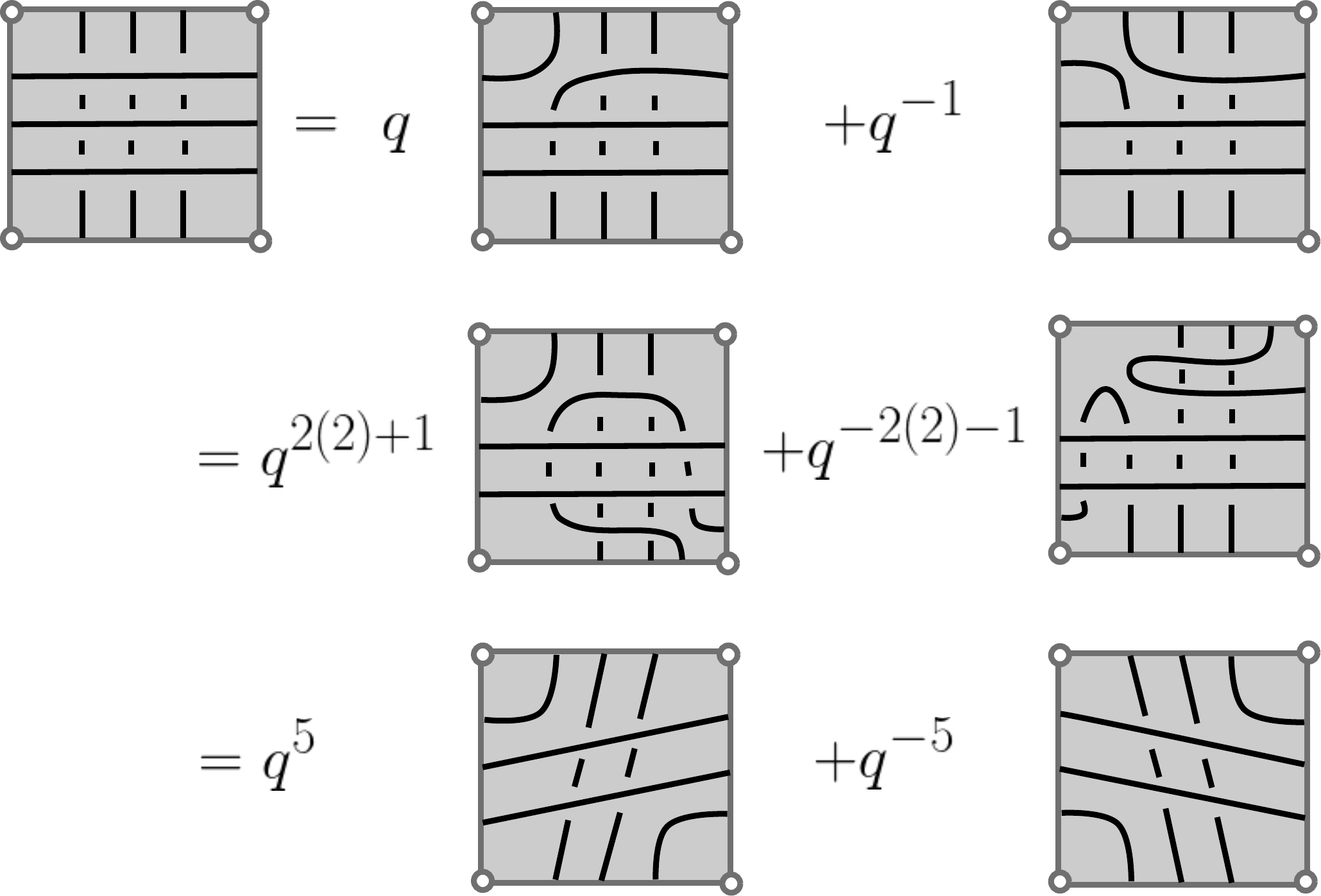}
    \caption{A depiction of the inductive hypothesis when decomposing $X^{(3)}$.}
    \label{fig:squareinducthyp}
\end{figure}

  Now define $h_+$ and $h_-$ be the tangle diagrams obtained as positive and negative resolutions, respectively, of the $(1,1)$-crossing in the diagram $X^{(k+1)}$.  Then by definition, the skein relation applied to the $(1,1)$-crossing of $X^{(k+1)}$ gives
  $$ X^{(k+1)} = q h_+ + q^{-1} h_-.$$
 Looking at consecutive resolutions in $h_+$ along the top row of crossings, meaning the $(1,2)$-crossing through the $(1,k+1)$-crossing, we see that each negative resolution would gives a trivial arc with the bottom boundary edge, meaning the only nonzero diagram occurs when all crossings are resolved positively. Then following a similar argument if we consecutively resolve the crossings down the rightmost column, meaning the $(2,k+1)$-crossing to the $(k+1,k+1)$-crossing, we have the only nonzero diagram occurs when each of these crossings are resolved positively. Together this implies
  $$ h_+= q^{2k}X_+\cdot X^{(k)}\in\cS(\SS).$$

  Similarly, a mirrored argument on $h_-$, meaning consecutively resolving the top row of crossings and the left-most column of crossings, gives that the only nonzero diagram occurs when each crossing is resolved negatively. Hence,
  $$ h_-=q^{-2k}X^{(k)}\cdot X_-\in\cS(\SS).$$
  Taken together we have
\be \label{eq.inductsep} X^{(k+1)}=q^{2k+1}X_+\cdot X^{(k)} + q^{-2k-1}X^{(k)}X_-,
\ee
which allows us to proceed with our inductive step.
%From here one can easily prove \eqref{eq.qq} by induction.  

\no{
\begin{figure}[htpb]
    \centering
    \includegraphics[scale=.45]{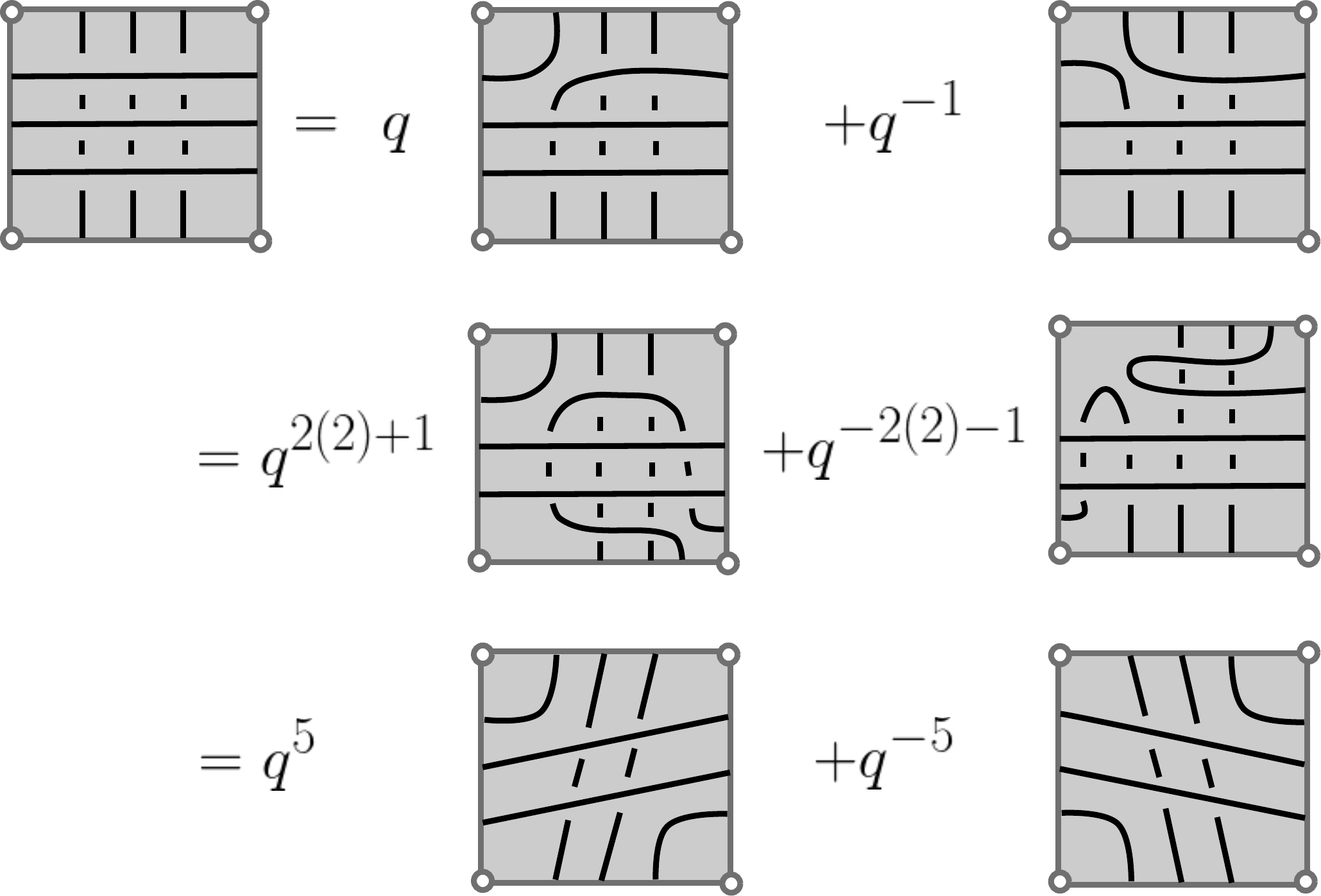}
    \caption{An illustration of the inductive hypothesis applied to the case $k=2$ and resolving the top left crossing}
    \label{fig:squareinducthyp}
\end{figure}
 }

Then applying our inductive hypothesis of equation \ref{eq.qq} to the right hand side of equation \ref{eq.inductsep} we have
$$q^{2k+1}(\sum_{i=0}^{k}q^{k^2-4ki+2i^2}{k \brack i}_{q^4}X_+^{(k+1-i)}X_-^{(i)})+q^{-2k-1}(\sum_{i=0}^{k}q^{k^2-4ki+2i^2}{k \brack i}_{q^4}X_+^{(k-i)}X_-^{(i+1)})$$

Combining like terms gives

\begin{align*}
%&=q^{(k+1)^2}X_+^{(k+1)}+q^{-(k+1)^2}X_-^{(k+1)}\\
%&+(\sum_{i=1}^kq^{k^2-4ki+2i^2+2k+1}{k \brack i}+q^{k^2-4k(i-1)+2(i-1)^2-2k-1}{k \brack i-1})X_+^{(k+1-i)}X_-^{(i)}\\
%&=q^{(k+1)^2}X_+^{(k+1)}+q^{-(k+1)^2}X_-^{(k+1)}\\
%&+(\sum_{i=1}^kq^{(k+1)^2-4(k+1)i+2i^2+4i}{k \brack i}+q^{(k+1)^2-4(k+1)i+2i^2}{k \brack i-1})X_+^{(k+1-i)}X_-^{(i)}\\
&q^{(k+1)^2}X_+^{(k+1)}+(\sum_{i=1}^kq^{(k+1)^2-4(k+1)i+2i^2}(q^{4i}{k \brack i}+{k \brack i-1})X_+^{(k+1-i)}X_-^{(i)})+q^{-(k+1)^2}X_-^{(k+1)}\\
\end{align*}
Then applying the the $q-$binomial identity from part $(b)$ of Lemma \ref{qfacts} simplifies to
\begin{align*}
&=\sum_{i=0}^{k+1}q^{(k+1)^2-4(k+1)i+2i^2}{k+1 \brack i}X_+^{(k+1-i)}X_-^{(i)},
\end{align*}

which completes the proof that equation \eqref{eq.qq} holds in $\cS(\SS)$.  

When $q^{1/2}=\omega$, the element $q^4=\omega^8$ is a root of 1 of order $N$. Then from part $(a)$ of Lemma \ref{qfacts} we have that $q^{m^2-4mj+2j^2}{m\brack j}_{q^4}=0$ unless $j=0$ or $j=N$. Applying this to equation \eqref{eq.qq} gives
\[X^{(N)}=\omega^{N^2}X_+^{(N)}+\omega^{-N^2}X_-^{(N)},\]
and our result follows as desired.
\end{proof}

\def\BB{\mathring\fB}

\subsection{The punctured bigon}
The {\em punctured bigon} $\BB$ is the result of removing an interior point from the bigon.
Let $Y$ be the  stated $\partial \BB$-tangle diagram on $\BB$  as in Figure \ref{fig:pbigon}, and $Y_+$ (respectively $Y_-$) be the result of the positive (respectively negative) resolution of the only crossing of $Y$. 

\begin{figure}[htpb]
    \centering
    \includegraphics[scale=.35]{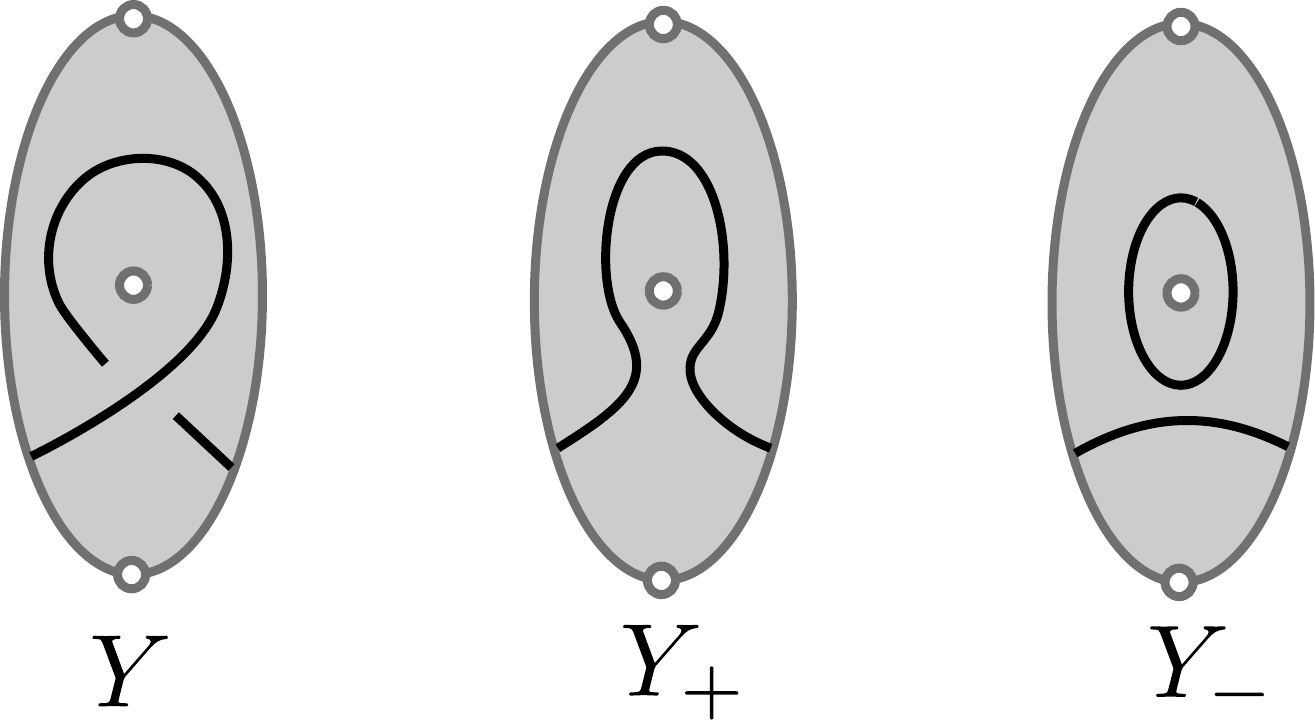}
    \caption{The punctured bigon $\BB$  and the stated tangles $Y, Y_+, Y_-$.}
    \label{fig:pbigon}
\end{figure}

\begin{lemma}
\label{pbigon} In $\Sx(\BB)$ one has
\be 
 \tP_\omega  (Y)   =\eta ^{2}\, \tP_\omega  (Y_+) +\eta ^{-2}\, \tP_\omega  (Y_-).
 \label{eq.s0}
\ee
\end{lemma}
\begin{proof}

Let $c$ be the ideal arc connecting the interior ideal point and the top ideal point,  as seen in Figure~\ref{fig:pbigon2}. 
\begin{figure}[htpb]
    \centering
    \includegraphics[scale=.35]{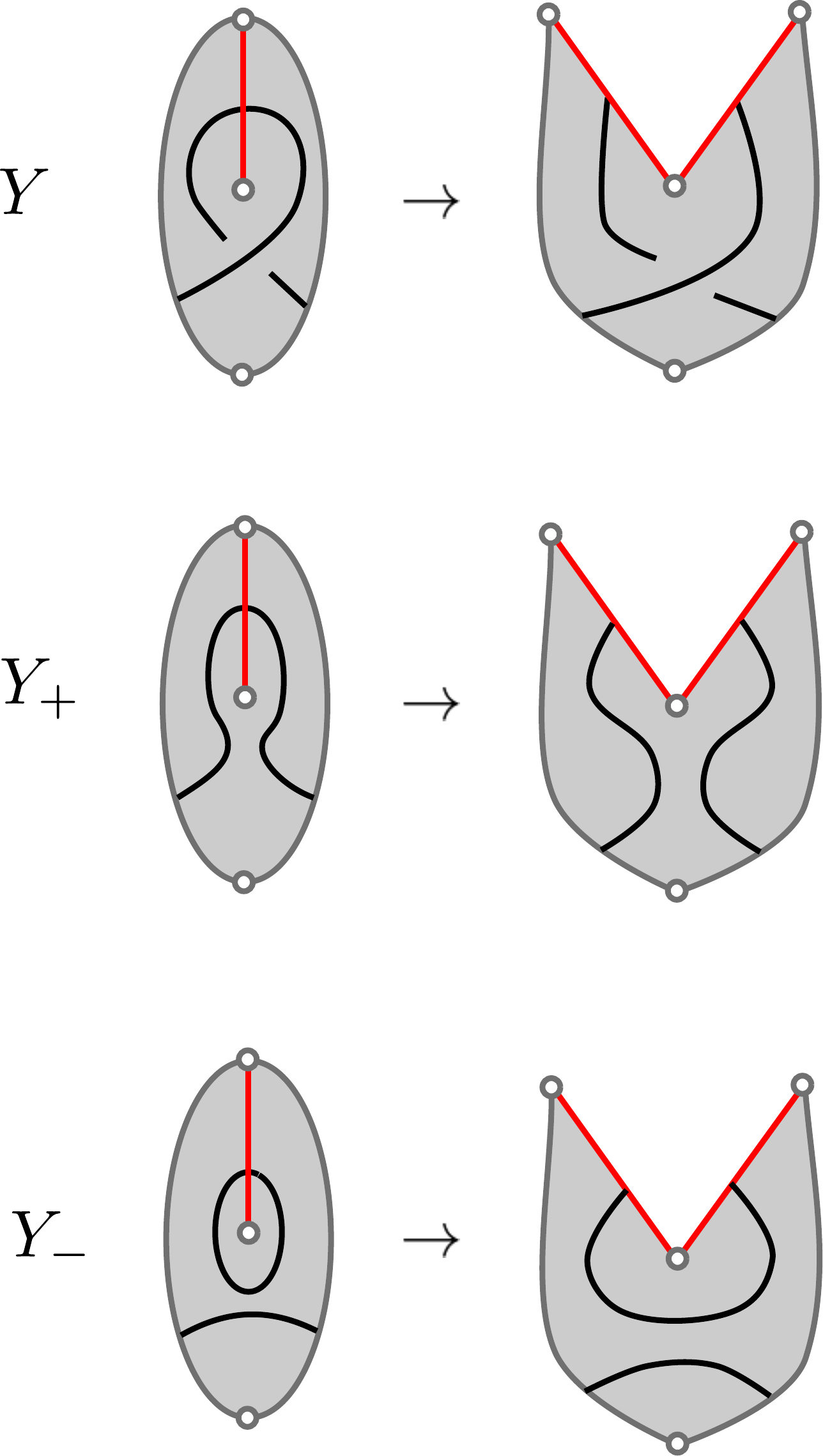}
    \caption{Cutting along the ideal arc $c$}
    \label{fig:pbigon2}
\end{figure}

The result of splitting $\BB$ along $c$ is the ideal square $\SS$. The two lifts of $Y$ are stated $\partial \SS$-tangle diagrams $X_1$ and $X_2$, each of which is an appropriately stated copy of $X$ from Figure \ref{fig:square}. Similarly, the two lifts of $Y_+$ are $X_{+,1}$ and $X_{+,2}$, each of which is a stated copy of $X_+$, and the two  lifts of $Y_-$ are $X_{-,1}$ and $X_{-,2}$, each of which is a stated copy of $X_-$.

 Since 
the splitting homomoprhism $\Theta_c: \Sx(\BB)\embed \Sx(\SS)$ is an algebra embedding, to prove Identity \eqref{eq.s0} one only needs to prove the validity of its image under $\Theta_c$, i.e.
\be 
\Theta_c( \tP_\omega  (Y))   =\eta ^{2}\, \Theta_c(  \tP_\omega  (Y_+))  +\eta ^{-2}\, \Theta_c( \tP_\omega  (Y_-)).
 \label{eq.s1}
\ee
By Lemma \ref{r.comm}, we have
\begin{align*}
\Theta_c( \tP_\omega  (Y)) & = \tPhi_\omega(X_1) + \tPhi_\omega(X_2)  \\
\end{align*} 
Using Lemma \ref{Square}, we have, for $i=1,2$
\be
\tPhi_\omega(X_i) =  \eta^2 \tPhi_\omega(X_{i,+}) +  \eta^{-2} \tPhi_\omega(X_{i,-}). \notag
\ee
Hence
\begin{align*}
\Theta_c( \tP_\omega  (Y)) & = \eta^2 (\tPhi_\omega(X_{+,1})+  \tPhi_\omega(X_{+,2})    ) +  \eta^{-2} \tPhi_\omega(X_{-,1}) + \tPhi_\omega(X_{-,2}) \\
& = \eta^2 \Theta_c( \tP_\omega  (Y_+)) +  \eta^{-2} \Theta_c( \tP_\omega  (Y_+)),
\end{align*} 
where the second identity follows from Lemma \ref{r.comm}. This completes the proof.
\end{proof}

\def\tA{\mathbb A}
\subsection{The marked annulus} The closed annulus is $[0,1] \times S^1$. The result of removing a point in each boundary component of the closed annulus is denoted by $\tA$, and its corresponding marked surface is called a marked annulus.

Let $Z$ be a stated $\partial \tA$-tangle diagram in $\tA$  as in Figure \ref{fig:annulus}, with arbitrary states on the boundary. Let $Z_+$ (respectively $Z_-$) be the result of the positive (respectively negative) resolution of the only crossing of $Z$.
%At every marked point, the states of  all three $Y, Y_+, Y_-$ are the same.

\begin{figure}[htpb]
    \centering
    \includegraphics[scale=.55]{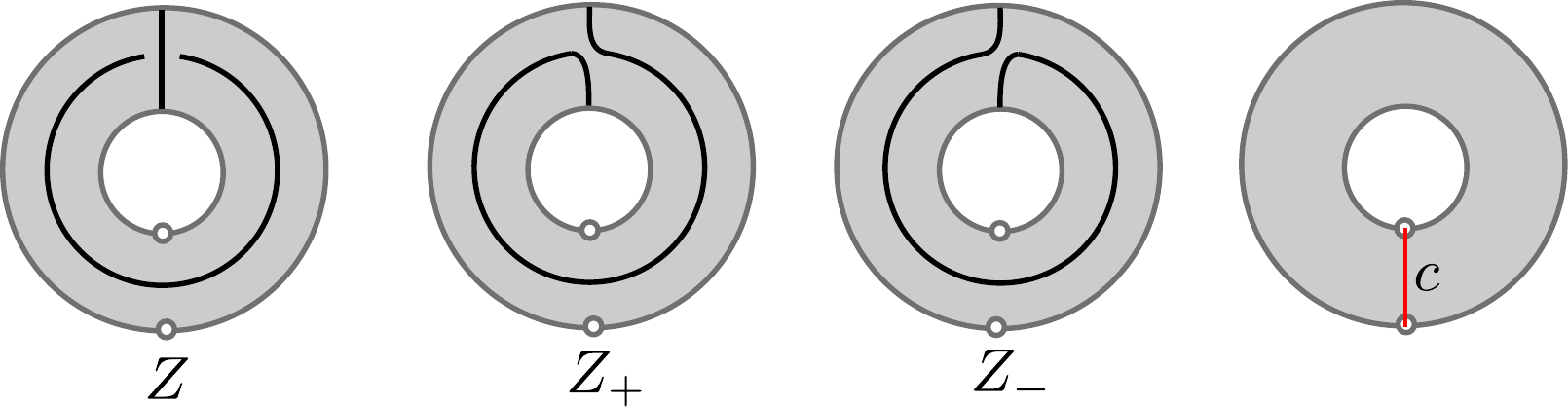}
    \caption{The annulus $\tA$ with stated tangle diagrams $Z, Z_+, Z_-$ on it, and the ideal arc $c$.}
    \label{fig:annulus}
\end{figure}

\begin{lemma}
\label{annulus} In $\Sx(\tA)$ one has
\be 
 \tP_\omega  (Z)   =\eta ^{2}\, \tP_\omega  (Z_+) +\eta ^{-2}\, \tP_\omega  (Z_-).
 \label{eq.s2}
\ee
\end{lemma}
\begin{proof}
Let $c$ be the ideal arc connecting the  ideal points as  seen in Figure~\ref{fig:annulus}.

The result of splitting $\tA$ along $c$ is the ideal square $\SS$. 
 The two lifts of $Z$ are stated $\partial \SS$-tangle diagrams $X_1$ and $X_2$, each of which is a stated copy of $X$ from Figure \ref{fig:square}. Similarly, the two lifts of $Z_+$ are $X_{+,1}$ and $X_{+,2}$, each of which is a stated copy of $X_+$, and the two  lifts of $Z_-$ are $X_{-,1}$ and $X_{-,2}$, each of which is a stated copy of $X_-$.
Now one can repeat the proof of Lemma \ref{pbigon} to obtain the result.
\end{proof}

\def\XY{(U,\cV)}
\subsection{Proof of Theorem \ref{main}}\label{sec.proof}
\begin{proof} (a)  We need to show that $\tP_\omega(\al)= \tP_\omega(\beta)$ if $\al$ and $\beta$ are respectively the left and the right sides of one of the defining relations seen in Figure \ref{fig:stated-relations}.  The support of each relation is a ball, $B$. Note that in each relation, $\al$ is a stated $\cN$-tangle which without loss of generality we may assume does not have any component disjoint from $B$.  Let $U$ be the union of the support ball, $B$, and a small open tubular neighborhood of $\al$, and let $\cV= U \cap \cN$. The functoriality of stated skein modules, along with the embedding of $(U,\cV)$ into $(M,\cN)$, implies that it is enough show that 

\be 
\tP_\omega(\al)= \tP_\omega(\beta)  \quad \text{in} \ \Sx\XY.
\label{eq.main}
\ee

Below we consider the 5 defining relations. First we notice that when $\partial \al=\emptyset$, then $\cV=\emptyset$, and  Identity \eqref{eq.main} was  already proved in \cite{BW1}; see  also \cite{Le2}
for alternative proof utilizing skein theoretic techniques. As such we will consider only the case when $\partial \al \neq \emptyset$.

We begin by looking at the skein relation seen in Figure \ref{fig:skein}.
\begin{figure}[htpb]
    \centering
    \includegraphics[scale=.35]{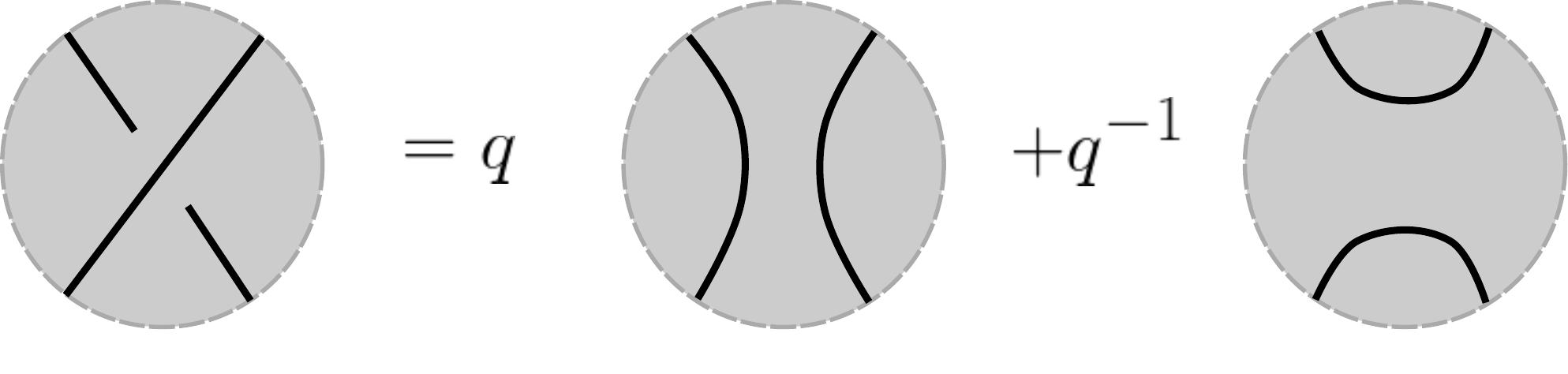}
    \caption{{\em Skein relation (A)}}
    \label{fig:skein}
\end{figure}

There are three cases: (i) $\al$ consists of two arcs, (ii) $\al$ consists of one arc, and (iii) $\al$ consists of one arc and one loop.

(i) Suppose $\al$ consists of two arcs. Then $\XY$ is pseudo-isomorphic  to the thickening of the ideal square $\SS$, and Identity \eqref{eq.main} was proved in Lemma \ref{Square}.

(ii) Suppose  $\al$ consists of a single arc. Then  $\XY$ is pseudo-isomorphic to the thickening of the punctured bigon $\BB$, and Identity \eqref{eq.main} was proved by Lemma \ref{pbigon}. We note that it is necessary to consider the reflection of Lemma \ref{pbigon} under the reflection anti-involution of Proposition \ref{r.reflection}, to cover the case where the crossing is opposite that of the statement of the lemma.

(iii) Suppose  $\al$  consists of one arc and one loop. Then  $\XY$ is pseudo-isomorphic to the thickening of the annulus $\AA$, and Identity \eqref{eq.main} was proved by Lemma \ref{annulus}. Here again we note the need for the image of Lemma \ref{annulus} under the reflection anti-involution to cover the case when the closed loop component crosses over the arc component.

We next look at the trivial knot relation seen in Figure \ref{fig:trivknot}.
\begin{figure}[htpb]
    \centering
    \includegraphics[scale=.35]{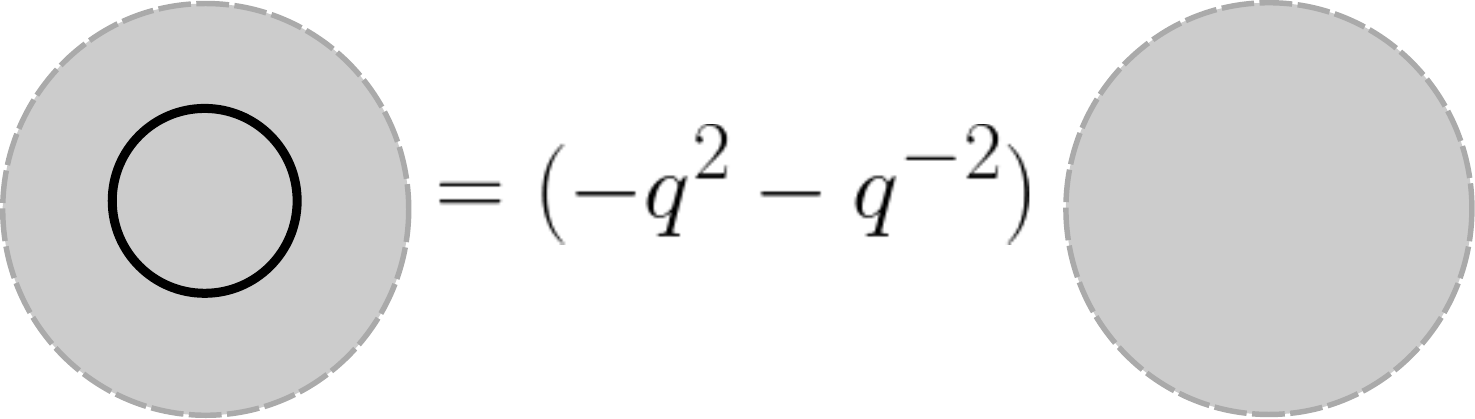}
    \caption{{\em Trivial knot relation (B)}}
    \label{fig:trivknot}
\end{figure}

In this case Identity \eqref{eq.main} becomes the easily verified identity
\[T_N(-\omega ^4-\omega ^{-4})=-\eta^4-\eta^{-4},\]
and as $\partial \al =\emptyset$, this  was  already proved in \cite{BW1,Le2}.

We next look at the trivial arc relation of type $1$ seen in Figure \ref{fig:trivarc1}.
\begin{figure}[htpb]
    \centering
    \includegraphics[scale=.35]{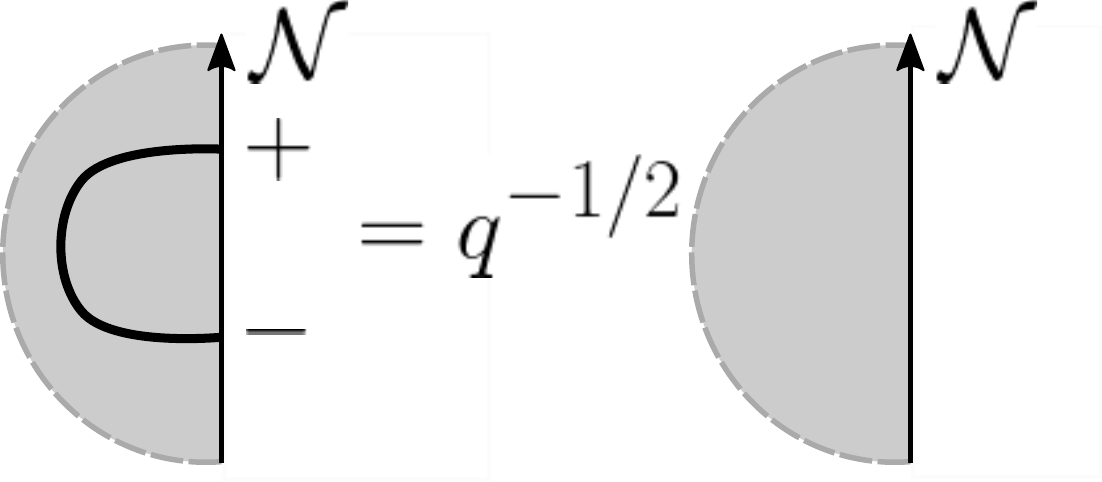}
    \caption{{\em Trivial arc relation of type $1$ (C)}}
    \label{fig:trivarc1}
\end{figure}

We need to prove that  $\tP_\omega(\al)= \eta^{-1}$.
\FIGc{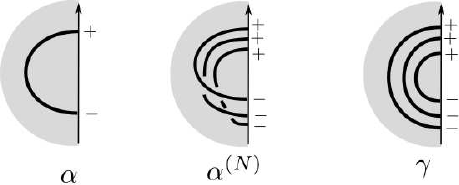}{Diagrams $\al$, $\al^{(N)}$, and $\gamma$,  with $N=3$.}{2.5cm}

Since $\al$ is arc, we have $\tP_\omega(\al) = \al^{(N)}$. We draw a diagram of $\al^{(N)}$ in Figure \ref{fig:trivial1}, which has $N(N-1)/2$ crossings. Using the height exchange move from equation \eqref{eq.reor} with $\nu=-1$ we have that 
$$  \al^{(N)} = (\omega^{2})^{-N(N-1)/2} \gamma,$$
where $\gamma$ is described as in Figure \ref{fig:trivial1}.  Repeated applications of the trivial arc relation of type $1$, shows that $\gamma = \omega^{-N}$.Taken together we have
$$ \tP_\omega(\al) = \al^{(N)} = (\omega^{2})^{-N(N-1)/2} \gamma =  (\omega^{2})^{-N(N-1)/2} \omega^{-N} = \eta^{-1},$$
which proves \eqref{eq.main} for this case.

We next look at the trivial arc relation of type $2$ seen in Figure \ref{fig:trivarc2}.
\begin{figure}[htpb]
    \centering
    \includegraphics[scale=.35]{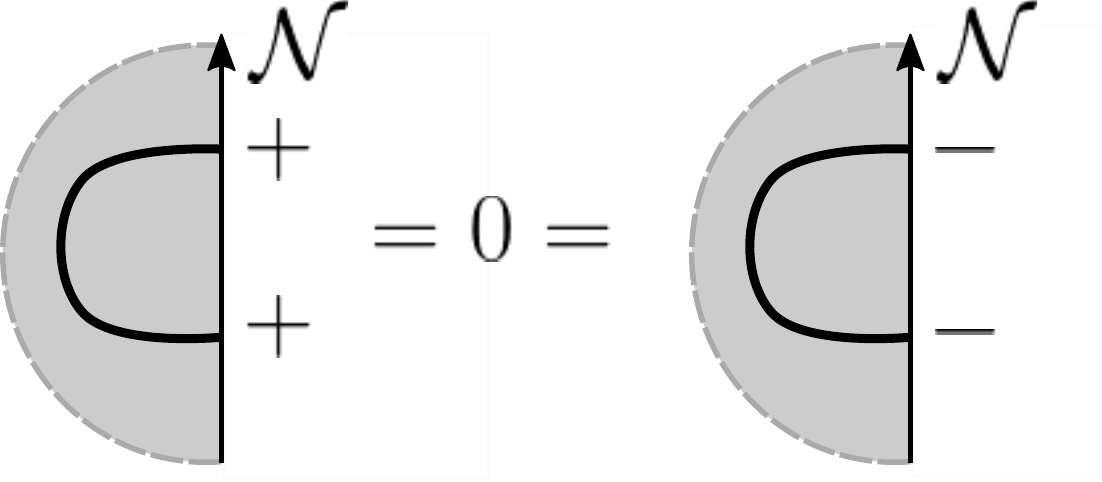}
    \caption{{\em Trivial arc relation of type $2$ (D)}}
    \label{fig:trivarc2}
\end{figure}

We need to prove that  $\tP_\omega(\al)= 0$, where $\al$ is a trivial arc with both states $+$ or both states $-$. 
Like in the previous case, the diagram of $\al^{(N)}$  has $N(N-1)/2$ crossings. Using the height exchange moves of Lemma \ref{r.refl} we can remove the crossings and get  
$  \al^{(N)} = u \gamma,$ where $u$ is a scalar, and $\gamma$ is as in Figure \ref{fig:trivial1}, except all the states are $+$ or all the states are $-$. The trivial arc relation of type $2$ implies that $\gamma=0$. Hence 
$ \tP_\omega(\al) = 0,$
which proves \eqref{eq.main} for this case.

Finally, we look at the state exchange relation seen in Figure \ref{fig:stateheight}.
\begin{figure}[htpb]
    \centering
    \includegraphics[scale=.35]{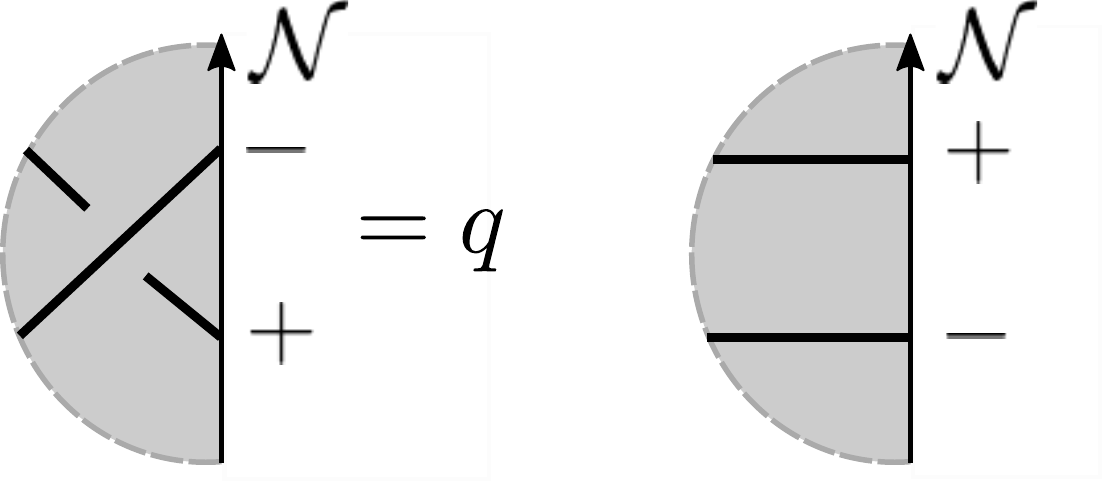}
    \caption{{\em The state exchange relation}}
    \label{fig:stateheight}
\end{figure}

This is the translation of the state exchange relation to punctured bordered surfaces.
By definition, we have the diagrams of $\tP_\omega(\al)$ and $\tP_\omega(\beta)$ as described in Figure \ref{fig:reorduuu}.
\FIGc{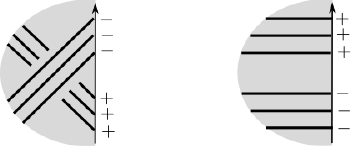}{The left is the diagram of  $\tP_\omega(\al)$ and the right is the diagram of $\tP_\omega(\beta)$, where  $N=3$.}{2.5cm}

The diagram of $\tP_\omega(\al)$ has $N^2$ crossings. Applying the relation seen in Figure \ref{fig:stateheight}, we get
\begin{align*}
\tP_\omega(\al) = (\omega^2)^{N^2}  \tP_\omega(\beta) = \eta^2 \tP_\omega(\beta),
\end{align*}
which proves \eqref{eq.main} for this case, and completes the proof of part (a).

\def\tTheta{\tilde \Theta}
(b) For every stated $\cN$-tangle $\al$ in $M$ we need to prove that
\be 
\Theta_{D}(\Po(\al)) = \Po(\TD(\al)).
\label{eq.m2}
\ee
Let $U$ be the union of a small open tubular neighborhood of $\al\cup D$. Then $U$ is the thickening of a surface which contains $\al$, and utilizing the embedding of $U$ into $M$ along with the functioriality of stated skein modules we only need to prove \eqref{eq.m2} in the case of $M=U$. Thus we may assume that $\al$ is a simple $\pfS$-tangle diagram on a punctured bordered surface $\fS$, with $c$ an ideal arc, and identity \eqref{eq.m2} becomes
\be 
\Theta_{c}(\Po(\al)) = \Po(\Theta_c(\al)).
\label{eq.m3}
\ee

We can assume that $\al$ is transverse to $c$. If $|\al\cap c|=1$, then \eqref{eq.m3} follows exactly from Lemma \ref{r.comm}.

\def\tal{\check \alpha}

Choose a finite subset $\cV\subset c$ such that $c\setminus \cV = \sqcup_{i=1}^k c_i$ and each $c_i$ intersects $\al$ at exactly one point.
 Let $\tfS = \fS \setminus \cV$.  We are in the situation described in Lemma \ref{r.cutk}. Thus, let $\fS'$ be the result of splitting $\fS$ along $c$, and $\tfS'$ be the result of splitting $\tfS$ along all $c_i$.
Choose an orientation for $c$, as in the discussion in Subsection \ref{sec.spl2}, allowing for well defined induced maps on the corresponding skein modules.  
Let $\tal\in \cS(\tfS)$ be the element defined by the same $\al$, but considered as an element of $\cS(\tfS)$. Then clearly 
\be 
\iota_*(\tal) = \al, \label{eq.p0}
\ee
where $\iota_*: \cS(\tfS) \to \cS(\fS)$ is the induced homomorphism.
Since $\tal$ intersects each $c_i$ in exactly one point, we have $\Theta_{c_i}(\Po(\tal)) = \Po(\Theta_{c_i}(\tal))$ again by Lemma \ref{r.comm}. As 
$\Theta_c: \cS(\tfS) \to \cS(\tfS')$ is the composition of the $\Theta_{c_i}$, we have
\be 
\Theta_{c} (\Po(\tal)) = \Po(\Theta_c(\tal)).
\label{eq.m4}
\ee
From the definition it is clear that $\iota_*$ commutes with $\Po$. By Lemma \ref{r.cutk} we have that
$\iota_*$ commutes with $\Theta_c$. Hence applying $\iota_*$ to \eqref{eq.m4} and using \eqref{eq.p0}, we get
$$
\Theta_{c}(\Po(\al)) = \Po(\Theta_c(\al)),
$$
which is Identity \eqref{eq.m3} as desired.
\end{proof}

\subsection{Specializing to surfaces}
\label{specializesurf}

Now we see how our main theorem can be specialized to the case of stated skein algebras of surfaces to give the following corollary:
\begin{corollary} \label{r.surfacePhi}
Suppose
 $\fS$ is a punctured bordered surface and $\omega$ is a complex root of unity. Let $N=\ord(\omega^8)$ and $\eta=\omega^{N^2}$.  
 
(a)  There exists a unique $\BC$-linear algebra homomorphism $\Phi_\omega : \cS_\eta(\fS) \to \cS_\omega (\fS)$ determined on generators by if $\alpha$ is a $\pfS-$arc then $\Phi_\omega(\alpha)=\alpha^{(N)}$, and if $\alpha$ is a $\pfS-$knot then $\Phi_\omega(\alpha)=T_N(\alpha).$

(b) Additionally, $\Phi_\omega$ is injective and we have that $\Phi_\omega $ commutes with the splitting homomorphism, meaning if $\hat{\fS}$ is the result of splitting $\fS$ along and ideal arc $\gamma$, then
 the following diagram commutes:
\be 
\begin{tikzcd}
\mathscr{S}_\eta(\fS)\arrow[r,"\Theta_{\gamma}"]\arrow[d,"\Phi_\omega "] & \mathscr{S}_\eta(\hat{\fS})\arrow[d,"\Phi_\omega "] \\
\mathscr{S}_\omega (\fS)\arrow[r,"\Theta_{\gamma}"]  & \mathscr{S}_\omega (\hat{\fS})\\
\end{tikzcd}
\label{eq.diasurf}
\ee
\end{corollary}
\begin{proof}
Aside from the injectivity, this is a direct specialization of Theorem \ref{main}.  We will show $\Phi_\omega$ is injective separately for triangulable surfaces and surfaces which are not triangulable, as described in Subsection \ref{IdealTriang}.

For triangulable surfaces, we will see injectivity of $\Phi_\omega$ as an application of the compatibility of $\Phi_\omega$ with the splitting homomorphism.  Let $\Delta$ be an ideal triangulation of $\fS$, then combining the discussion of Subsection \ref{IdealTriang} with Theorem \ref{surfacesplit} we have the following injective map coming from the composition of splitting homomorphims:
\[\Theta_\Delta:\cS(\fS)\hookrightarrow \bigotimes_{\mathfrak{T}\in \Delta} \cS(\mathfrak{T}), \]
where $\mathfrak{T}$ is the ideal triangle.
We see that $\Theta_\Delta$ fits into the following commutative diagram by repeated applications of the compatibility of $\Phi_\omega$ with the splitting homomorphisms:
\be
\label{triangdecomp}
\begin{tikzcd}
\mathscr{S}_\eta(\fS)\arrow[r,"\Phi_\omega"]\arrow[d,hook,"\Theta_\Delta"'] & \mathscr{S}_\omega(\fS)\arrow[d,hook,"\Theta_\Delta "] \\
\bigotimes\cS_\eta(\mathfrak{T})\arrow[r,"\bigotimes \Phi_\omega"']  & \bigotimes \cS_\omega(\mathfrak{T}).\\
\end{tikzcd}
\ee
Then noting that the tensor product of injective maps between vector spaces is injective, we see the injectivity of $\Phi_\omega$ in general to be reduced to the injectivity of $\Phi_\omega$ in the specific case of the ideal triangle. Following Theorem \ref{surfbasis} we see that a basis of $\cS(\mathfrak{T})$ is given by crossingless diagrams of arcs which meet distinct boundary edges in $\mathfrak{T}$, where on a given boundary edge there is never a negative state after a positive with respect to the orientation of that boundary edge induced by the orientation of the ideal triangle.  Then note for an arc $\alpha$, we have within a small neighborhood of $\alpha$ that $\alpha^{(N)}$ is the $N$-strand half twist equally stated on each boundary edge.  Using that all the states on a given boundary edge are equal we have that $\frac{N(N-1)}{2}$ height exchange moves, as seen in Lemma \ref{r.refl}, can be applied to remove all of the crossings at the cost of a nonzero scalar multiple. Additionally, noting that the framed power will not introduce a negative state after any positive states we see that $\Phi_\omega$ sends each basis diagram to a non-zero multiple of a basis diagram, meaning $\Phi_\omega$ is injective.

Now we proceed to surfaces which are not triangulable.  Let $\fS$ be a punctured bordered surface, where there is a compact oriented surface $\bfS$ and a finite set $\cV\subset \bfS$ such that $\fS\cong \bfS\setminus \cV$.  The first case is when $|\cV|=0$.  This case is well known to experts, and is stated in Theorem $3.8$ of \cite{FKL1}, but we include a proof here for completeness. 

Following \cite{FKL2}, we look to define a lead term map for $\cS(\fS)$ determined by Dehn-Thurston coordinates for simple closed curves on $\fS$.  Let $B$ be the basis of isotopy classes of simple closed curves of $\cS(\fS)$.  

If the genus of $\fS$ is greater than one and $|\cV|=0$ h, then following Subsection $3.6$ of \cite{FKL2}, we have coordinate datum determined by Dehn-Thurston coordinates of simple closed curves.  Fix an ordered pants decomposition of $\fS$, $P=\{P_1,...,P_{3g-3}\}$, meaning a collection of simple closed curves that cut $\fS$ into $2g-2$ pairs of pants.  Fix a dual graph to $P$, meaning an embedded trivalent graph with $3g-3$ edges such that each $P_i$ is met transversely by exactly one edge and each pair of pants in $\fS\backslash P$ contains exactly one trivalent vertex.  Then define
\[\nu:B\rightarrow \mathbb{Z}^{6g-6}\]
by $(\nu(\alpha))_i=I(\alpha,P_i)$ for $1\leq i\leq 3g-3$, and $(\nu_2(\alpha))_j$ for $3g-3<j\leq 6g-6$ is the twist coordinate of $\alpha$ at $P_i$ with respect to the edge of the dual graph that meets $P_i$.  

Then $\nu$ is an injective map with an image that is a submonoid of $\mathbb{Z}^{6g-6}$, which refer to as the monoid of curves in $\fS$.  Additionally, let $\mu$ denote the inverse of $\nu$ defined on the monoid of curves.  

Now let $x=\sum_{b\in S}c_b b\neq 0$ be a linear combination with $b\in S\subseteq B$ where $c_b\neq 0$ for $b\in S$.  Then we proceed by analyzing two cases separately.  First we define $S_{max}$ to be the subset of $S$ of basis vectors, $b$, such that $\sum_{i=1}^{3g-3} (\nu(b))_i=k$ and $k$ is maximal among all $b\in S$.  

First we assume that $k>0$.  In this case we define a lead term map, $lt(x)=\sum_{b\in S_{max}}c_b b$.  We look to show that $\Phi_\omega(x)\neq 0$, by showing that $lt(\Phi_\omega(x))\neq 0$.  We begin by utilizing $\mu$, to describe each $b\in S_{max}$ as some $\mu(\vec{n_b})$ where $\{\vec{n_b}\}$ are distinct vectors in the monoid of curves.  Now noting that $b^r=\mu(r\vec{n}_b)$, we see that \[lt(T_N(b))=\mu(N\vec{n}_b).\]  
Giving
\[lt(\Phi_\omega(x))=\sum_{b\in S_{max}} c_b\mu(N\vec{n}_b),\]
which is nonzero as $\sum_{i\in S_{max}} c_b\mu(\vec{n}_b)$ is nonzero.  Thus $\Phi_\omega$ is injective in this case.

Now we need to consider the case when the value $k$ used in defining $S_{max}$ is zero. This corresponds to basis vectors which never intersect any of the curves in the ordered pants decomposition, meaning $x\in \mathbb{C}[P_1,...,P_{3g-3}]$ the polynomial algebra in commuting variables.  In this case, $\Phi_\omega(x)=T_N(x)$ and an additional lead term argument can be done using the ordinary degree of polynomials, to see $T_N(x)\neq 0$, meaning $\Phi_\omega$ is injective. 

The final case case of surfaces with $|\cV|=0$ is the torus, which was explicitly described by Frohman and Gelca in \cite{FG}.  Now of the remaining surfaces which are not triangulable we have that the skein algebras of the sphere with one or fewer marked points and the disk with one or fewer boundary marked point are just the ground field, and so $\Phi_\omega$ is injective immediately.  For the sphere with two marked points we can use compatibility with the splitting homomorphism to reduce the case of ideal bigon.  This last case of the ideal bigon follows from a simplification of the argument given for the ideal triangle.  In particular, basis diagrams are sent directly to basis diagrams as the orientation of the two boundary edges can be chosen so the framed power of any arc is crossingless.
\end{proof}

Even in the case of surfaces, the three-dimensional nature of Theorem \ref{main} can provide additional insight beyond Corollary \ref{r.surfacePhi}. For example,  Corollary \ref{r.surfacePhi} does not explain why the image under $\Phi_\omega$ of the curve $\al$ in Figure \ref{fig:SurfaceExtra} is given by the threading of $\al$ by $T_N$. To calculate $\Phi_\omega(\al)$ using Corollary \ref{r.surfacePhi} one first resolves the crossing of $\al$ to get a linear sum of two simple diagrams, then apply the formula in Corollary \ref{r.surfacePhi} to each diagram. It is far from trivial to show that the result is the same as the threading of $\al$ by $T_N$.

\begin{figure}[htpb]
    \centering
    \includegraphics[scale=.3]{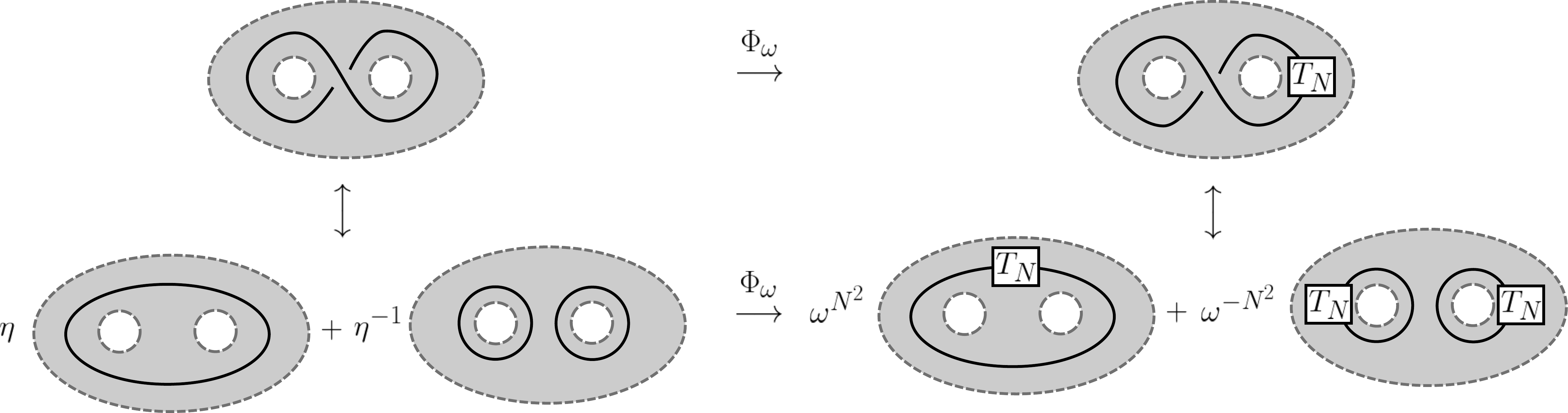}
    \caption{The right hand side of this diagram is an example of a nontrivial equality for the stated skein algebras of surfaces arising from extending the Chebyshev-Frobenius homomorphism to $3-$manifolds. Note that this example is for the skein algebra of unmarked surfaces.}
    \label{fig:SurfaceExtra}
\end{figure}

\begin{remark}
\label{PrevKnown}
(a)
When $\partial \fS=\emptyset$, the existence of a Chebyshev homomorphism for the thickening of $\fS$
was first proved in \cite{BW1} using the quantum trace map. 
 This result was recovered in \cite{Le2} utilizing skein theoretic techniques.

(b) For the case of the positive submodule of the stated skein module for marked $3-$manifolds Theorem \ref{main} was proven in \cite{LP}.

(c) Corollary \ref{r.surfacePhi}, in the case that the order of $\omega$ is odd only, is also proved in \cite{KQ}, using independent arguments.  

(d) A discussion of the Frobenius homomorphism for the skein algebra of surfaces with no marked points, and its relation to quantum Hamiltonian reduction can be found in \cite{GJS}.   Additionally, in the case of surfaces, quantum moduli algebras provide an alternative viewpoint to stated skein algebras, from this perspective a Frobenius homomorphism has been developed by \cite{BR}.
\end{remark}

\begin{remark}
The injectivity of the Chebyshev-Frobenius homomorphism in the case of surfaces is in contrast to the case of stated skein modules of $3-$manifolds in general.  Even in the case of closed $3-$manifolds, for example  $S^2\times S^1\# S^2\times S^1$,  the kernel of the Chebyshev homomorphism may have non-trivial kernel.
This will be explored in upcoming work of the second author and Constantino \cite{CL2}.
\end{remark}

\subsection{Transparency and the Chebyshev-Frobenius image}We will show that the image of the Chebyshev-Frobenius map is always {\em transparent} or {\em skew-transparent}.

\begin{theorem}\label{r.transparent}
Suppose $\MN$ is a marked 3-manifold, and $\omega $ is a root of unity, with $N=\text{ord}(\omega^8)$.  Then let $N'=\text{ord}(\omega^4)$ and $\mu=(-1)^{N'+1}$.  If $\alpha$ is a $\cN-$tangle, then $\Phi_\omega(\alpha)$ is $\mu-$transparent in the sense that the identity in Figure \ref{fig:ImageTrans} holds. 

\begin{figure}[htpb]
    \centering
    \includegraphics[scale=.45]{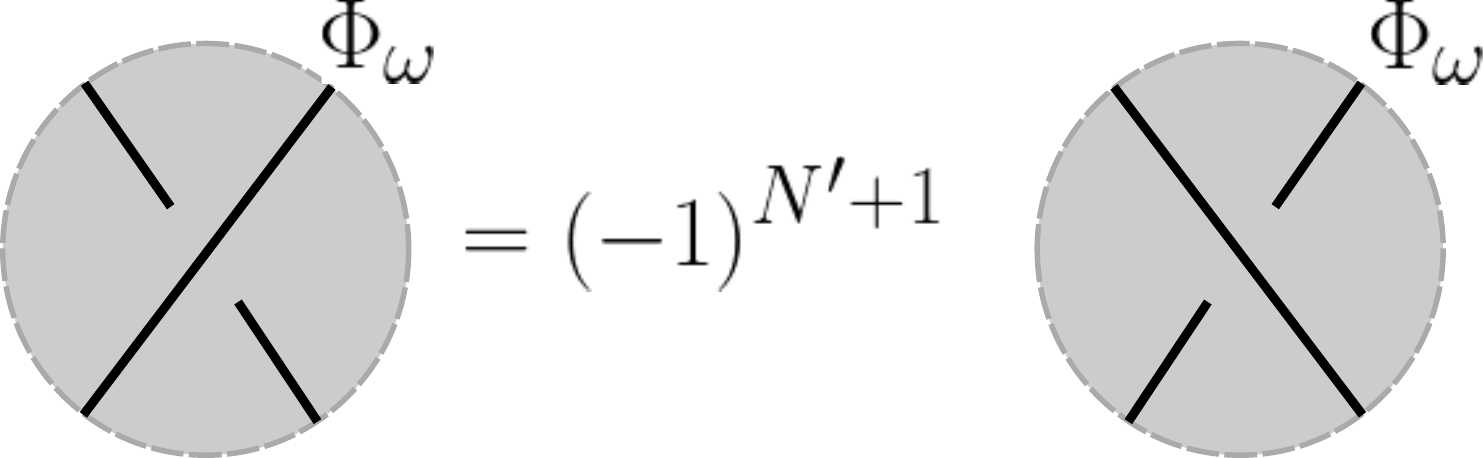}
    \caption{A depiction of $\mu-$transparency.}
    \label{fig:ImageTrans}
\end{figure}

% Then the image of $\Phi_\omega$ is $\mu-$transparent as described in Figure \ref{fig:ImageTrans}.
 In words this means the following. Suppose $T_1$ is a stated $\cN$-tangle disjoint from  $\al$, and $T_2$ is $\cN$-isotopic to $T_1$ through an isotopy which is disjoint from $\al$ in all but one instance, where the intersection is transverse and in a single point. Then we have
\[\Phi_\omega(\alpha)\cup T_1=\mu(\Phi_\omega(\alpha)\cup T_2) \quad \in \mathscr{S}(M,\cN).\]

\end{theorem}

\begin{remark} We say transparency is $\mu$-transparency when $\mu=1$ and skew-transparency is $\mu$-transparency when $\mu=-1$.

\end{remark}

\begin{proof}
 We can assume that $\al$ has one component. If $\al$ is a an $\cN$-knot the $\mu$-transparency of $\Phi_\omega(\al)=T_N(\al)$ was proved in \cite{Le2,LP}.

Suppose $\al$ is an $\cN$-arc. Potentially up to isotopy of $T_1$ and $T_2$ we work with the local picture seen in Figure \ref{fig:transinduct}.  Then following a version of the consecutive resolutions argument seen in the proof of Lemma \ref{Square}, we see that only the totally positive and totally negative resolutions result in nonzero diagrams.

\begin{figure}[htpb]
    \centering
    \includegraphics[scale=.35]{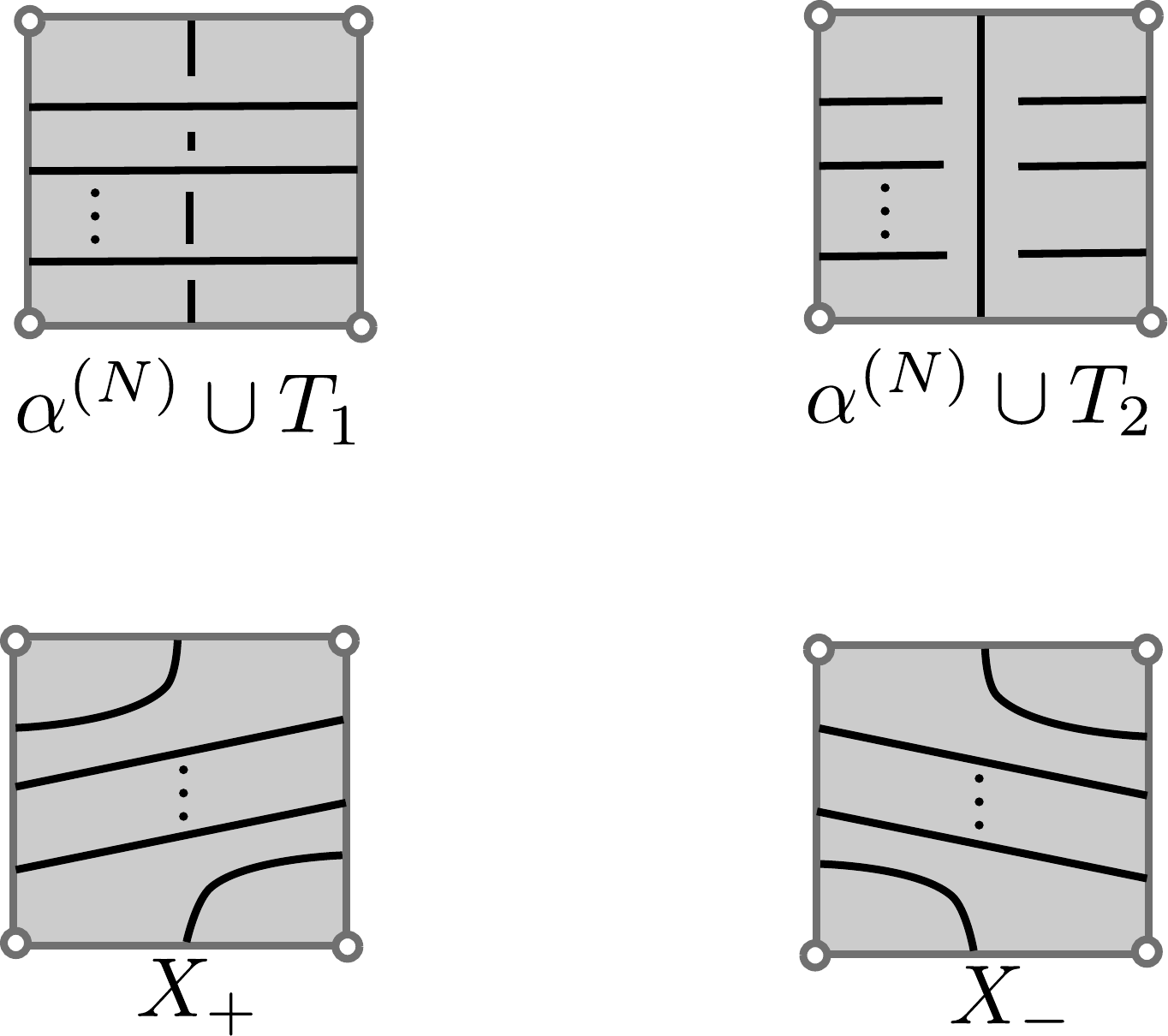}
    \caption{A local picture depicting the crossing of $\alpha^{(N)}$, $T_1$,$T_2$ and the resolutions of the corresponding crossings.}
    \label{fig:transinduct}
\end{figure}
 It follows that
 \be 
 \al^{(N)} \cup T_1 = \omega^{2N} X_+ + \omega^{-2N} X_-.
 \ee
Similarly the image under the reflection anti-involution gives, 
\be 
 \al^{(N)} \cup T_2 = \omega^{-2N} X_+ + \omega^{2N} X_-.
 \ee
Which tells us

$$\al^{(N)} \cup T_1= \omega^{4N} ( \al^{(N)} \cup T_2) =\mu ( \al^{(N)} \cup T_2),$$
as desired.

\end{proof}

\begin{corollary}
(a)  If $\text{ord}(\omega^4)$ is odd we have $\Phi_\omega(\alpha)$ is transparent, and if $\text{ord}(\omega^4)$ is even we have $\Phi_\omega(\alpha)$ is skew-transparent.

(b) If $\alpha$ is a $\cN-$arc we have $\alpha^{(2N)}$ is transparent, and if $\alpha$ is a $\cN-$knot we have $T_{2N}(\alpha)$ is transparent.
\end{corollary}
\begin{proof} (a) follows right away from Theorem 
We have 
\[\Phi_\omega(\alpha)\cup T_1 = (-1)^{N'+1}\Phi_\omega(\alpha)\cup T_2,\]
In particular, for any one component tangle $\alpha$ we have
\[\Phi_\omega(\alpha^{(2)})\cup T_1=\Phi_{\omega}(\alpha^{(2)})\cup T_2,\]
meaning $\Phi_\omega(\alpha^{(2)})$ is transparent.  In particular, if $\alpha$ is a $\cN-$arc we have have $\alpha^{(2N)}$ is transparent and if $\alpha$ is a $\cN-$knot we have $T_N(\alpha^{2})$ is transparent.  Additionally, we have $T_{2N}(\alpha)=T_N(\alpha^2)+2$, and so $T_{2N}(\alpha)$ is also transparent.
\end{proof}

\subsection{Centers of skein algebras of surfaces}
We have the following corollary in the case of surfaces.

\begin{corollary}
\label{SurfaceCenter}
Suppose $\fS$ is a punctured bordered surface and $\omega$ is a root of unity with $N=\ord(\omega^8)$.  

(a) If $\alpha$ is a $\pfS$-knot then $T_{2N}(\alpha)$ is central, and if $\alpha$ is a stated $\pfS-$arc then  $\alpha^{(4N)}$ is central.

(b)  More precisely, let $N'=\ord(\omega^4)$ and $N''=\ord(\omega^2)$, then we have that if $\alpha$ is a $\pfS-$knot then $T_{N'}(\alpha)$ is central, and if $\alpha$ is a stated $\pfS-$arc then  $\alpha^{(N'')}$ is central.  
\end{corollary}
\begin{proof}

We first look to prove $(b)$.  If $\alpha$ is a $\cN-$knot then centrality follows immediately from transparency.  So our result follows from the previous discussion on transparency.  If $\alpha$ is a $\pfS-$arc, then transparency is not enough enough to imply centrality, and in particular commuting with other $\pfS$-arcs having both endpoints on the same boundary component requires additional considerations to reorder the end points as needed. This can be accomplished through height exchange moves, where it may be necessary to use transparency to change a crossing before applying the allowed height exchange moves.  In particular, we have that these moves applied to the boundary of $\alpha^{(k)}$ will give a factor of $\omega^{2k}$.  This allows us to extend the previous results on transparency to centrality by utilizing that $\omega^{2N''}=1$ by definition.

Then part $(a)$ follows from noting that $2N$ is a multiple of $N'$ and $4N$ is a multiple of $N''$.
\end{proof}

\begin{remark}
In general, the ordinary power $\alpha^{k}$ is not central.  Moreover as seen below in Proposition \ref{nevercentral}, there exists $\alpha$ such that for all $k\geq 1$ and for all  roots of unity $\omega$ of order greater than 8, we have that $\alpha^k$ is not central.  
\end{remark}

\subsection{Framed powers compared to algebra powers}
\label{FramedVsAlg}
\begin{proposition}
\label{nevercentral}
If there is an arc on a punctured bordered surface $\fS$  such that both endpoints of the arc meet the same boundary edge, but the arc does not bound a disk, then there exists a stated arc $\alpha$ such that $\alpha^m$ is only central when $q^4=1$.
\end{proposition}
\begin{proof}  Utilizing functoriality this reduces to a computation in the punctured monogon.  This set up is seen in Figure \ref{fig:monogonarc}.

\begin{figure}[htpb]
    \centering
    \includegraphics[scale=.35]{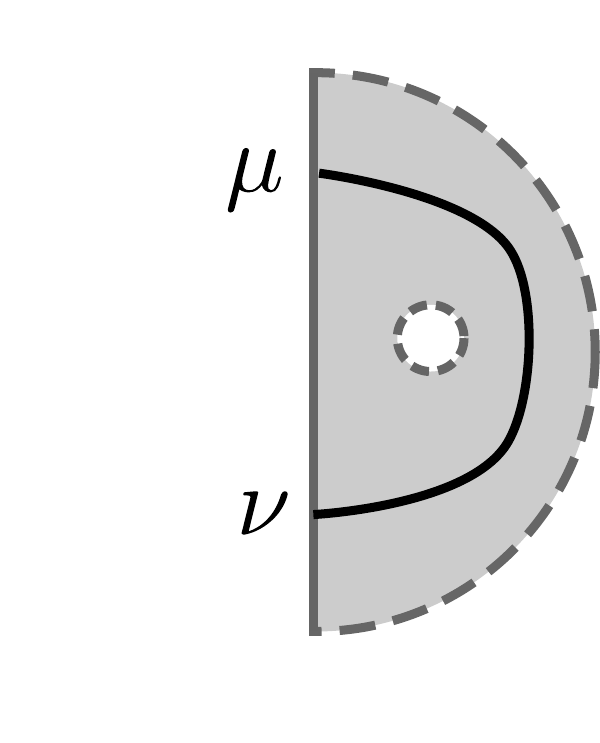}
    \caption{The image of an arc in a punctured bordered surface which does not bound a disk and both end points meet the same boundary edge in the punctured monogon.}
    \label{fig:monogonarc}
\end{figure}

Let $w$ be the arc in Figure \ref{fig:monogonarc} be stated with $\mu=\nu=+$, let $x$ be the arc in Figure \ref{fig:monogonarc} stated with $\mu=-$ and $\nu=+$ $-$, let $y$ be the arc in Figure \ref{fig:monogonarc} stated with $\mu=+$ and $\nu=-$ multiplied by $q^2$, and let $z$ denote the loop encircling the puncture multiplied by $q^{-1/2}$.  Then as seen in Figure \ref{fig:monogonproduct}
\[wy=q^4yw.\]
\begin{figure}[htpb]
    \centering
    \includegraphics[scale=.35]{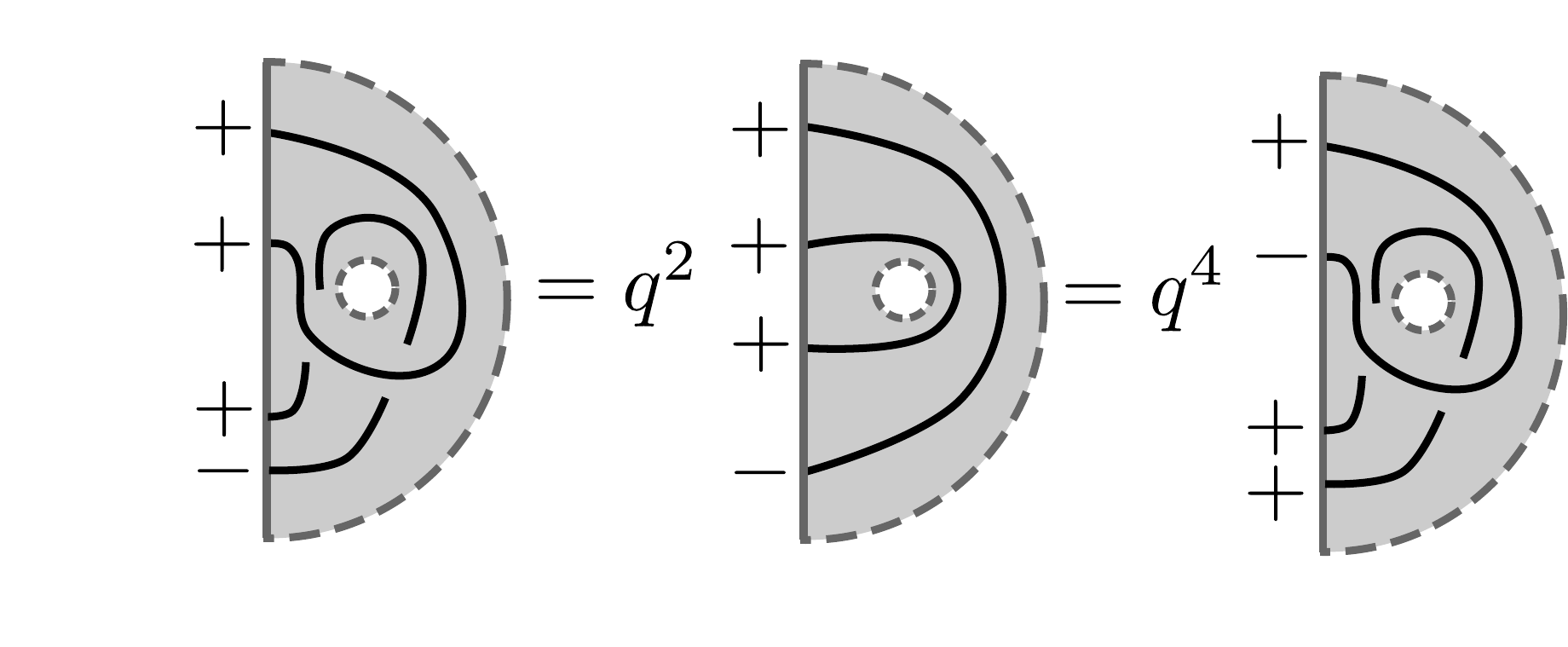}
    \caption{Using height exchange moves to see that $w$ and $y$ $q-$commute.}
    \label{fig:monogonproduct}
\end{figure}
Now utilizing the state exchange relation we have $x=y+z$, and so
\[x^m=(y+z)^m,\]
which as $z$ is a central element we may use the classical binomial theorem and 
\[x^m=\sum_{k=0}^m\binom{m}{k}y^{k}z^{m-k},\]
where we note that $\binom{m}{k}\neq 0$.

Then 
\[wx^m-x^mw=\sum_{j=0}^{m}\binom{m}{k}(1-q^{4k})y^kz^{m-k.}w\]
Now we observe that $\{y^{k}z^{N-k}\}$ form a linearly independent set in the stated skein algebra of the punctured monogon.  Using height exchange relations, we see that $y^k$ is a a power of $q$ multiple of the simple diagram of non-crossing arcs connecting $k$ positive and $k$ negative states. Then as $z$ is central we have that $y^{N-k}z^k$ is  a nonzero multiple of a simple diagram, which form a basis by Theorem \ref{surfbasis}, meaning they are linearly independent.  Now as $\{y^kz^{m-k}\}$ forms a linearly independent set, we additionally have $\{y^kz^{m-k}w\}$ forms a linearly independent set.  Then we have $a^kw=wa^k$ if and only if $1-q^{4k}=0$ for all $k$.

\end{proof}

\def\SP{\Sigma_{\cP}}

\section{A quantum torus driven point of view}
\label{QTorusSect}

This section is devoted to giving an interpretation of the Chebyshev-Frobenius homomorphism as it relates to a quantum trace map, meaning embeddings of the stated skein algebra of surfaces into quantum tori.

In this section let $\SM$ be a marked surface satisfying the following properties: 
\begin{itemize}
\item[($\star$)]
$\Sigma$ is the result of removing a finite set $\cV=\{v_1, \dots , v_m\}$ of interior points from  an oriented compact connected surface $\bS$ with non-empty boundary, and $\cP\subset \pS$ is a finite set such that every connected component of $\pS$ intersects non-trivially $\cP$. Additionally, $\SM$ is not a monogon or a bigon.
\end{itemize}

\no{

\subsection{Quantum tori and Frobenius homomorphisms}
Let $P$ be an anti-symmetric matrix with integer entries, $\cR$ is a commutative domain with an invertible element $q$, we define a quantum torus over $\cR[z_1,...,z_k]$ by
\[\mathbb{T}(P;\omega, \cR[z_1,...,k_k]):=\cR[z_1,...,z_k]\langle \{x_i\}_{i=1}^{|P|}|x_i x_j=\omega^{P_{ij}}x_j x_i\rangle.\]
We will drop $\cR[z_1,...,z_k]$ if it clear from the text.
For every positive integer $N$  there is an $\cR$-algebra homomorphism 
\[F_N:\mathbb{T}(P;\omega^{N^2}, \cR)\rightarrow \mathbb{T}(P;\omega, \cR)\]
given by 
\[F_N(x_i)=x_i^N,\]

and additionally, 
\[F_N(z_i)= T_N(v_i).\]  

In the remainder of this section all quantum tori will be over $\cR[v_1, \dots, v_m]$.
}

\subsection{Quasitriangulations and their associated quantum tori} In \cite{LY} it is shown that the skein algebra $\cS\SM$ can be embedded into a quantum torus which is an algebra with notably simple algebraic structure. Let us briefly recall the embedding. The idea is to find a large enough set of $q$-commuting elements in $\cS\SM$ and try to embed $\cS\SM$ into the quantum torus generated by these $q$-commuting elements.

\def\Ed{\cE_\partial}
\def\hEd{\widehat{\cE_\partial}}
\def\bE{\bar {\cE}}

%\subsection{Embedding of $\cS\SM$ into quantum torus} 

A {\em $\cP$-arc} in $\Sigma$ is an immersion $\al:[0,1] \to \Sigma$ such that $\al(0), \al(1) \in \cP$ and the restriction of $\al$ onto $(0,1)$ is an embedding into $\Sigma\setminus \cP$. Two $\cP$-arcs are {\em $\cP$-disjoint} if they are disjoint in $\Sigma\setminus \cP$.
Two $\cP$-arcs are {\em $\cP$-isotopic}  if they are isotopic in the class of $\cP$-arc. A $\cP$-arc is {\em boundary} if it is $\cP$-isotopic to a $\cP$-arc which is in the boundary  $\pS$. A $\cP$-arc is {\em trivial} if it bounds a disk in $\Sigma$. 

 A {\em $\cP$-quasitriangulation $\cE$ of $\Sigma$} is a maximal collection of non-trivial, pairwise $\cP$-disjoint, and pairwise non $\cP$-isotopic $\cP$-arcs. Fix a $\cP$-quasitriangulation $\cE$ of $\Sigma$. Let $\Ed$ be the subset of $\cE$ consisting of all boundary $\cP$-arcs.

% Every boundary $\cP$-arc is $\cP$-isotopic to an element of $\cE$, and after a $\cP$-isotopy, we will assume that the $\cE$ contains the set of all boundary $\cP$-arcs, which will be denoted by $\Ed$.

Consider a copy  $\hEd=\{ \hat e \mid e \in \Ed\}$ of $\Ed$, and let $\bE= \cE \sqcup \hEd$. We will consider $\bE$ as a subset of $\cS\SM$ by identifying $e\in \cE$ with the element $X_e\in \cS\SM$ as follows.  First, if $e$ is a $\cP$-arc having identical endpoints, which is a point $p\in \cP$, then let $e'\subset \Sigma\times (-1,1)$ be the same $e$ with the right incident half edge slightly raised so that it is higher than the left incident half edge, see Figure \ref{fig:legraise}.

\begin{figure}[htpb]
    \centering
    \includegraphics[scale=.45]{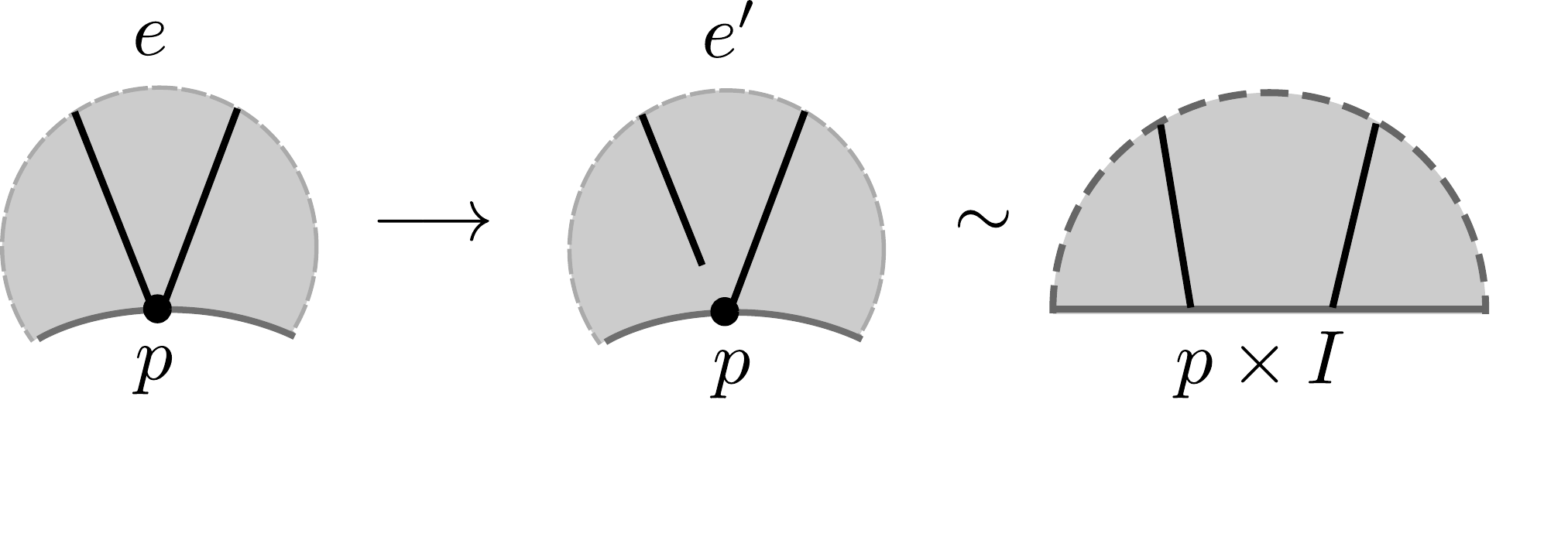}
    \caption{The assignment of $e'$ to a $\mathcal{P}$-arc with identical endpoints $p\in\mathcal{P}$.}
    \label{fig:legraise}
\end{figure}

Now define $X_e$ by:

\begin{itemize}
    \item If $e\in \cE$ has distinct endpoints, then $X_e\in \cS\SM$ is the arc $e$ with positive states at both endpoints. If $e\in \cE$ has identical  endpoints, then $X_e= \omega e'$, where both end points of $e'$ have positive states.
 
    \item If $\hat e\in \hEd$, where $e\in \Ed$ is a boundary $\mathcal{P}$-arc having distinct endpoints, then $X_{\hat e}\in \cS\SM$ is the arc $e$  with one positive state followed by a negative state where the ordering is given with respect to the orientation of the boundary.
    
    \item If $\hat e\in \hEd$, where $e\in \Ed$ is a boundary $\mathcal{P}$-arc having identical  endpoints, then $X_e= \omega^{-1} e'$, where the higher endpoint of $e'$ has  a negative state and the lower end point has a positive state.
\end{itemize}
The factors of $\omega$ and $\omega^{-1}$ are introduced so that $X_e$ is invariant under the reflection involution. For simplicity of notation, we identify $e\in \bE$ with $X_e\in \cS\SM$. The height exchange moves show that any two $e,e'\in \bE$ are $q$-commuting: For $a,b\in \bE$ have
\be 
ab = q^{P(a,b)} ba,
\label{eq.rel1}
\ee
where $P= P_{\bE}:\bE \times \bE\to \BZ$ is the 
\begin{itemize}
\item the number of occurrences of $b$ meeting a vertex counterclockwise to $a$ minus the number of occurrences of $b$ meeting a vertex clockwise to $a$, if $a,b\in \cE$, as seen in Figure \ref{fig:Pmatrix}.
\begin{figure}[htpb]
    \centering
    \includegraphics[scale=.45]{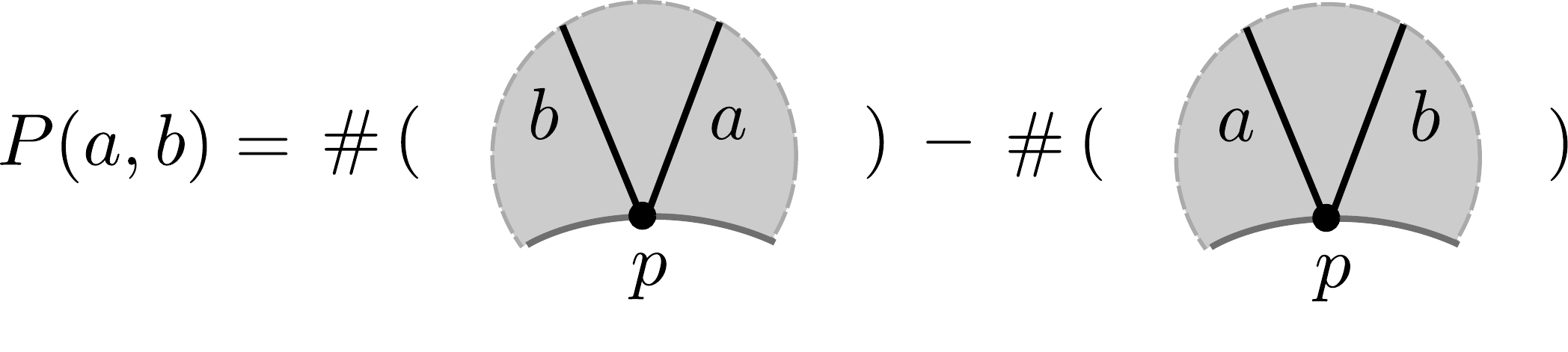}
    \caption{A count that determines how $a$ and $b$ $q$-commute through the application of height exchange moves.}
    \label{fig:Pmatrix}
\end{figure}

\item   the number of occurrences of $b$ meeting a vertex counterclockwise to $a$ added to the number of occurrences of $b$ meeting a vertex clockwise to $a$, if $a\in \hat{\cE}$ and $b \in \cE$, as seen in Figure \ref{fig:Pmatrixhat}.
\begin{figure}[htpb]
    \centering
    \includegraphics[scale=.45]{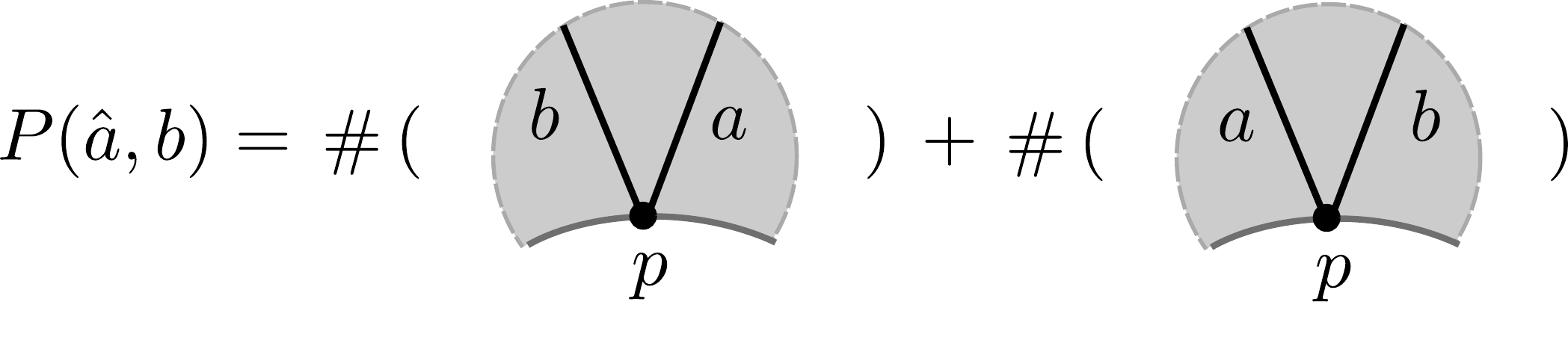}
    \caption{A count that determines how $\hat{a}$ and $b$ $q$-commute through the application of height exchange moves.}
    \label{fig:Pmatrixhat}
\end{figure}
\item $P(f,e)$ if $a=\hat e$ and $b=\hat f$.
\end{itemize}

Identify each interior point $v\in \cV$ with the element of $\cS\SM$ represented by  a small loop surrounding the puncture $v$. Let  $\cR[\cV]$ be the $\cR$-algebra of polynomials in the variables $v_i\in \cV$ with coefficients in $\cR$.   Then $\cR[\cV]$ is a subalgebra of the center of $\cS\SM$.

\no{Let $\cV=\{v_1,\dots, v_k\}$ be the set of interior punctures. Identify each $v_i$ with the element of $\cS\SM$ represented by  a small loop surrounding the puncture $v$. Let  $\cR[\cV]$ be the $\cR$-algebra of polynomials in the variables $v_i\in \cV$ with coefficients in $\cR$. As  
for $\bn=(n_1,\dots, n_k)\in \BN^k$ let $v^\bn :=\prod (v_i)^{n_i}$. Then $\{ v^\bn \mid  \bn \in \BN^k\} $ is an $\cR$-basis of $\cR[\cV]$. Clearly every element of $\cR[\cV]$ is central in $\cS\SM$.
}

\def\XE{ \cX_\omega(\Sigma; \cE) }
\def\XEp{ \cX_\omega^+(\Sigma; \cE) }
\def\bm{{\mathbf m }}
The algebra
$$ \cX_\omega(\Sigma; \cE):= \cR[\cV]\la a^{\pm 1} \mid a \in \bE\ra /\la ab = \omega^{2 P(a,b)} ba
\ra $$ is known as the {\em quantum torus associated to $P$}, with ground ring $\cR[\cV]$. \no{ Numerate elements of $\bE$ so that $\bE=\{ a_1,\dots, a_r\}$. For $\bm=(m_1,\dots, m_r)\in \BZ^r$ let 
$$a^\bm :=  \omega^{- \sum_{i<j} P(a_i, a_j)} a_1^{m_1} \dots a_r^{m_r}.$$
The prefactor $\omega^{- \sum_{i<j} P(a_i, a_j)}$ is added so that the result is invariant under the reflection involution. Then the set $\{ a^\bm \mid \bm \in \BZ^r\}$ is an $\cR[\cV]$-basis of $\XE$.} Informally this is the $\cR[\cV]$-algebra of Laurent polynomials in the variables $a\in \bE$ which $q$-commute according to the rule seen in equation \eqref{eq.rel1}. 
 The subalgebra spanned by non-negative power monomials, i.e. the $\cR[\cV]$-subalgebra generated by $a\in \bE$,  is known as the quantum space, and is denoted by $\XEp$.

%The set $ U:=\{ a^\bm  v^\bk  \mid \bm \in \BN^r, \bn \in \BN^k\}$ is an $\cR$-basis of $\XEp$.

The relation \eqref{eq.rel1} shows that there is a algebra homomorphism from $\XEp$ to $\cS\SM$ sending $e$ to $X_e$. Using the height exchange relation and the explicit basis of $\cS\SM$ one can easily show that this algebra homomorphism is an embedding. See \cite{LY,LY2} for details.

\begin{theorem}[\cite{LY}, Theorem $5.1$]
\label{qtrace}
Suppose $\SM$ is a marked surface satisfying ($\star$) and $\cE$ is a $\cP$-quasitriangulation of $\Sigma$. 
The algebra embedding $\XEp \embed \cS\SM$ can be extended to an algebra embedding 
\[\phi_{\cE,\omega}:\cS\SM\hookrightarrow \cX_\omega(\Sigma; \cE)\]
\end{theorem}

When $\Sigma$ has no interior points, the restriction of $\phi_{\cE,\omega}$ to the Muller subalgebra $\cS^+\SM$ was first constructed by Muller \cite{Muller}, and if furthermore $\omega=1$, then this  map expresses the Penner lambda length of loop or a $\cP$-arc as a Laurent polynomial in the lambda length of the edges of the quasitriangulation. Thus $\phi_{\cE,\omega}$ could be called a ``quantum trace map".

The algebra embedding $\phi_{\cE,\omega}:\cS\SM\hookrightarrow \cX_\omega(\Sigma; \cE)$ can be considered as a coordinate map of $\cS\SM$, where the coordinates depend on a $\cP$-quasitriangulation.

\subsection{Restricting the Frobenius homomorphism} 

Given any coordinate system, one can define functions in terms of these coordinates. However, such a function makes sense only when the definition does not depend on the coordinate system itself. We look to see how the choice of a $\cP$-quasitriangulation provides an analogue of coordinates for the stated skein algebra of a marked surface, and how the Chebyshev-Frobenius homomorphism fits into this picture.

The simple nature of the quantum torus $\XE$ allows us to define the following $\cR$-algebra homomorphism, known as the $N$-th Frobenius homomorphism: For every positive integer $N$  let 
\[F_N:\mathcal X_{\omega^{N^2}}(\Sigma;\cE) \rightarrow \XE\]
be the $\cR$-algebra homomorphism  given by 
\begin{align}
F_N(a)& = a ^N \quad   \text{for } \ a \in \bE,\\
F_N(v)& =  T_N(v)  \quad \text{for } \ v \in \cV.
\end{align} 
It is easy to check that $F_N$ respects the defining relations.

Following Theorem \ref{qtrace} we have that for every $\omega$ a non-zero complex number and every pseudo-triangulation $\cE$, we have the following diagram: 
\be
\begin{tikzcd}
\cS_{\omega^{N^2}}\SM\arrow[r,hook]
\arrow[d,dashed,"?"]  
& \cX_{\omega^{N^2}}(\Sigma; \cE)\arrow[d,"F_N"] \\
\cS_{\omega}\SM\arrow[r,hook] & \cX_{\omega}(\Sigma; \cE)
\end{tikzcd}
\ee
and we may naturally ask when there exists a restriction of $F_N$ to the corresponding stated skein algebras, and when is this restriction independent of the underlying pseudo-triangulation.

\def\XE{ \cX_\omega(\Sigma; \cE) }
\begin{theorem}
\label{FrobeniusRestrict}
Suppose $\omega\in \mathbb{C}^{\times}$, $N\geq 2$, and $\eta=\omega^{N^2}$.  Suppose $\SM$ is a marked surface satisfying ($\star$) and has at least two different $\cP$-quasitriangulation.  The $N$-th Frobenius homomorphism $F_N:\mathcal X_{\omega^{N^2}}(\Sigma;\cE) \rightarrow \XE$ restricts to a map $\mathscr{S}_\eta\SM\rightarrow \mathscr{S}_\omega\SM$, and the restriction is independent of the choice of quasitriangulation if and only if $\omega$ is a root of unity such that $\ord(\omega^8)=N$.  Moreover, when $\omega$ is a root of unity with $N=\ord(\omega^8)$, the Chebyshev-Frobenius homomorphism, $\Phi_\omega$, is the unique restriction of $F_N$.

\end{theorem}

\begin{proof}
First define $$ \cX_\omega(\Sigma; \cE^+):= \cR[\cV]\la a^{\pm 1} \mid a \in \cE\ra /ab = \omega^{2 P(a,b)} ba
\ra $$ to be the restriction of $\cX_\omega(\Sigma;\cE)$ to only the positively stated $\cP$-arcs of the quasi-triangulation $\cE$.
Then utilizing the notation of Subsection \ref{positivesubmod}, and defining the restriction of $\phi_{\cE,\omega}$ to the positive submodule to be $\phi_{\cE,\omega}^+$, we observe that by the definition of $\cX_\omega(\Sigma;\cE)$ and $\cX_\omega(\Sigma; \cE^+)$ each face of the following diagram commutes
\[\begin{tikzcd}
 \mathscr{S}_\eta\SM\arrow[rrr,hook,"\phi_{\cE,\eta}"] \arrow[ddd,"\Phi_\omega"'] & & & \cX_{\eta}(\Sigma; \cE) \arrow[ddd,"F_N"]\\
 &\mathscr{S}_\eta^+\SM\arrow[r,hook,"\phi^+_{\cE,\eta}"]\arrow[d,"\Phi_\omega^+"']\arrow[lu,hook] & \cX_\eta(\Sigma; \cE^+)\arrow[d,"F_N"]\arrow[ru,hook] & \\
 &\mathscr{S}_\omega^+\SM\arrow[r,hook,"\phi^+_{\cE,\omega}"']\arrow[ld,hook] & \cX_\omega(\Sigma; \cE^+)\arrow[dr,hook] & \\
 \mathscr{S}_{\omega}\SM\arrow[rrr,hook,"\phi_{\cE,\omega}"'] & & &\cX_{\omega}(\Sigma; \cE) \\
 \end{tikzcd}
\]

From this we look to leverage results of the second author and Paprocki for the positive submodule to see that if $\omega$ is not a root of unity, then the Frobenius map of quantum tori will not restrict to the stated skein algebra of $\SM$.  Namely a rephrasing of Theorem $8.2$ in \cite{LP} gives the following.  As $\omega\in \mathbb{C}^{\times}$, $N\geq 2$, and $\eta=\omega^{N^2}$ and $\SM$ is a marked surface satisfying ($\star$), then if $F_N: \cX_{\omega^{N^2}}(\Sigma; \cE^+)\rightarrow  \cX_\omega(\Sigma; \cE^+)$ restricts to a map $\mathscr{S}^+_{\omega^{N^2}}\SM\rightarrow \mathscr{S}^+_\omega\SM$ for all quasitriangulations $\cE$ and the restriction does not depend on the quasitriangulations, we have that $\omega$ must be a root of unity and $\ord(\omega^8)=N$.

With this in mind it suffices to show that when $\ord(\omega^8)=N$ the outer square of the above diagram communtes, meaning for any quasitriangualtion the restriction to $\mathscr{S}\SM$ is the Chebyshev-Frobenius homomorphism.  As $\mathscr{S}\SM$ contains $\cX_+$, in addition to noting equivalence on the ground ring $\mathcal{R}[v_1,\dots,v_m]$, it will suffice to check for any $p\in \cX^+(\Sigma,\cE)$, that
\[F_N(\phi_{\cE,\eta}(p))=\phi_{\cE,\omega}(\Phi_\omega(p)).\]

We have $3$ cases to consider.

In the first case we have a $\alpha\in \cE\sqcup \hat{\cE}$, such that the endpoints meets two distinct marked points.  Here we have $\alpha=X_\alpha$.  In this case, the framed power and algebra power are equal and
\[F_N(\phi_{\cE,\eta}(X_\alpha))=F_N(X_\alpha)=X_\alpha^N=\phi_{\cE,\omega}(\alpha)^N=\phi_{\cE,\omega}(\alpha)^{(N)}=\phi_{\cE,\omega}(\Phi_\omega(\alpha))=\phi_{\cE,\omega}(\Phi_\omega(X_\alpha)).\] 

In the next case we have a $\beta\in \cE$, such that the endpoints meet a single marked point.  Here we have $X_\beta=\omega\beta$.  In this case repeated applications of height exchange moves implies
\[\beta^{(N)}=\omega^{-N(N-1)}\beta^N\]
and
\begin{align*} 
F_N(\phi_{\cE,\eta}(X_\beta))&=F_N(X_\beta)=X_\beta^N=(\omega\phi_{\cE,\omega}(\beta))^N= \phi_{\cE,\omega}(\omega^{N}\beta^N)= \phi_{\cE,\omega}(\omega^{N^2}\beta^{(N)})\\
&=\phi_{\cE,\omega}(\eta\beta^{(N)}) =\phi_{\cE,\omega}(\eta\Phi_\omega(\beta))=\phi_{\cE,\omega}(\Phi_\omega(\eta\beta))=\phi_{\cE,\omega}(\Phi(X_\beta)).\\
\end{align*}

In the final case we have $\gamma\in \hat{\cE}$, meet at a single marked point.  Here we have $X_\gamma=\omega^{-1}\gamma$.  In this case repeated applications of height exchange moves implies
\[\gamma^{(N)}=\omega^{N(N-1)}\gamma^N\]
and

\begin{align*} 
F_N(\phi_{\cE,\eta}(X_\gamma))&=F_N(X_\gamma)=X_\gamma^N=(\omega^{-1}\phi_{\cE,\omega}(\gamma))^N= \phi_{\cE,\omega}(\omega^{-N}\gamma^N)= \phi_{\cE,\omega}(\omega^{-N^2}\gamma^{(N)})\\
&=\phi_{\cE,\omega}(\eta^{-1}\gamma^{(N)}) =\phi_{\cE,\omega}(\eta^{-1}\Phi_\omega(\gamma))=\phi_{\cE,\omega}(\Phi_\omega(\eta^{-1}\gamma))=\phi_{\cE,\omega}(\Phi(X_\gamma)).\\
\end{align*}

Thus we have that when $\omega$ is a root of unity such that $\ord(\omega^8)=N$ and $\eta=\omega^{N^2}$, then the Chebyshev-Frobenius homomorphism is a restriction of the Frobenius homomorphism of quantum tori as desired.

\end{proof}

\section{A quantum group driven point of view}
\label{QGroupSect}
\label{Appen}
This section is devoted to giving an interpretation of the Chebyshev-Frobenius homomorphism rooted in the theory of quantum groups. 

Recall that the stated skein algebra $\SB$ of the bigon is isomorphic to the co-braided Hopf algebra $\OSL$,  meaning in addition to the definition in Subsection \ref{sec.bigon}, we have the following structure on $\mathcal{O}_{q^2}(SL(2)).$
\begin{definition}
\label{co-Rmat}
$\mathcal{O}_{q^2}(SL(2))$ is co-braided with co-R-matrix
\[\rho:\mathcal{O}_{q^2}(Sl(2))\otimes{O}_{q^2}(Sl(2))\rightarrow \mathcal{R}\]
defined on generators as 
\[\rho\left(\begin{array}{cccc}
a\otimes a & b\otimes b & a\otimes b & b\otimes a  \\
c\otimes c & d\otimes d & c\otimes d & d\otimes c \\
a\otimes c & b\otimes d & a\otimes d & b\otimes c \\
c\otimes a & d\otimes b & c\otimes b & d\otimes a \\
\end{array}\right)=\left(\begin{array}{cccc}
q & 0 & 0 & 0  \\
0 & q & 0 & 0 \\
0 & 0 & q^{-1} & q-q^{-3} \\
0 & 0 & 0 & q^{-1} \\
\end{array}\right)\]
%Where it suffices to give the values on generators as we also have 
%\[\rho(xy\otimes z)=\sum \rho(x^\prime \otimes z^\prime)\rho(y^{\prime\prime}\otimes z^{\prime\prime})\epsilon(x^{\prime\prime})\epsilon(y^\prime)\]
%and
%\[\rho(x\otimes yz)=\sum \rho(x^\prime \otimes z^\prime)\rho(x^{\prime\prime}\otimes y^{\prime\prime})\epsilon(z^{\prime\prime})\epsilon(y^\prime).\]
\end{definition}
Furthermore, the isomorphism mentioned above respects this structure and gives a geometric description of the co-braiding as seen in Figure \ref{fig:CoBraid}.

\begin{figure}[htpb]
    \centering
    \includegraphics[scale=.35]{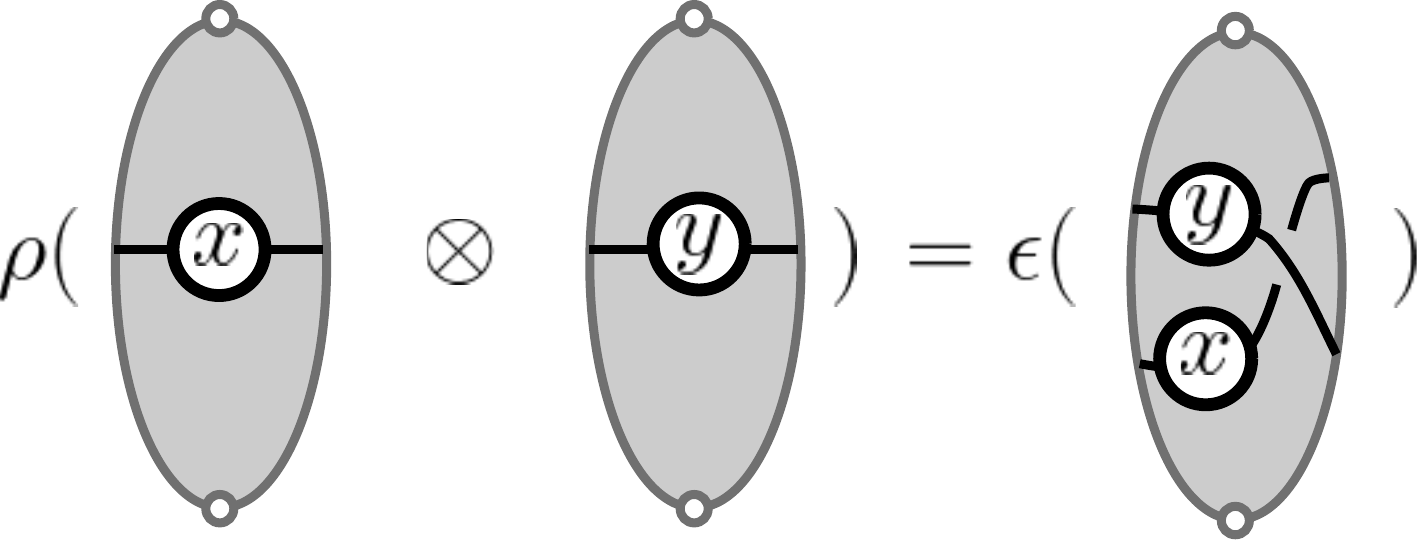}
    \caption{The co-braiding on $\mathscr{S}(\mathcal{B})$}
    \label{fig:CoBraid}
\end{figure}

Recall the quantized enveloping algebra $U_{q^2}(\mathfrak{sl}_2)$ is the Hopf algebra generated over $\mathbb{Q}(q^{1/2})$ by $K^{\pm}, E, F$ with relations
\[KE=q^4EK,\quad KF=q^{-4}FK,\quad [E,F]=\frac{K-K^{-1}}{q^2-q^{-2}}.\]
with coproduct and antipode given by 
\[\Delta(K)=K\otimes K, \Delta(E)=1\otimes E+ E\otimes K, \Delta(F)=K^{-1}\otimes F+ F\otimes 1\]
\[S(K)=K^{-1}, S(E)=-EK^{-1}, S(F)=-KF.\]
Additionally, an integral refinement of the quantized enveloping algebra was introduced by Lusztig, denoted $U^L_{q^2}(\mathfrak{sl}_2)$.  This is the subalgebra of $U_{q^2}(\mathfrak{sl}_2)$ generated by $K^{\pm 1}$ and the divided powers 
\[E^{(r)}:=\frac{E^r}{[r]!},\]
\[F^{(r)}:=\frac{F^r}{[r]!}.\]

Moreover, for $q^{\frac{1}{2}}$ specialized to a root of unity $\omega$, such that $\ord(\omega^8)=N$, we have this subalgebra is also generated by $K^{\pm}, E, F, E^{(N)},F^{(N)}$, as seen in Definition-Proposition $9.3.1$ of \cite{CP}.

There is a Hopf pairing
\[\langle , \rangle_{q^2}:U_{q^2}(\mathfrak{sl}_2)\otimes \mathcal{O}_{q^2}(SL(2))\rightarrow \mathbb{Q}(q^{1/2}),\]
which is a non-degenerate and 
\begin{align*}
&\langle \left( \begin{array}{cc}
a & b\\
c & d
\end{array}\right),K\rangle=\left( \begin{array}{cc}
q^2 & 0\\
0 & q^{-2}
\end{array}\right)\\
&\langle \left( \begin{array}{cc}
a & b\\
c & d
\end{array}\right),E\rangle=\left( \begin{array}{cc}
0 & 1\\
0 & 0
\end{array}\right)\\
&\langle \left( \begin{array}{cc}
a & b\\
c & d
\end{array}\right),F\rangle=\left( \begin{array}{cc}
0 & 0\\
1 & 0
\end{array}\right)
\end{align*}

This non-degenerate Hopf pairing $\langle , \rangle$ restricted to $U^L_{q^2}(\mathfrak{sl}_2)$ is integral, meaning
\[U^L_{q^2}(\mathfrak{sl}_2)\otimes \mathcal{O}_{q^2}(SL(2))\rightarrow \mathbb{Z}[q^{\pm 1/2}],\]

see for example Lemma $6.1$ of \cite{DeConLyu}, and note that the integrality can be deduced from the calculations in Appendix \ref{HopfPairCompute}.

The following propositions describe the quantum Frobenius map defined by Lusztig for quantized enveloping algebras at roots of unity.
\begin{proposition}[Chapter $35$ \cite{Lus2}]
Let $\omega$ be a root of unity, with $N=\ord(\omega^8)$ and $\eta=\omega^{N^2}$.  The quantum Frobenius map is a homomorphism of Hopf algebras 
\[f:U^L_{\omega^4}(\mathfrak{sl}_2)\rightarrow U^L_{\eta^4}(\mathfrak{sl}_2).\]
such that 
\begin{align*}
&f(K)=(-1)^{N+1}K\\
&f(E)=0\\
&f(F)=0\\
&f(E^{(N)})=E\\
&f(F^{(N)})=F
\end{align*}
This map is a surjective homomorphism of quasi-triangular Hopf algebras.
\end{proposition}

\begin{theorem}
\label{uniqueness} 

Let $\Phi_\omega$ be the Chebyshev-Frobenius homomorphism.  We have the Frobenius homomorphism in the sense of Lusztig is equal to the Hopf dual of the Chebyshev-Frobenius homomorphism for the stated skein algebra of the bigon, $\Phi_\omega^*:U_{\omega^4}(\mathfrak{sl}_2)\rightarrow U_{\eta^4}(\mathfrak{sl}_2)$.  
\end{theorem}

\begin{proof}
This follows from a sequence of observations and computations which we will state in full detail and prove with the remainder of this section.  First observe that the composition of the Chebyshev-Frobenius homomorphism for the stated skein algebra of the bigon composed with the isomoprhism to $\mathcal{O}_{q^2}(SL(2))$ is exactly the map defined in Proposition \ref{bigonHopf}.  Then the Hopf dual of this map is seen in Proposition \ref{HopfDual}.
\end{proof}

%\begin{remark}
%We will us Sweedler's notation for the coproduct of Hopf algebras where 
%\[\Delta(x)=\sum x^\prime\otimes x^{\prime\prime}.\]
%\end{remark}

\begin{proposition}
\label{bigonHopf}
Let $\omega$ be a root of unity with $N=\ord(\omega^8)$ and $\eta=\omega^{N^2}$. The map
\[\Phi_\omega:\mathcal{O}_{\eta^4}(SL(2))\rightarrow \mathcal{O}_{\omega^4}(SL(2))\]
defined on the algebra generators of $\mathcal{O}_{\eta^4}(SL(2))$ by
\[a\mapsto a^N, \qquad b\mapsto b^N\]
\[c\mapsto c^N, \qquad d\mapsto d^N\]
is an embedding of co-braided Hopf algebras.
\end{proposition}
\begin{proof}
We see that as a corollary of parts $(a)$ and $(b)$ of Theorem \ref{main} specialized to the stated skein algebra of the bigon, that $\Phi_\omega$ is an algebra and coalgebra morphism respectively.

Injectivity of this map follows from Corollary \ref{r.surfacePhi}.

We also have that $\Phi_\omega(1)=1$ and $\epsilon(\Phi_\omega(\alpha_{\mu,\nu}))=\epsilon(\alpha_{\mu,\nu}^N)=\epsilon(\alpha_{\mu,\nu})^N=\delta_{\mu,\nu}^N=\delta_{\mu,\nu}=\epsilon(\alpha_{\mu,\nu}).$  Then general results in the theory of Hopf algebras tells us that as $\Phi_\omega$ is a unital and counital bialgebra homomorphism it automatically preserves the antipode, but this may be checked directly.  

As a final step we look to show that $\Phi$ respects the co-braiding, meaning the following diagram commutes

\[\begin{tikzcd}
\mathcal{O}_{\eta^4}(SL(2))\otimes \mathcal{O}_{\eta^4}(SL(2)) \arrow[r, "\rho"] \arrow[d,"\Phi\otimes\Phi"']
& \mathbb{Z}[\eta^{\pm 1}] \arrow[d,hook] \\
\mathcal{O}_{\omega^4}(SL(2))\otimes \mathcal{O}_{\omega^4}(SL(2)) \arrow[r, "\rho"]
& \mathbb{Z}[\omega^{\pm 1}]
\end{tikzcd}\]

where the inclusion on the right is given by $\eta=\omega^{N^2}$.  This may be checked directly on generators, but we argue geometrically as seen in Figure \ref{fig:CoBraidMorph}.
\begin{figure}[htpb]
    \centering
    \includegraphics[scale=.35]{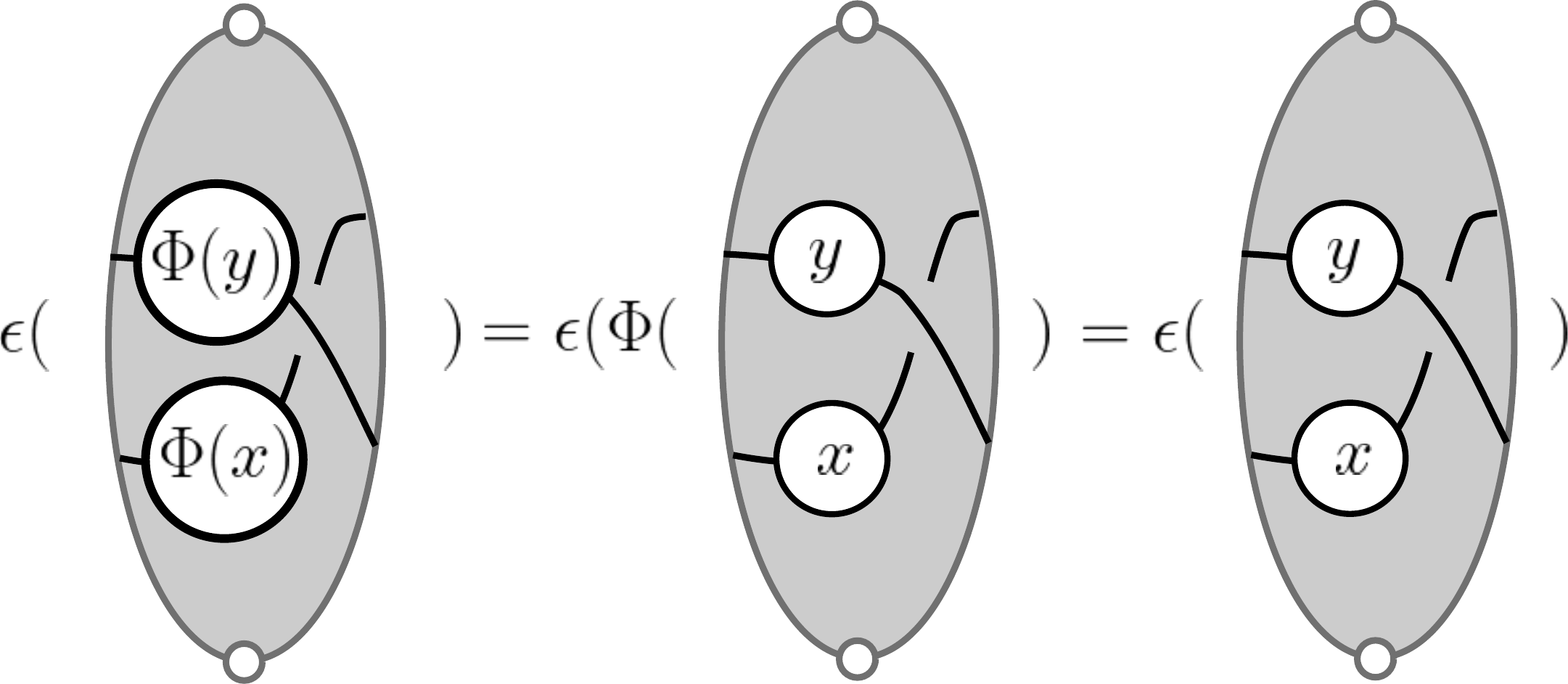}
    \caption{A geometric description of how $\Phi_\omega$ preserves the co-braiding.}
    \label{fig:CoBraidMorph}
\end{figure}

\end{proof} 

\begin{proposition}
\label{HopfDual}
When $\omega$ is a root of unity with $N=\ord(\omega^8)$ and $\eta=\omega^{N^2}$.  We have that the Frobenius map $f$ and the map $\Phi_\omega$ are dual with respect to the Hopf pairings between $\CO_{\eta^4}(SL(2))$ and $U_{\eta^4}(\mathfrak{sl}_2)$ and between  $\CO_{\omega^4}(SL(2))$ and $U_{\omega^4}(\mathfrak{sl}_2)$, meaning the following diagram commutes:

\[\begin{tikzcd}
 \CO_{\eta^4}(SL(2))\otimes U_{\omega^4}(\mathfrak{sl}_2) \arrow[rr,"\Phi\otimes Id"]\arrow[d,"Id\otimes f"]& &  \CO_{\omega^4}(SL(2))\otimes U_{\omega^4}(\mathfrak{sl}_2)\arrow[d,"\langle\text{,}\rangle_{\omega^4}"]\\
  \CO_{\eta^4}(SL(2))\otimes U_{\eta^4}(\mathfrak{sl}_2) \arrow[r,"\langle\text{,} \rangle_{\eta^4}"]&\mathbb{Z}[\eta^{\pm 4}]\arrow[r,hook,"\eta=\omega^{N^2}"] & \mathbb{Z}[\omega^{\pm 4}] 
\end{tikzcd}\]
\end{proposition}
\begin{proof}
This is a direct computation on the generators of the algebras.  We will use matrix notation to indicate the map applied to each entry of the matrix.
\begin{align*}
\langle \left(\begin{array}{cc}
    a & b \\
    c & d
\end{array}\right),& (-1)^{N+1} K\rangle_{\eta^4}\\
&=\left(\begin{array}{cc}
    (-1)^{N+1}\eta^{4} & 0 \\
    0 & (-1)^{N+1}\eta^{-4}
\end{array}\right)=\left(\begin{array}{cc}
    \omega^{-4N^2+4N}\omega^{4N^2} & 0 \\
    0 & \omega^{4N^2-4N}\omega^{-4N^2}
\end{array}\right)\\
&=\left(\begin{array}{cc}
    \omega^{4N} & 0 \\
    0 & \omega^{-4N}
\end{array}\right)=\langle \left(\begin{array}{cc}
    a^N & b^N \\
    c^N & d^N
\end{array}\right), K\rangle_{\omega^4}\\
\langle \left(\begin{array}{cc}
    a & b \\
    c & d
\end{array}\right),  0\rangle_{\eta^4}&=\left(\begin{array}{cc}
    0 & 0 \\
    0 & 0
\end{array}\right)=\langle \left(\begin{array}{cc}
    a^N & b^N \\
    c^N & d^N
\end{array}\right),  E\rangle_{\omega^4}\\
\langle \left(\begin{array}{cc}
    a & b \\
    c & d
\end{array}\right),  0\rangle_{\eta^4}&=\left(\begin{array}{cc}
    0 & 0 \\
    0 & 0
\end{array}\right)=\langle \left(\begin{array}{cc}
    a^N & b^N \\
    c^N & d^N
\end{array}\right),  F\rangle_{\omega^4}\\
\langle \left(\begin{array}{cc}
    a & b \\
    c & d
\end{array}\right),  E\rangle_{\eta^4}&=\left(\begin{array}{cc}
    0 & 1 \\
    0 & 0
\end{array}\right)=\langle \left(\begin{array}{cc}
    a^N & b^N \\
    c^N & d^N
\end{array}\right),  E^{(N)}\rangle_{\omega^4}\\
\langle \left(\begin{array}{cc}
    a & b \\
    c & d
\end{array}\right),  F\rangle_{\eta^4}&=\left(\begin{array}{cc}
    0 & 0 \\
    1 & 0
\end{array}\right)=\langle \left(\begin{array}{cc}
    a^N & b^N \\
    c^N & d^N
\end{array}\right),  F^{(N)}\rangle_{\omega^4}
\end{align*}

Where further details can be found in the Appendix \ref{HopfPairCompute}.

\end{proof}

Now as an application we look to understand how the Chebyshev-Frobenius of the stated skein algebras of surfaces behaves with respect to the gluing of an ideal triangle along two distinct boundary edges.

Let $\fS$ be a punctured bordered surface, which may be disconnected, together with two distinguished boundary edges $e_1$ and $e_2$.  Then define $\underline{\fS}$ to be the result of gluing an ideal triangle along the edges $e_1$ and $e_2$.  The following result is based on the work of the second author and Constantino, seen in \cite{CL}, describing the stated skein algebra of $\underline{\fS}$ in terms of $\fS$. 

The result will require the notion of the self braided tensor product of a comodule algebra.  Let $U$ is a co-quasitriangular Hopf algebra, and $A$ a algebra admitting two different co-actions, $\Delta_1$ and $\Delta_2$, which make $A$ into a comodule algebra over $U$.  Then if the two co-actions commute a new twisted multiplication can be defined on $A$ defined by 
\[x\underline{*} y=\sum x'y'\rho(u\otimes v)\]
where $\Delta_2(x)=\sum x'\otimes u$ and $\Delta_1(y)=\sum y'\otimes v$.  In a sense this twisted algebra, denoted $\underline{\bigotimes} A$, is defined to be a simultaneous $U$ comodule algebra with a coaction defined by $\underline{\Delta}(x)=\sum x'\otimes u_1u_2$, where $\Delta_1(x)=\sum x'\otimes u_1$ and $\Delta_2(x)=\sum x'\otimes u_2$.  
\begin{proposition}
\label{triangGlue}
Assume $\fS$ and $\underline{\fS}$ are as described above.  Then the Chebyshev-Frobenius homomorphism fits into the following commutative diagram
\be
\begin{tikzcd}
\mathscr{S}_{\eta}(\underline{\fS})\arrow[r,"\Phi_\omega"]
\arrow[d,"\cong"']  
&  \mathscr{S}_{\omega}(\underline{\fS})\arrow[d,"\cong"] \\
\underline{\bigotimes}\mathscr{S}_{\eta}(\fS)\arrow[r,"\Phi_\omega"']
&  \underline{\bigotimes}\mathscr{S}_{\omega}(\fS) \\
\end{tikzcd}
\ee
where $\underline{\bigotimes}\mathscr{S}(\fS)$ denotes the self braided tensor product, which has the same underlying generating set as $\cS(\fS)$ but with a twisted multiplication allowing for an analogous definition of $\Phi_\omega$.
\end{proposition}

\begin{proof}
Note that we will use $e_i$ interchangeably to describe a distinguished boundary edge and an ideal arc isotopic to the boundary edge.  Then we see that $\cS(\fS)$ is a comodule algebra algebra over $\mathcal{O}_{q^2}(SL(2))$ over the distinguished edges $e_1$ and $e_2$ in a way that commutes.  Thus we can define the self braided tensor product $\underline{\bigotimes}\cS(\fS)$.   Now let $x,y\in\cS(\fS)$,
\[\Theta_{e_2}(x)=\sum x'\otimes u,\]
and
\[\Theta_{e_1}(y)=\sum y'\otimes v\]
then 
\[x\underline{*} y=\sum x'y'\rho(u\otimes v)\]
where $\rho$ is the co-R-matrix of $\mathcal{O}_{q^2}(SL(2))$. The canonical isomorphisms of the vertical arrows in the commutative diagram are proven in Theorem 4.17 of \cite{CL}.  Then the result follows from noting that $\Phi_\omega$ commutes with both $\rho$ (as seen in Proposition \ref{bigonHopf}) and the splitting homormoprhism (as seen for surfaces in Corollary \ref{r.surfacePhi}).  Then we see the diagram commutes on a generating set, and so commutes in general.
\end{proof}

When specializing the above proposition to the case of $\fS=\fS_1\sqcup \fS_2$, with $e_1\subseteq \fS_1$ and $e_2\subseteq \fS_2$ we have that the self braided tensor product recovers the ordinary braided tensor product, $\cS(\fS_1)\underline{\otimes}\cS(\fS_2)$.  The unfamiliar reader, may take the definition of the braided tensor product to be the self braided tensor product in the case that $A=A_1\otimes A_2$ is the tensor product of two comodule algebras.   In particular, realizing the ideal triangle as the result of gluing two bigons along a triangle we have the following isomorphism (seen in Theorem $4.17$ of \cite{CL}):
\[\cS(\mathfrak{T})\cong \mathcal{O}_{q^2}(SL(2))\underline{\otimes} \mathcal{O}_{q^2}(SL(2)).\]
Now utilizing the compatibility of $\Phi_\omega$ with the splitting homomorphism, we note that $\Phi_\omega$ for a triangulable surface is exactly determined by $\Phi_\omega:\cS_\eta(\mathfrak{T})\rightarrow \cS_\omega(\mathfrak{T})$, which by Proposition \ref{triangGlue} is determined by the Hopf dual of Lustztig's Frobenius homomorphism.  This can be summarized in the following corollary which extends the commutative diagram seen in equation \ref{triangdecomp}.
\begin{corollary}
\label{surfextend}
Let $\fS$ be a triangulable punctured bordered surface.  Then the Chebyshev-Frobenius homomorphism fits into the following commutative diagram for any ideal triangulation $\Delta$:

\[\begin{tikzcd}[column sep=huge]
\cS_\eta(\fS) \arrow[r,"\Phi_\omega"]\arrow[d,hook,"\Theta_\Delta"'] & \cS_\eta(\fS)\arrow[d,hook,"\Theta_\Delta"] \\
\displaystyle{\bigotimes_{\mathfrak{T}\in\Delta} \cS_\eta(\mathfrak{T})}\arrow[d,"\cong"']\arrow[r,"\bigotimes \Phi_\omega"] &\arrow[d,"\cong"] \displaystyle{\bigotimes_{\mathfrak{T}\in\Delta}\cS_\omega(\mathfrak{T})} \\
\bigotimes \mathcal{O}_{\eta^4}(SL(2))\underline{\otimes}\mathcal{O}_{\eta^4}(SL(2))\arrow[r,"\bigotimes (f^*\otimes f^*)"] & \bigotimes\mathcal{O}_{\omega^4}(SL(2))\underline{\otimes}\mathcal{O}_{\omega^4}(SL(2))
\end{tikzcd}\]
where the bottom map is determined by the Hopf dual of Lusztig's Frobenius homomorphism, $f$.
\end{corollary}

\section{Acknowledgments}  The authors would like to thank F. Bonahon, C. Frohman, J. Kania-Bartozynska, A. Kricker, G. Masbaum, G. Muller,  A. Sikora, D. Thurston. The second author  would like to thank the CIMI Excellence Laboratory, Toulouse, France, for inviting him on a Excellence Chair during the period of January -- July 2017 when part of this work was done. The second author is supported in part by NSF grant DMS 1811114.  The first author is supported in part by NSF Grant DMS-17455.
\appendix

\section{Divided powers and the Hopf pairing}
\label{HopfPairCompute}

In this appendix we provide the computations required to show that $\Phi_\omega$ is the hopf dual of the frobenius homomorphism. 

\begin{proposition}
With the notation of Section \ref{QGroupSect}, we have for any $m$ and any $p$.
\[
    \langle \left(\begin{array}{cc}
    a^m & b^m \\
    c^m & d^m
\end{array}\right),  K\rangle_{q^2}=\left(\begin{array}{cc}
    q^{2m} & 0 \\
    0 & q^{-2m}
\end{array}\right)\]
and 
\[
    \langle \left(\begin{array}{cc}
    a^m & b^m \\
    c^m & d^m
\end{array}\right),  E^{(p)}\rangle_{q^2}=\left(\begin{array}{cc}
    \delta_{0,p} & \delta_{m,p} \\
    0 & \delta_{0,p}
\end{array}\right)\]
and 

\[
    \langle \left(\begin{array}{cc}
    a^m & b^m \\
    c^m & d^m
\end{array}\right),  F^{(p)}\rangle_{q^2}=\left(\begin{array}{cc}
    \delta_{0,p} & 0 \\
    \delta_{m,p} & \delta_{0,p} 
\end{array}\right)\]
\end{proposition}
\begin{proof}
Each of the twelve computations will follow from induction.

First we recall 
\[\Delta(K)=K\otimes K,\]

and so inducting on $m$

\begin{align*} 
&\langle a^m,K\rangle=\langle a,K\rangle \langle a^{m-1},K\rangle=q^{2}q^{2m-2}=q^{2m}.\\
&\langle b^m,K\rangle =\langle b,K\rangle \langle b^{m-1},K\rangle=0\\
&\langle c^m,K\rangle =\langle c,K\rangle \langle c^{m-1},K\rangle=0\\
&\langle d^m,K\rangle=\langle d,K\rangle \langle d^{m-1},K\rangle=q^{-2}q^{-2m+2}=q^{-2m}.\\
\end{align*}

Next recall 
\[\Delta(E^{(p)})=\sum_{i=0}^p q^{2i(p-i)} E^{(p-i)}\otimes E^{(i)}K^{p-i}.\]

First we see that if $p=0$ we have 
\[
    \langle \left(\begin{array}{cc}
    a^m & b^m \\
    c^m & d^m
\end{array}\right),  1 \rangle_{q^2}=\left(\begin{array}{cc}
    \epsilon(a^m) & \epsilon(b^m) \\
    \epsilon(c^m) & \epsilon(d^m)
\end{array}\right)=\left(\begin{array}{cc}
    1 & 0 \\
    0 & 1
\end{array}\right)\]

Now we induct on $m$, assuming $p\neq 0$,

\begin{align*}
\langle a^m,E^{(p)}\rangle&=\sum_{i=0}^pq^{2i(p-i)} \langle a^{m-1},E^{(p-i)}\rangle\langle a, E^{(i)}K^{p-i}\rangle\\
&=\sum_{i=0}^p q^{2i(p-i)}\langle a^{m-1},E^{(p-i)}\rangle(\langle a, E^{(i)}\rangle\langle a,K^{p-i}\rangle +\langle b,E^{(i)}\rangle\langle c,K^{p-i}\rangle ) \\
&=\sum_{i=0}^p q^{2i(p-i)}q^{2(p-i)}\langle a^{m-1},E^{(p-i)}\rangle\langle a, E^{(i)}\rangle=0 \\
\langle d^m,E^{(p)}\rangle & =\sum_{i=0}^pq^{2i(p-i)} \langle d^{m-1},E^{(p-i)}\rangle\langle d, E^{(i)}K^{p-i}\rangle\\
&=\sum_{i=0}^p q^{2i(p-i)}\langle d^{m-1},E^{(p-i)}\rangle(\langle d, E^{(i)}\rangle\langle d,K^{p-i}\rangle +\langle c,E^{(i)}\rangle\langle b,K^{p-i}\rangle ) \\
&=\sum_{i=0}^p q^{2i(p-i)}q^{-2(p-i)}\langle d^{m-1},E^{(p-i)}\rangle\langle d, E^{(i)}\rangle=0 \\
\langle c^m,E^{(p)}\rangle&=\sum_{i=0}^pq^{2i(p-i)} \langle c^{m-1},E^{(p-i)}\rangle\langle c, E^{(i)}K^{p-i}\rangle\\
&=\sum_{i=0}^p q^{2i(p-i)}\langle c^{m-1},E^{(p-i)}\rangle(\langle c, E^{(i)}\rangle\langle a,K^{p-i}\rangle +\langle d,E^{(i)}\rangle\langle c,K^{p-i}\rangle ) \\
&=\sum_{i=0}^p q^{2i(p-i)}q^{2(p-i)}\langle c^{m-1},E^{(p-i)}\rangle\langle c, E^{(i)}\rangle=0 \\
\langle b^m,E^{(p)}\rangle&=\sum_{i=0}^pq^{2i(p-i)} \langle b^{m-1},E^{(p-i)}\rangle\langle b, E^{(i)}K^{p-i}\rangle\\
&=\sum_{i=0}^p q^{2i(p-i)}\langle b^{m-1},E^{(p-i)}\rangle(\langle a, E^{(i)}\rangle\langle b,K^{p-i}\rangle +\langle b,E^{(i)}\rangle\langle d,K^{p-i}\rangle ) \\
&=\sum_{i=0}^p q^{2i(p-i)}q^{-2(p-i)}\langle b^{m-1},E^{(p-i)}\rangle\langle b, E^{(i)}\rangle=\sum_{i=0}^p q^{2i(p-i)}q^{-2(p-i)}\delta_{m-1,p-i}\delta_{1,i}\\
&=q^{2(p-1)}q^{-2(p-1)}\delta_{m-1,p-1}=\delta_{m,p}\\
\end{align*}

Next recall 
\[\Delta(F^{(p)})=\sum_{i=0}^p q^{-2i(p-i)} F^{(i)}K^{-(p-i)}\otimes F^{(p-i)}.\]

First we see that if $p=0$ we have 
\[
    \langle \left(\begin{array}{cc}
    a^m & b^m \\
    c^m & d^m
\end{array}\right),  1 \rangle_{q^2}=\left(\begin{array}{cc}
    \epsilon(a^m) & \epsilon(b^m) \\
    \epsilon(c^m) & \epsilon(d^m)
\end{array}\right)=\left(\begin{array}{cc}
    1 & 0 \\
    0 & 1
\end{array}\right)\]

Now we induct on $m$, assuming $p\neq 0$,

\begin{align*} 
\langle a^m,F^{(p)}\rangle&=\sum_{i=0}^pq^{-2i(p-i)} \langle a,F^{(i)}K^{-(p-i)}\rangle\langle a^{m-1}, F^{(p-i)}\rangle\\
&=\sum_{i=0}^p q^{-2i(p-i)}(\langle a, F^{(i)}\rangle\langle a,K^{-(p-i)}\rangle +\langle b,F^{(i)}\rangle\langle c,K^{-(p-i)}\rangle )\langle a^{m-1},F^{(p-i)}\rangle \\
&=\sum_{i=0}^p q^{-2i(p-i)}q^{-2(p-i)}\langle a,F^{(i)}\rangle\langle a^{m-1}, F^{(p-i)}\rangle=0 \\
\langle d^m,F^{(p)}\rangle&=\sum_{i=0}^pq^{-2i(p-i)} \langle d,F^{(i)}K^{-(p-i)}\rangle\langle d^{m-1}, F^{(p-i)}\rangle\\
&=\sum_{i=0}^p q^{-2i(p-i)}(\langle d, F^{(i)}\rangle\langle d,K^{-(p-i)}\rangle +\langle c,F^{(i)}\rangle\langle b,K^{-(p-i)}\rangle )\langle d^{m-1},F^{(p-i)}\rangle \\
&=\sum_{i=0}^p q^{-2i(p-i)}q^{2(p-i)}\langle d,F^{(i)}\rangle\langle d^{m-1}, F^{(p-i)}\rangle=0 \\
\langle b^m,F^{(p)}\rangle&=\sum_{i=0}^pq^{-2i(p-i)} \langle b,F^{(i)}K^{-(p-i)}\rangle\langle b^{m-1}, F^{(p-i)}\rangle\\
&=\sum_{i=0}^p q^{-2i(p-i)}(\langle a, F^{(i)}\rangle\langle b,K^{-(p-i)}\rangle +\langle b,F^{(i)}\rangle\langle d,K^{-(p-i)}\rangle )\langle b^{m-1},F^{(p-i)}\rangle \\
&=\sum_{i=0}^p q^{-2i(p-i)}q^{2(p-i)}\langle b,F^{(i)}\rangle\langle b^{m-1}, F^{(p-i)}\rangle=0 \\
\langle c^m,F^{(p)}\rangle&=\sum_{i=0}^pq^{-2i(p-i)} \langle c,F^{(i)}K^{-(p-i)}\rangle\langle c^{m-1}, F^{(p-i)}\rangle\\
&=\sum_{i=0}^p q^{-2i(p-i)}(\langle c, F^{(i)}\rangle\langle a,K^{-(p-i)}\rangle +\langle d,F^{(i)}\rangle\langle c,K^{-(p-i)}\rangle )\langle c^{m-1},F^{(p-i)}\rangle \\
&=\sum_{i=0}^p q^{-2i(p-i)}q^{2(p-i)}\langle c,F^{(i)}\rangle\langle c^{m-1}, F^{(p-i)}\rangle\sum_{i=0}^p q^{-2i(p-i)}q^{2(p-i)}\delta_{1,i}\delta_{m-1,p-i}\\
&=q^{-2(p-1)}q^{2(p-1)}\delta_{m-1,p-1}=\delta_{m,p}\\
\end{align*}

\end{proof}

%%%%%%%%%%%%%%%%%%%%%%%%%%%%%%%
%%%%%%%%%%%%%%%%%%%%%%%%%%%%%%%

\end{document}